\documentclass[11pt,a4paper]{amsart}
\usepackage{a4wide}

\usepackage[english]{babel}
\usepackage[T1]{fontenc}

\usepackage{amsfonts,amssymb,amsmath}
\usepackage{comment}
\usepackage{tikz-cd}

\usepackage{xcolor}
\usepackage[colorlinks,
    linkcolor={blue!40!black},
    citecolor={blue!40!black},
    urlcolor={blue!70!black}]{hyperref}

\usepackage{url}
\usepackage[all]{xy}

\usepackage{enumerate}
\usepackage{amsfonts,amssymb,amsmath,pinlabel,array,hhline}
\usepackage{slashed}
\usepackage{tabulary}
\usepackage{fancyhdr}
\usepackage{a4wide}
\usepackage{bbm,dsfont}
\usepackage[position=b]{subcaption}
\usepackage{graphicx}

\newtheorem{theorem}{Theorem}[section]

\newtheorem{corollary}[theorem]{Corollary}
\newtheorem{lemma}[theorem]{Lemma}
\newtheorem{proposition}[theorem]{Proposition}

\theoremstyle{definition}
\newtheorem{definition}[theorem]{Definition}

\newtheorem{remark}[theorem]{Remark}
\newtheorem{construction}[theorem]{Construction}

\newtheorem{claim}{Claim}
\newtheorem*{claim*}{Claim}

\newcommand{\Z}{\mathbb{Z}}
\newcommand{\Q}{\mathbb{Q}}
\newcommand{\R}{\mathbb{R}}

\newcommand{\id}{\operatorname{Id}}
\newcommand{\Id}{\operatorname{Id}}
\newcommand{\Sp}{\operatorname{Sp}}
\newcommand{\Hom}{\operatorname{Hom}}
\newcommand{\coker}{\operatorname{coker}}
\newcommand{\Sesq}{\operatorname{Sesq}}
\newcommand{\ev}{\operatorname{ev}}
\newcommand{\pt}{\operatorname{pt}}

\newcommand{\PD}{\operatorname{PD}}
\newcommand{\Bl}{\operatorname{Bl}}
\newcommand{\Aut}{\operatorname{Aut}}

\newcommand{\Iso}{\operatorname{Iso}}

\newcommand{\Homeo}{\operatorname{Homeo}}

\newcommand{\im}{\operatorname{im}}
\newcommand{\rk}{\operatorname{rk}}
\newcommand{\incl}{\operatorname{incl}}
\newcommand{\sm}{\setminus}
\newcommand{\wh}{\widehat}
\newcommand{\wt}{\widetilde}
\newcommand{\ol}{\overline}
\newcommand{\la}{\langle}
\newcommand{\ra}{\rangle}

\newcommand{\La}{\Lambda}

\newcommand{\ks}{\operatorname{ks}}
\newcommand{\eps}{\varepsilon}

\newcommand{\bsm}{\left(\begin{smallmatrix}}
\newcommand{\esm}{\end{smallmatrix}\right)}

\begin{document}
\title{Embedded surfaces with infinite cyclic knot group}

\author{Anthony Conway}
\address{Max-Planck-Institut f\"ur Mathematik, Bonn, Germany}
\email{anthonyyconway@gmail.com}

\author{Mark Powell}
\address{Department of Mathematical Sciences, Durham University, Lower Mountjoy, United Kingdom}
\email{mark.a.powell@durham.ac.uk}

\begin{abstract}
We study locally flat, compact, oriented surfaces in~$4$-manifolds whose exteriors have infinite cyclic fundamental group.
We give algebraic topological criteria for two such surfaces, with the same genus~$g$, to be related by an ambient homeomorphism, and further criteria that imply they are ambiently isotopic.
Along the way, we provide a classification of a subset of the topological~$4$-manifolds with infinite cyclic fundamental group, and we apply our results to rim surgery.
\end{abstract}
\maketitle

\section{Introduction}

We study locally flat embeddings of compact, orientable surfaces in compact, oriented, simply-connected topological~$4$-manifolds, where the complement of the surface has infinite cyclic fundamental group.  Extending the terminology for knotted spheres, we call this group the \emph{knot group}, so we shall study knotted surfaces with knot group~$\Z$, or \emph{$\Z$-surfaces}.

We will present algebraic criteria for pairs of $\Z$-surfaces to be ambiently isotopic.
As part of the proof we obtain an algebraic classification of a certain subset of the 4-manifolds with boundary and fundamental group $\Z$, simultaneously generalising work of Freedman-Quinn~\cite{FreedmanQuinn} on the closed case and of Boyer~\cite{Boyer} on simply-connected 4-manifolds with nonempty boundary; see Section~\ref{sec:class-result-intro}.  We apply our results to show that, in simply-connected $4$-manifolds, 1-twisted rim surgery on a surface with knot group $\Z$ yields a topologically ambiently isotopic surface, extending results of Kim-Ruberman~\cite{KimRuberman-2} and Juh\'asz-Miller-Zemke~\cite{JuhaszMillerZemke}; see Section~\ref{sub:RimIntro}.

\subsection{Surfaces in \texorpdfstring{$S^4$}{the 4-sphere} and \texorpdfstring{$D^4$}{the 4-ball}}

We start with our main results on surfaces in~$S^4$ and~$D^4$ as the most important special cases, before going on to explain more general statements for surfaces in any simply-connected $4$-manifold.

\begin{theorem}
\label{thm:UnknottingIntro}
Any two locally flat, embedded, closed, orientable, genus~$g \neq 1,2$ surfaces $\Sigma_0,\Sigma_1 \subseteq S^4$ with~$\pi_1(S^4 \sm {\Sigma_i})\cong \Z$ for~$i=0,1$ are topologically ambiently  isotopic.
\end{theorem}

\begin{theorem}
\label{thm:D4Intro}
Let~$K$ be an Alexander polynomial one knot in~$S^3$.
Any two locally flat, embedded, compact, orientable genus~$g \neq 1,2$ surfaces~$\Sigma_0,\Sigma_1 \subseteq D^4$ with boundary~$K$ and~$\pi_1(D^4 \sm {\Sigma_i})\cong \Z$ for~$i=0,1$ are topologically ambiently isotopic rel.\ boundary.
\end{theorem}

In a previous paper of ours, we proved Theorem~\ref{thm:D4Intro} in the genus zero case~\cite[Theorem~1.2]{ConwayPowell}.
With regards to Theorem~\ref{thm:UnknottingIntro}, a genus~$g$ surface~$\Sigma \subseteq S^4$ is \emph{unknotted} if it bounds a locally flat, embedded handlebody in~$S^4$.
The \emph{unknotting conjecture} for locally flat surfaces posits that a locally flat, embedded, closed, oriented surface~$\Sigma \subseteq S^4$ is unknotted if and only if~$\pi_1(S^4 \setminus \nu \Sigma)=\Z$. The forwards direction holds because any two embeddings of a handlebody in~$S^4$ are ambiently isotopic, so an unknotted surface is ambiently isotopic to a standard embedding.
For the reverse direction, Freedman and Quinn~\cite[Theorem~11.7A]{FreedmanQuinn} proved the~$g=0$ case of~$2$-knots, when~$\Sigma \cong S^2$, while Hillman and Kawauchi claimed it for~$g \geq 1$~\cite{KawauchiHillman}.
Theorem~\ref{thm:UnknottingIntro} offers a new proof for surfaces of genus~$g \geq 3$.
We also give new proofs for the $g=0$ cases, but it should be noted that this specialisation produces somewhat similar proofs to those of Freedman-Quinn and our previous work \cite{ConwayPowell}.

Our proof of Theorem~\ref{thm:UnknottingIntro} differs from that of Hillman-Kawauchi~\cite{KawauchiHillman}, who do not have any genus restrictions.
 In particular, in a key step in the proof, one considers a closed~$4$-manifold~$M$ built from the two surface exteriors, and shows that~$M$ is homeomorphic to~$S^1 \times S^3 \# _{i=1}^{2g} S^2 \times S^2.$
To see this, we control the~$\Z[\Z]$-valued
intersection form of~$M$, whereas~\cite[Proof of Lemma 3.1]{KawauchiHillman} just calculates
the~$\Z$-valued intersection form, and appeals to separate work of Kawauchi~\cite{KawauchiSplitting}, in
which it was claimed that every closed, spin~$4$-manifold with
fundamental group~$\Z$ splits as~$S^1 \times S^3 \# X$, where~$X$ is a closed, simply-connected 4-manifold.
It would follow from this claim that computing the~$\Z$-valued form suffices.
A mistake in \cite{KawauchiSplitting} was found by Hambleton-Teichner~\cite{HambletonTeichner}. Although Kawauchi later updated his theorem~\cite{KawauchiRevised} to include
the hypothesis that the~$\Z$-intersection form be indefinite, which is the
case for the 4-manifolds arising in~\cite{KawauchiHillman}, the consensus in the community seems to be that an independent account is also desirable.

\subsection{Surfaces in simply-connected~$4$-manifolds with boundary \texorpdfstring{$S^3$}{the 3-sphere}.}

To fix our terminology, throughout the article, by a \emph{$4$-manifold} we shall always mean a compact, connected, oriented, topological~$4$-manifold with connected boundary.
Surfaces will always be compact, connected, and orientable.
We will consider locally flat, embedded, closed, oriented surfaces~$\Sigma \subseteq X$ with~$X$ a closed~$4$-manifold, and locally flat properly embedded surfaces~$\Sigma \subseteq N$, where~$N$ is a~$4$-manifold with boundary~$S^3$ and~$\partial \Sigma=K \subseteq S^3$ is a fixed knot.
The exteriors of such surfaces will be denoted by~$X_\Sigma$ and~$N_\Sigma$ respectively.

Next we formulate our most general statement, for the nonempty boundary case, pairs of surfaces $\Sigma_i \subseteq N$, where $\partial \Sigma_i \subseteq \partial N=S^3$ is a fixed knot.
 Theorem~\ref{thm:D4Intro} is an immediate consequence of Theorem~\ref{thm:WithBoundaryIntro} below, which we state after introducing the prerequisites.
One quickly deduces consequences for closed surfaces in closed~$4$-manifolds~$X$, in particular Theorem~\ref{thm:UnknottingIntro}, by removing an unknotted~$(D^4,D^2)$ pair from~$(X,\Sigma)$, as we  will explain in Section~\ref{subsection:closed-4-manifolds}.

Given a compact, oriented~$4$-manifold~$M$ with~$\pi_1(M)\cong \Z$, we write $\pi_1(M)=\Z$ to mean that we have chosen an isomorphism between $\pi_1(M)$ and $\Z$. In the case of the exterior $X_\Sigma$ or $N_\Sigma$ of an oriented surface $\Sigma$, such an identification is determined by the orientations on $\Sigma$ and on $X$ or $N$ respectively.  Set~$\Lambda:=\Z[t^{\pm 1}]$ and write $H_*(M;\Lambda)$ for the homology of the infinite cyclic cover~$\widetilde{M}$ of~$M$, considered as a~$\Lambda$-module.
Taking signed intersections in~$\widetilde{M}$ endows the homology~$\Lambda$-module~$H_2(M;\Lambda)$ with a sesquilinear Hermitian intersection form
\[\lambda_M \colon  H_2(M;\Lambda) \times H_2(M;\Lambda) \to \Lambda.\]
The adjoint of~$\lambda_M$ is the~$\La$-module homomorphism
\[\wh{\lambda}_{M} \colon H_2(M;\La) \to \Hom_{\La}(H_2(M;\La),\La) =: H_2(M;\La)^*.\]
If two~$4$-manifolds~$M_0,M_1$ with infinite cyclic fundamental group are orientation-preserving homeomorphic, then their~$\Lambda$--intersection forms are isometric. That is, there exists an isomorphism~$F \colon H_2(M_0;\La) \xrightarrow{\cong} H_2(M_1;\La)$ such that~$\wh{\lambda}_{M_0} = F^* \wh{\lambda}_{M_1} F$.  We write~$F \colon \lambda_{M_0} \cong \lambda_{M_1}$  and call~$F$ an \emph{isometry}.

Let~$Y$ be a compact oriented 3-manifold, and let~$\varphi \colon \pi_1(Y) \twoheadrightarrow \Z$ be an epimorphism.
Associated with this data, there is a~$\Lambda$-module~$H_1(Y;\Lambda)$, called the \emph{Alexander module}, which is the first homology group of the infinite cyclic cover associated to~$\ker(\varphi)$.
If this module is torsion over~$\Lambda=\Z[t^{\pm 1}]$, then it is endowed with a sesquilinear Hermitian \emph{Blanchfield form}
$$ \Bl_Y \colon H_1(Y;\Lambda) \times H_1(Y;\Lambda) \to \Q(t)/\Lambda.~$$
More details on this pairing appear in Subsection~\ref{sub:BoundaryIntersectionForm}, but we note that~$\Bl_Y$ should be thought of as the analogue of the linking pairing of a~$\Q$-homology sphere on the level of infinite cyclic covers.

As we recall in Sections~\ref{sec:Union} and~\ref{sec:UnionIsUnion}, if~$M_0,M_1$ are~$4$-manifolds with~$\pi_1(M_i)=\Z$ whose boundaries~$\partial M_i$ have torsion Alexander modules and if~$\pi_1(\partial M_i) \to \pi_1(M_i)$ is surjective for~$i=0,1$, then an isometry~$F \colon \lambda_{M_0} \cong \lambda_{M_1}$ of the~$\Lambda$--intersection forms induces an isometry of the Blanchfield forms of the boundary
\[\partial F \colon (H_1(\partial M_0;\Lambda),\Bl_{\partial M_0}) \xrightarrow{\cong} (H_1(\partial M_1;\Lambda),\Bl_{\partial M_1}).\]
Here is the construction: via universal coefficients, Poincar\'e-Lefschetz duality, and the long exact sequence of a pair, the composition
\[H_2(M_i;\La)^* \cong H^2(M_i;\La) \cong H_2(M_i,\partial M_i;\La) \xrightarrow{\delta} H_1(\partial M_i;\La)\]
induces an identification
$\coker(\wh{\lambda}_{M_i}) \cong H_1(\partial M_i;\La)$,
 and the map
\[F^{-*} := (F^*)^{-1} \colon H_2(M_0;\La)^* \to H_2(M_1;\La)^*\]
induces~$\partial F \colon \coker(\wh{\lambda}_{M_0}) \to \coker(\wh{\lambda}_{M_1})$.

Now we focus on the case where~$M_i=N_{\Sigma_i}$ are surface exteriors, with~$N$ a compact, simply-connected~$4$-manifold with boundary~$S^3$.
The boundary~$\partial N_{\Sigma_i}$ is homeomorphic to \[M_{K,g}:=E_K \cup_\partial (\Sigma_{g,1} \times S^1),\] where~$E_K=S^3 \setminus \nu (K)$ is the knot exterior and~$\Sigma_{g,1}$ denotes the (abstract) genus~$g$ surface with one boundary component.
In Proposition~\ref{prop:AutomorphismInTheBoundaryCase}, we show that every automorphism~$h$ of the Blanchfield pairing~$\Bl_{M_{K,g}}$ decomposes as~$h_K \oplus h_\Sigma$, where~$h_K$ is an automorphism of~$\Bl_K:=\Bl_{E_K}$ and~$h_\Sigma$ is an automorphism of~$\Bl_{\Sigma_{g,1} \times S^1}$.

Let $f_K \colon E_K \to E_K$ be an orientation-preserving homeomorphism that is the identity on~$\partial E_K$. Extend $f_K$ via the identity on $\nu K$ to an orientation-preserving self-homeomorphism of $S^3$.  The mapping class group of $S^3$ is trivial, so there is an isotopy  $\Theta(f_K) \colon S^3 \times [0,1] \to S^3$ between the extension and the identity, such that $\Theta(f_K)|_{S^3 \times \{0\}} = \Id$ and $\Theta(f_K)|_{E_K \times \{1\}} = f_K$.

Here is the central theorem of the paper on embedded surfaces with knot group~$\Z$.


\begin{theorem}
\label{thm:WithBoundaryIntro}
Let~$X$ be a closed, simply-connected, oriented~$4$-manifold, let~$N=X \setminus \mathring{D}^4$ be a punctured~$X$, and let~$K \subseteq S^3=\partial N$ be a knot.
Let~$\Sigma_0,\Sigma_1 \subseteq N$ be two locally flat, embedded, compact, oriented genus~$g$ surfaces with the same oriented boundary~$K$ and~$\pi_1(N_{\Sigma_i})=\Z$ for~$i=0,1$.
Suppose there is an isometry~$F \colon \lambda_{N_{\Sigma_0}} \cong \lambda_{N_{\Sigma_1}}$ and write~$\partial F=h_K \oplus h_\Sigma$.
\begin{itemize}
\item  If~$h_K$ is induced by an orientation-preserving homeomorphism~$f_K \colon E_K \to E_K$ that is the identity on~$\partial E_K$, then~$f_K$ extends to an orientation-preserving homeomorphism of pairs
$$(N,\Sigma_0) \xrightarrow{\cong} (N,\Sigma_1)$$
inducing the given isometry~$F \colon H_2(N_{\Sigma_0};\Lambda) \cong H_2(N_{\Sigma_1};\Lambda)$.
\item
If in addition~$N=D^4$, then for any choice of isotopy $\Theta(f_K)$, the surfaces~$\Sigma_0$ and~$\Sigma_1$ are topologically ambiently isotopic via an ambient isotopy of $D^4$ extending $\Theta(f_K)$.
\end{itemize}
\end{theorem}



In particular note that if $h_K =\id$, then we can take $f_K=\id$ and $\Theta_{f_K}$ the constant isotopy, so that the homeomorphism of pairs in the first item can be assumed to fix the boundary pointwise, and the ambient isotopy in the second item can be assumed to be rel.\ boundary.

In general,~$\lambda_{N_{\Sigma_0}}$ and~$\lambda_{N_{\Sigma_1}}$ need not be isometric, even for~$N = D^4$, as shown by examples due to Oba~\cite{Oba}.

The deduction of the last item uses Alexander's coning trick, which shows that every homeomorphism of $D^4$ that restricts to the identity on $\partial D^4$ is topologically isotopic to the identity. So in $D^4$, a  homeomorphism of pairs can be upgraded to a topological ambient isotopy. See Section~\ref{sub:SurfacesManifoldWithBoundary} for details.

As explained in Section~\ref{sub:RimIntro} below, we will apply Theorem~\ref{thm:WithBoundaryIntro} to 1-twisted rim surgery, a method which has been effective at producing exotic embeddings of surfaces.
In Section~\ref{sub:PushedInIntro}, we also apply Theorem~\ref{thm:WithBoundaryIntro} to study Seifert surfaces that are pushed in to $D^4$.
Further applications can be obtained by finding classes of knots $K$ for which every automorphism of the Blanchfield pairing is realised by a symmetry of the knot exterior $E_K$.

\subsection{Ambient isotopy in closed~$4$-manifolds.}\label{subsection:closed-4-manifolds}

For closed surfaces embedded in closed, simply-connected $4$-manifolds, by applying the classification of self-homeomorphisms due to Kreck, Perron, and Quinn~\cite{KreckIsotopy,Perron-Isotopy,QuinnIsotopy}, we can potentially upgrade a homeomorphism of pairs obtained from  Theorem~\ref{thm:WithBoundaryIntro} to an ambient isotopy.  The theorem (see \cite{QuinnIsotopy} or \cite[Theorem~10.1]{FreedmanQuinn}) is that two self-homeomorphisms of a closed, simply-connected~$4$-manifold are isotopic if and only if they induce the same self-isomorphism on second homology. An analogous classification of homeomorphisms for simply-connected~$4$-manifolds with boundary has not yet been proven.

\begin{theorem}
\label{thm:Unknotting4ManifoldIntro}
Let~$X$ be a closed, simply-connected, oriented~$4$-manifold.
Let~$\Sigma_0, \Sigma_1 \subseteq X$ be two locally flat, embedded, closed, oriented genus~$g$ surfaces with~$\pi_1(X_{\Sigma_i})=\Z$ for~$i=0,1$.
\begin{enumerate}
\item\label{item:closed-case-1}
If the intersection forms~$\lambda_{X_{\Sigma_0}}$ and~$\lambda_{X_{\Sigma_1}}$ are isometric via an isometry~$F$, then there is an orientation-preserving homeomorphism of pairs
$$\Phi \colon (X,\Sigma_0) \xrightarrow{\cong} (X,\Sigma_1)$$
inducing the given isometry~$\Phi_* = F \colon H_2(X_{\Sigma_0};\Lambda) \cong H_2(X_{\Sigma_1};\Lambda)$.
\item\label{item:closed-case-2}
The isometry~$F$ induces an isometry~$F_\Z \colon H_2(X) \to H_2(X)$ of the standard intersection form~$Q_X$ of~$X$.
The surfaces $\Sigma_0$ and~$\Sigma_1$ are topologically ambiently isotopic if and only if $F_\Z=\id$.
 \end{enumerate}
\end{theorem}

With regards to \eqref{item:closed-case-2}, we refer to Lemma~\ref{lem:InducedIsometryOfH2(X)} for details of how the isometry~$F \colon \lambda_{\Sigma_0} \cong \lambda_{\Sigma_1}$ induces an isometry~$F_\Z$ of the standard intersection form~$Q_X$.
Theorem~\ref{thm:Unknotting4ManifoldIntro} also has applications to rim surgery, as we explain in Section~\ref{sub:RimIntro}.


Theorem~\ref{thm:Unknotting4ManifoldIntro}~\eqref{item:closed-case-1} follows from Theorem~\ref{thm:WithBoundaryIntro}. Here is a short outline; see Section~\ref{subsection:surfaces-in-closed-4-manifolds} for details.
After an isotopy, we assume $\Sigma_0$ and $\Sigma_1$ coincide on a disc~$D^2$. Remove an open neighbourhood of this common disc $(\mathring{D}^4,\mathring{D}^2)$ from~$(X,\Sigma_i)$ to obtain a pair $(N,\widetilde{\Sigma}_i)$, with $\partial N \cong S^3$ and $\partial \wt{\Sigma}_i$ an unknot~$K$.
The exterior of~$\Sigma_i$ in~$X$ equals the exterior of~$\widetilde{\Sigma}_i$ in~$N$, so the assumptions of Theorem~\ref{thm:WithBoundaryIntro} hold. Here, since the unknot has trivial Alexander module, $h_K = \id$ and so we take $f_K = \id$. Then Theorem~\ref{thm:WithBoundaryIntro} produces a homeomorphism of pairs~$(N,\widetilde{\Sigma}_0) \cong (N,\widetilde{\Sigma}_1)$ rel.\ boundary, which we complete with the identity on the missing 4-ball to prove~\eqref{item:closed-case-1}.  Deducing \eqref{item:closed-case-2} uses the classification of homeomorphisms from~\cite{QuinnIsotopy}, \cite[Theorem~10.1]{FreedmanQuinn} mentioned above.

\begin{remark}
\label{rem:Sunukjian}
It is worth contrasting Theorem~\ref{thm:Unknotting4ManifoldIntro} with a result of Sunukjian.
Indeed,~\cite[Theorem 7.2]{Sunukjianpi1Z} shows that closed surfaces $\Sigma_0,\Sigma_1 \subseteq X$ of the same genus with knot group $\Z$ are topologically isotopic provided $b_2(X) \geq |\sigma(X)|+2$.
Thus, when $X$ is ``big enough'', restrictions on the equivariant intersection form are not needed to establish isotopy.
\end{remark}




\subsection{Deducing Theorems~\ref{thm:UnknottingIntro} and~\ref{thm:D4Intro} from Theorems~\ref{thm:WithBoundaryIntro} and~\ref{thm:Unknotting4ManifoldIntro}}
\label{sub:DeducingIntro}

In the closed case, with $X=S^4$ if we have an isometry $F$ of the intersection form of $X_{\Sigma_0} = S^4_{\Sigma_0}$ and $X_{\Sigma_1} = S^4_{\Sigma_1}$ then the map $F_{\Z}\colon H_2(S^4) \to H_2(S^4)$ is automatically the identity automorphism. 

For the case with nonempty boundary, with $N=D^4$, Alexander polynomial one implies that $H_1(E_K;\Lambda)=0$, so $h_K=\Id$ automatically, and we may take $f_K \colon E_K \to E_K$ also to be the identity.

Therefore, in both cases, it suffices to find an isometry between the intersection forms of the surface exteriors.  It is an open question whether this is true in general.  For genus zero, i.e.\ for discs and spheres, this is automatically the case since the $\Lambda$-coefficient second homology of the surface exterior vanishes.  Therefore the genus zero cases of Theorems~\ref{thm:UnknottingIntro} and \ref{thm:D4Intro}, due to \cite[Theorem~11.7A]{FreedmanQuinn} and \cite[Theorem~1.2]{ConwayPowell} respectively, follow from Theorems~\ref{thm:WithBoundaryIntro} and~\ref{thm:Unknotting4ManifoldIntro}.

For genus at least three, our strategy to show that the intersection forms are isometric is as follows.
As explained above, the exterior of a closed surface in $S^4$ can be considered as the exterior of a properly embedded surface in $D^4$, so we discuss only the latter case.

By \cite[Theorem~5]{BaykurSunukjian} (see also Theorem~\ref{thm:BS}), $\Sigma_0$ and $\Sigma_1$ are stably equivalent, meaning that they become isotopic after adding some number of trivial tubes.  Therefore the intersection forms of~$\lambda_{D^4_{\Sigma_0}}$ and~$\lambda_{D^4_{\Sigma_1}}$ satisfy~$\lambda_{D^4_{\Sigma_0}} \oplus \mathcal{H}_2^{\oplus n}\cong \lambda_{D^4_{\Sigma_1}} \oplus \mathcal{H}_2^{\oplus n}$ for some~$n \geq 0$, where
\[\mathcal{H}_2 := \begin{pmatrix} 0&t-1 \\ t^{-1}-1 &0 \end{pmatrix}.\]
  Then provided $g$ is at least three we are able to leverage algebraic cancellation results for hyperbolic forms from \cite{Bass} (see also~\cite{HambletonTeichner,Khan,Magurn-vdKallen-Vaserstein,CrowleySixt}) to improve such a stable isometry to an isometry.
While we refer to Subsection~\ref{sub:intersection-forms-in-D4} for details, we record one of the aforementioned intermediate results as it might be of independent interest.

\begin{proposition}
\label{prop:StabiliseToUnknotIntro}
Let~$X$ be a closed, simply-connected, oriented~$4$-manifold.
Let~$\Sigma_0,\Sigma_1 \subseteq N=X\setminus \mathring{D}^4$ be two locally flat, properly embedded, compact, oriented genus~$g$ surfaces with boundary the same oriented knot $K \subseteq S^3$ and~$\pi_1(N_{\Sigma_i})=\Z$ for~$i=0,1$.
There exists an integer~$n \geq 0$ and an isometry
\[\lambda_{N_{\Sigma_0}} \oplus \mathcal{H}_2^{\oplus n} \cong \lambda_{N_{\Sigma_1}}\oplus \mathcal{H}_2^{\oplus n}. \]
\end{proposition}

Whereas Theorems~\ref{thm:UnknottingIntro} and~\ref{thm:D4Intro} hold for $g \neq 1,2$, Proposition~\ref{prop:StabiliseToUnknotIntro} leads to some results for arbitrary $g$.
Indeed, Corollary~\ref{cor:IntersectionFormAlexanderPolynomial1} shows that for any genus $g$ surface $\Sigma \subseteq N=X \setminus \mathring{D^4}$ with knot group~$\Z$ and $\partial \Sigma$ an Alexander polynomial one knot, $\lambda_{N_\Sigma} \oplus \mathcal{H}_2^{\oplus n}\cong Q_X \oplus \mathcal{H}_2^{\oplus (g+n)}$ for some~$n \geq 0$, where $Q_X$ denotes the standard intersection form of $X$.
The same result holds for closed surfaces~$\Sigma \subseteq X$ with knot group~$\Z$.



\subsection{Application to rim surgery}
\label{sub:RimIntro}

Rim surgery is an effective way to produce smoothly inequivalent embeddings of surfaces~\cite{FintushelStern,KimHJ,KimRuberman-2,Mark-Thomas,JuhaszMillerZemke}.
Given a locally flat embedded closed oriented surface~$\Sigma$ in a closed~$4$-manifold~$X$, a knot~$J \subseteq~S^3$ and simple closed curve~$\alpha \subseteq \Sigma$, \emph{$n$-roll~$m$-twist rim surgery} outputs another locally flat embedded closed oriented surface~$\Sigma_{n}^m(\alpha,J) \subseteq X$.
The case where~$m=0,n=0$ was first introduced by Fintushel-Stern~\cite{FintushelStern}, while~$m$-twisted rim surgery was introduced by Kim~\cite{KimHJ}, and~$n$-roll~$m$-twist rim surgery first appeared in~\cite{KimRuberman}.

Kim and Ruberman showed that if~$X$ is a simply-connected~$4$-manifold,~$\pi_1(X_\Sigma)=\Z_d$ is a finite cyclic group and~$(m,d)=1$, then~$\Sigma_n^m(\alpha,J)$ and~$\Sigma$ are topologically isotopic~\cite[Theorem 1.3 and Proposition 3.1]{KimRuberman}.
In order to extend this to infinite cyclic fundamental groups, we consider~$1$-twist rim surgery,  and set~$\Sigma(\alpha,J):=\Sigma_{0}^1(\alpha,J)$. 
In this case if~$\pi_1(X_\Sigma)=\Z$, then also~$\pi_1(X_{\Sigma(\alpha,J)})=\Z$.
Our first result on rim surgery then reads as follows.

\begin{theorem}
\label{thm:RimSurgeryClosedIntro}
Let~$X$ be a closed, simply-connected~$4$-manifold, let~$\Sigma \subseteq X$ be a locally flat, embedded, orientable surface with knot group~$\Z$, let~$\alpha \subseteq \Sigma$ be a simple closed curve and let~$J$ be a knot.
Then the surfaces~$\Sigma$ and~$\Sigma(\alpha,J)$ are topologically ambiently isotopic.
\end{theorem}

For good fundamental groups $\pi_1(X \sm \Sigma)$ and $J$ a slice knot, Kim and Ruberman showed in~\cite[Theorem~4.5]{KimRuberman-2} that there is a homeomorphism of pairs $(X,\Sigma) \cong (X,\Sigma(\alpha,J))$, so for $\pi_1(X \sm \Sigma) \cong \Z$ our result is an extension.

Rim surgery is also defined for a properly embedded surface in a~$4$-manifold~$N$ with boundary.
Juh\'asz, Miller, and Zemke showed that if~$\Sigma \subseteq N$ is a properly embedded locally flat oriented surface, and if~$\alpha$ bounds a locally flat disc in~$N_{\Sigma}$, then~$\Sigma(\alpha,J)$ is topologically isotopic to~$\Sigma$~\cite[Corollary 2.7]{JuhaszMillerZemke}.
In the case where~$\Sigma$ has knot group~$\Z$, we generalise this result as follows.

\begin{theorem}
\label{thm:RimSurgeryIntro}
Let~$X$ be a closed, simply-connected~$4$-manifold, and let~$N:=X \setminus \mathring{D}^4$ be a punctured~$X$.
Let~$\Sigma \subseteq N$ be a locally flat, properly embedded, orientable surface with knot group~$\Z$, let~$\alpha \subseteq \Sigma$ be a simple closed curve and let~$J$ be a knot.
There is a rel.\ boundary orientation-preserving homeomorphism of pairs
$$ (N,\Sigma) \xrightarrow{\cong}(N,\Sigma(\alpha,J)).$$
If~$N=D^4$ the surfaces~$\Sigma$ and~$\Sigma(\alpha,J)$ are topologically ambiently isotopic rel.\ boundary.
\end{theorem}

Note that in both of the previous two results, we have no restrictions on the genera of the surfaces.
This is because we show that the~$\Lambda$-intersection forms of~$N_{\Sigma(\alpha,J)}$ and~$N_\Sigma$ coincide, allowing us to apply the first item of Theorem~\ref{thm:WithBoundaryIntro}, for which there is no genus restriction.

For the case~$N=D^4$, Theorem~\ref{thm:RimSurgeryIntro} means that the construction of exotic surfaces in~\cite{JuhaszMillerZemke} applies more generally than realised in that article, because the condition that~$\alpha$ bounds a locally flat embedded disc in~$N_{\Sigma}$ is not needed.

Finally, we refer the interested reader to subsequent work of Juh\'{a}sz and the second named author for a related result involving $n$-roll~$m$-twist surgery on (non-rim) tori in $S^4$~\cite[Theorem~1.3]{JuhaszPowell}.

\subsection{Application to pushed-in Seifert surfaces.}
\label{sub:PushedInIntro}

Pushing a Seifert surface for a knot $K$ into $D^4$ yields a surface with knot group $\Z$.
It is then intriguing to wonder whether any two Seifert surfaces of the same genus for a knot $K$ become isotopic once pushed into $D^4$; see for instance~\cite[Problem 1.20 (C)]{KirbyProblemList} and~\cite[Section~6]{LivingstonSurfaces}.
In Theorem~\ref{thm:PushedIn}, relying on Theorem~\ref{thm:WithBoundaryIntro}, we show that this is the case for Alexander polynomial one knots.
\begin{theorem}
\label{thm:PushedInIntro}
If~$F_0,F_1 \subseteq D^4$ are genus~$g$ pushed-in Seifert surfaces for an Alexander polynomial one knot~$K$, then they are topologically ambiently isotopic rel.\ boundary.
\end{theorem}
This result provides another setting where we are able to obtain results for~$g=1,2$.
In particular, it rules out the most naive potential counter examples to the conjecture that Theorem~\ref{thm:D4Intro} holds for $\Z$-surfaces of any genus.

\subsection{A classification result for topological 4-manifolds}\label{sec:class-result-intro}

Several of the steps from the proof of Theorem~\ref{thm:WithBoundaryIntro} fit in a classification scheme that applies to a wider class of~$4$-manifolds than surface exteriors.
We state the resulting theorem, which is analogous to Boyer's classification for the simply-connected case~\cite{Boyer}.
\medbreak

We say that a compact, oriented, topological~$4$-manifold~$M$ with nonempty connected boundary has \emph{ribbon boundary} if the inclusion induced map~$\iota \colon \pi_1(\partial M)\to \pi_1(M)$ is surjective.
For example, in simply-connected 4-manifolds, (connected) surface exteriors with infinite cyclic fundamental group have ribbon boundaries.
Set~$\Lambda:=\Z[t^{\pm 1}]$.
If we have an identification~$\pi_1(M)=\Z$, then we have infinite cyclic covers~$\widetilde{M} \to M$ and~$\partial \widetilde{M} \to \partial M$ and~$\Lambda$-modules
$H_*(M;\Lambda)=H_*(\widetilde{M})$  and~$H_*(\partial M;\Lambda)=H_*(\partial \widetilde{M}).$

Let~$M_0,M_1$ be two~$4$-manifolds with infinite cyclic fundamental group whose boundaries are ribbon and have~$\Lambda$-torsion Alexander modules.
Fix an orientation-preserving homeomorphism~$f \colon \partial M_0 \to \partial M_1$ that intertwines the inclusion induced maps~$\varphi_i \colon \pi_1(\partial M_i) \twoheadrightarrow  \Z$, and an isometry~$F \colon (H_2(M_0;\Lambda),\lambda_{M_0}) \to (H_2(M_1;\Lambda),\lambda_{M_1})$ of the~$\Lambda$-valued intersection forms~$\lambda_{M_0}$ and~$\lambda_{M_1}$.
As described in Subsection~\ref{sub:Compatible} below,~$f$ and~$F$ induce isometries of the boundary Blanchfield forms:
$$ f_*,\partial F \colon (H_1(\partial M_0;\Lambda),\Bl_{\partial M_0}) \to (H_1(\partial M_1;\Lambda),\Bl_{\partial M_1}).$$
We call~$(f,F)$ a \emph{compatible pair} if~$f$ and~$F$ induce the same isometry.
For conciseness, we write
$$ \Homeo_\varphi^+(\partial M_0,\partial M_1) := \lbrace f \in  \Homeo^+(\partial M_0,\partial M_1) \mid \varphi_1 \circ f_*=\varphi_0  \rbrace~$$
for the set of orientation-preserving homeomorphisms of the boundaries that intertwine the maps to~$\Z$.

Recall that the Kirby-Siebenmann invariant~$\ks(M_i) \in \Z/2$ is the unique obstruction for the stable tangent bundle~$M_i \to \operatorname{BTOP}$ to lift to~$M_i \to \operatorname{BPL}$, or equivalently for~$M_i$ to be smoothable after adding copies of~$S^2 \times S^2$~\cite[Theorem~8.6]{FreedmanQuinn}, \cite[Section~8]{FriedlNagelOrsonPowell}.
This is relevant in the next theorem for the non-spin case; note that for~$4$-manifolds with fundamental group~$\Z$, whether or not they are spin is determined by the intersection pairing.

\begin{theorem}
\label{thm:Boyer}
Let~$M_0$ and~$M_1$ be two compact, oriented~$4$-manifolds, with $\pi_1(M_i)=\Z$ for $i=0,1$, whose boundaries are ribbon and have~$\Lambda$-torsion Alexander modules.
Let~$f \in \Homeo_{\varphi}^+(\partial M_0,\partial M_1)$ be a homeomorphism
and let~$F  \colon \lambda_{M_0} \cong \lambda_{M_1}$ be an isometry.  If~$M_0$ and~$M_1$ are not spin, then assume that the Kirby-Siebenmann invariants satisfy~$\ks(M_0)=\ks(M_1) \in \Z/2$. Then the following assertions are equivalent:
\begin{enumerate}
\item the pair~$(f,F)$ is compatible;
\item the homeomorphism~$f$ extends to an orientation-preserving homeomorphism
$$ \Phi \colon M_0 \xrightarrow{\cong} M_1$$
inducing the given isometry~$F \colon H_2(M_0;\Lambda) \cong H_2(M_1;\Lambda)$.
\end{enumerate}
\end{theorem}

\begin{remark}
\label{rem:BoyerFreedmanQuinn}
Theorem~\ref{thm:Boyer} should be compared with both Boyer's classification of simply-connected 4-manifolds with boundary~\cite{Boyer} and with Freedman-Quinn's classification of closed~$4$-manifolds with infinite cyclic fundamental group~\cite{FreedmanQuinn}.
Boyer uses a notion similar to our compatible pairs that he calls \emph{morphisms}.
In a nutshell, given simply-connected~$4$-manifolds~$M_0,M_1$ with rational homology spheres as their boundaries, Boyer shows that a morphism~$(f,F)$ can be extended to a homeomorphism~$M_0 \to M_1$~\cite[Theorem~0.7 and Proposition 0.8]{Boyer}.
Note that Boyer's methods differ from ours, as he does not use the union of forms.
On the other hand, Freedman-Quinn show that closed~$4$-manifolds with infinite cyclic fundamental groups are classified by their~$\Lambda$--intersection form and their Kirby-Siebenmann invariant~\cite[Theorem~10.7A]{FreedmanQuinn},~\cite{Stong-Wang-2000},~\cite{HambletonKreckTeichner}.
\end{remark}

A major step in the proof of Theorem~\ref{thm:Boyer} is the following intermediate result that might be of independent interest;
a more detailed statement and a proof can be found in Theorem~\ref{thm:UnionIsS1S3ConnectSumSimplyConnected}.

\begin{theorem}
\label{thm:UnionIsS1S3ConnectSumSimplyConnectedIntro}
Let~$M_0$ and~$M_1$ be two compact, oriented~$4$-manifolds, with $\pi_1(M_i)=\Z$ for $i=0,1$, whose boundaries are ribbon and have~$\Lambda$-torsion Alexander modules.
Let~$(f,F)$ be a compatible pair. If~$M_0$ and~$M_1$ are not spin, then assume that the Kirby-Siebenmann invariants satisfy~$\ks(M_0)=\ks(M_1) \in \Z/2$.
Then there is an orientation-preserving homeomorphism
$$M_0 \cup_f -M_1 \cong S^1 \times S^3 \#_{i=1}^{a} S^2 \times S^2 \#_{j=1}^b S^2 \wt{\times} S^2$$
for some~$a,b$ with~$a+b=b_2(M_0)$. If~$M_0$ and~$M_1$ are spin then~$b=0$.
\end{theorem}

\begin{remark}
\label{rem:HambletonTeichner}
We compare Theorem~\ref{thm:UnionIsS1S3ConnectSumSimplyConnectedIntro} with a particular case of a result due to Hambleton-Teichner~\cite{HambletonTeichner}.
Hambleton and Teichner show that if~$M$ is a closed, oriented, topological 4-manifold with infinite cyclic fundamental group and with~$b_2(M)-|\sigma(M)| \geq 6$, then~$M$ is homeomorphic to the connected sum of~$S^1 \times S^3$ with a unique closed, simply-connected 4-manifold~\cite[Corollary~3]{HambletonTeichner}.
In most cases, this result is stronger than Theorem~\ref{thm:UnionIsS1S3ConnectSumSimplyConnectedIntro}.
However, there are some situations (such as the union of two genus one surfaces exteriors~$X_{\Sigma_0},X_{\Sigma_1} \subseteq S^4$ with~$\pi_1(X_{\Sigma_i})=\Z$ for~$i=0,1$) where Theorem~\ref{thm:UnionIsS1S3ConnectSumSimplyConnectedIntro} applies while~\cite[Corollary~3]{HambletonTeichner} does not.
Finally, note that Sunukjian has used these results to study closed surfaces~$\Sigma \subseteq X$ with~$\pi_1(X \setminus \nu \Sigma)=\Z$ in closed 4-manifolds that satisfy~$b_2(X) \geq |\sigma(X)|+6$~\cite[Theorem~7.2]{Sunukjianpi1Z}.
\end{remark}

\subsection{Isometries of the Blanchfield form}
\label{sub:Programme}
Theorem~\ref{thm:Boyer} provides a criterion for extending a homeomorphism~$f \colon \partial M_0 \to \partial M_1$ to a homeomorphism~$M_0 \to M_1$: one must fit~$f$ into a compatible pair.
However, in practice, finding compatible pairs is difficult and so we provide some sufficient conditions for their existence.
\medbreak
Recall that if~$V$ is a~$4$-manifold with~$\pi_1(V)=\Z$ whose boundary~$Y=\partial V$ is ribbon and has torsion Alexander module, then an isometry~$F$ of the~$\Lambda$--intersection form~$\lambda_V$ induces an isometry
$$\partial F \colon (H_1(Y;\Lambda),\Bl_Y) \xrightarrow{\cong} (H_1(Y;\Lambda),\Bl_Y).$$
We write~$\Aut(\lambda_V)$ and~$\Aut(\Bl_Y)$ for the groups of isometries of~$\lambda_V$ and~$\Bl_Y$, and note that  there is a left action of~$\Aut(\lambda_V)$ on~$\Aut(\Bl_Y)$ given by~$F \cdot h=h \circ \partial F^{-1}$.
We will be interested in a quotient of the orbit set
$$ \Aut(\Bl_Y)/\Aut(\lambda_V).$$
In order to find compatible pairs, we need to know when an isometry of the Blanchfield form is induced by a homeomorphism.
As we have already alluded to, we note in Proposition~\ref{prop:HomeoInduceIsometry} that any homeomorphism~$f \in \Homeo_{\varphi}^+(Y)$ induces a~$\Lambda$-isometry~$f_*$ of~$\Bl_Y$ by lifting~$f$ to the infinite cyclic covers.
Here, $\Homeo^+_{\varphi}(Y) := \Homeo^+_{\varphi}(Y,Y)$ denotes the set of orientation-preserving self-homeomorphisms of~$Y$ that intertwine the map $\pi_1(Y) \to \Z$.
The assignment~$(f,F) \cdot h= f_* \circ h \circ \partial F^{-1}$ then gives rise to a left action
$$\Homeo^+_{\varphi}(Y) \times \Aut(\lambda_V) \curvearrowright \Aut(\Bl_Y).$$
We return to our~$4$-manifolds~$M_0,M_1$ whose boundaries are ribbon and have torsion Alexander module.
If~$M_0$ and~$M_1$ are orientation-preserving homeomorphic, then a compatible pair exists; recall Theorem~\ref{thm:Boyer}.
Assuming that we are given a homeomorphism~$f' \in \Homeo_{\varphi}^+(\partial M_0,\partial M_1)$ and an isometry~$F' \colon \lambda_0 \cong \lambda_1$,
 Proposition~\ref{prop:FindCompatible} notes that~$(f',F')$ gives rise to a compatible pair~$(f,F)$ if the composition~$f'_* \circ \partial {F'}^{-1}$ is trivial in~$ \Aut(\Bl_{\partial M_1})/  \Homeo^+_{\varphi}(\partial M_1) \times \Aut(\lambda_{M_1}).$
Using this result, we obtain the following consequence of Theorem~\ref{thm:Boyer}.

\begin{theorem}
\label{thm:FindCompatible}
Let~$M_0$ and~$M_1$ be two compact, oriented~$4$-manifolds, with $\pi_1(M_i)=\Z$ for $i=0,1$, whose boundaries are ribbon and have~$\Lambda$-torsion Alexander modules.
Assume that~$M_0$ and $M_1$ have orientation-preserving homeomorphic boundaries, via a homeomorphism that intertwines the maps to~$\Z$, and isometric intersection forms.
If the orbit set
$$\Aut(\Bl_{\partial M_1})/\Homeo^+_{\varphi}(\partial M_1) \times \Aut(\lambda_{M_1})$$ is trivial,
then there is an orientation-preserving homeomorphism~$M_0 \cong M_1$.
\end{theorem}

Theorem~\ref{thm:FindCompatible} is helpful to summarise a programme to prove that two 4-manifolds~$M_0,M_1$ (with infinite cyclic fundamental group and identical ribbon boundaries with torsion Alexander modules) are homeomorphic:
\begin{enumerate}
\item decide whether the~$\Lambda$--intersection forms~$\lambda_{M_0}$ and~$\lambda_{M_1}$ are isometric;
\item show that the orbit set~$\Aut(\Bl_{\partial M_1})/\Homeo^+_{\varphi}(\partial M_1) \times \Aut(\lambda_{M_1})$ is trivial.
\end{enumerate}

To illustrate this paradigm, we explain how Theorem~\ref{thm:WithBoundaryIntro} can be applied to classify $\Z$-surfaces for some knots purely in terms of their intersection pairings.  For any knot $K \subseteq S^3$, multiplication by a monomial $\pm t^k$ gives rise to an automorphism of its Blanchfield pairing $\Bl_{E_K}$ on the exterior of the knot. For some knots $K$, these are the only automorphisms, and in this case we obtain a simpler classification of $\Z$-surfaces, i.e.\ locally flat, oriented surface~$\Sigma$ in~$N$ with~$\pi_1(N_{\Sigma}) \cong \Z$,  with boundary~$K$.

\begin{corollary}\label{corollary:classn-without-boundary-condition-intro}
 Let~$X$ be a closed, simply-connected, oriented~$4$-manifold, and let~$N:=X \setminus \mathring{D}^4$ be a punctured~$X$.
 Let $\Sigma_0$ and $\Sigma_1$ be two $\Z$-surfaces in $N$ with the same genus and oriented boundary $\partial \Sigma_0 = \partial \Sigma_1 = K \subseteq S^3$.
 Suppose that $\Aut(\Bl_{E_K}) \subseteq \{\pm t^k \mid k \in \Z\}$.
 \begin{enumerate}[(i)]
   \item There is a rel.\ boundary homeomorphism of pairs $(N,\Sigma_0) \cong (N,\Sigma_1)$ if and only if there is an isometry $\lambda_{N_{\Sigma_0}} \cong \lambda_{N_{\Sigma_1}}$.
   \item  If $N=D^4$ then $\Sigma_0$ and $\Sigma_1$ are ambiently isotopic rel.\ boundary if and only if there is an isometry  $\lambda_{D^4_{\Sigma_0}} \cong \lambda_{D^4_{\Sigma_1}}$.
 \end{enumerate}
\end{corollary}

\begin{proof}
  Let $F \colon \lambda_{N_{\Sigma_0}} \cong \lambda_{N_{\Sigma_1}}$ be an isometry.  Consider the isomorphism $g:= \partial F|_{H_1(E_K;\Lambda)}$ induced on the Alexander module of $K$. It is an automorphism of the Blanchfield form and therefore by the hypothesis on $\Aut(\Bl_{E_K})$, $g$ is multiplication by $\pm t^k$ for some $k \in \Z$.  Precompose $F$ with multiplication by $\pm t^{-k}$, to obtain a new isometry inducing $\id$ on $H_1(E_K;\Lambda)$. Then apply Theorem~\ref{thm:WithBoundaryIntro} with $f_K=\id$.
\end{proof}

\begin{remark}\label{remark:examples-where-corollary-applies-unitary-units}
  If the Alexander module $H_1(E_K;\Lambda)$ is cyclic with order $\Delta_K$, then the automorphisms of the Blanchfield pairing can be computed directly as follows. Suppose that $\Bl_K(1,1) = b/\Delta_K \in \Q(t)/\Lambda$ for some $b \in \Lambda$. Then $\Aut(\Bl_{E_K}) = \{[p] \in \Lambda/\Delta \Lambda \mid [p\cdot \ol{p} \cdot b] = [b] \in \Lambda / \Delta \Lambda\}$.   In particular for $K = \pm T_{2,3}$ a trefoil, it is not too hard to compute that $\Aut(\Bl_{E_K}) = \{\pm t^k \mid k =-1,0,1 \}$, so that Corollary~\ref{corollary:classn-without-boundary-condition-intro} applies to classify all $\Z$-surfaces with boundary a trefoil in terms of the isometry type of the intersection pairing.
\end{remark}


\subsection*{Organisation}

In Section~\ref{sec:Union}, we review some notions on linking forms and define the union of Hermitian forms over~$\Lambda$.
In Section~\ref{sec:UnionIsUnion}, we show how this algebraic union can be used to express the intersection form of a union of two 4-manifolds; we also prove Theorem~\ref{thm:UnionIsS1S3ConnectSumSimplyConnectedIntro}.
In Section~\ref{sec:Proof}, we prove Theorem~\ref{thm:Boyer}, our partial classification result for~$4$-manifolds with infinite cyclic fundamental group.
In Section~\ref{sec:Knotted}, we prove Theorems~\ref{thm:WithBoundaryIntro}~and \ref{thm:Unknotting4ManifoldIntro}, our main results on properly embedded surfaces in simply-connected 4-manifolds.
In Section~\ref{sec:IntersectionForms}, we discuss the equivariant intersection forms for surfaces with knot group $\Z$, before focusing on surfaces in $S^4$ and $D^4$ in Section~\ref{sec:S4D4}, where we prove Theorems~\ref{thm:UnknottingIntro} and~\ref{thm:D4Intro}.
In Section~\ref{sec:RimSurgery}, we prove Theorems~\ref{thm:RimSurgeryClosedIntro} and~\ref{thm:RimSurgeryIntro} on rim surgery.
In Appendix~\ref{sec:BS}, we adapt the work of Baykur-Sunukjian to properly embedded surfaces in~$4$-manifolds with boundary.

\subsection*{Acknowledgments}
We thank Peter Teichner for pointing out a mistake in a previous version of this work and for stressing the importance of~$(-t)$-Hermitian forms in the proof of Theorem~\ref{thm:HyperbolicIntersectionForm}.
We are indebted to anonymous referees for several very useful suggestions which helped us to improve the exposition and simplify some proofs.
We thank Lisa Piccirillo for helpful discussions, and \.{I}nan\c{c} Baykur and Nathan Sunukjian for discussions on their stabilisation theorem, which helped us when writing the appendix.
Part of this work was completed while the authors were visitors at the Max Planck Institute for Mathematics in Bonn.
MP was partially supported by EPSRC New Investigator grant EP/T028335/1 and EPSRC New Horizons grant EP/V04821X/1.
In the published version of this article, Theorems~\ref{thm:RimSurgeryClosedIntro} and~\ref{thm:RimSurgeryIntro} are stated for $n$-roll $1$-twist rim surgeries for $n \in \Z$.
For $n \neq 0$, in general, $n$-roll $1$-twist rim surgery need not preserve the knot group of a $\Z$-surface~\cite{Teragaito}, and so these theorems were incorrect as stated.
The statements and proofs remain valid for $n=0$, i.e. for $1$-twist rim surgery.
We are grateful to Yujie Lin for bringing this error to our attention.

\subsection*{Conventions}
\begin{enumerate}
\item From now on, all manifolds are assumed to be compact, connected, based and oriented; if a manifold has a nonempty boundary, then the basepoint is assumed to be in the boundary.
We work in the topological category with locally flat embeddings unless otherwise stated.
\item $X$ will always denote a closed, simply-connected, oriented $4$-manifold and we will write $N:=X \setminus \mathring{D}^4$ for the complement of an open ball $\mathring{D}^4 \subseteq X$.
\item A \emph{$\Z$-surface} $\Sigma$ will always refer to a compact, connected, oriented, locally flat, embedded surface in a $4$-manifold, whose knot group is isomorphic to $\Z$; a surface $\Sigma \subseteq X$ is understood to be closed, while a surface $\Sigma \subseteq N$ is understood to be properly embedded.
\item If $P$ is manifold and $Q \subseteq P$ is a submanifold with tubular neighborhood $\nu Q \subseteq P$, then $P_Q := P \sm \nu Q$ will always denote the exterior of $Q$ in $P$.
The only exception to this is that the exterior of a knot $K$ in $S^3$ will be denoted $E_K$ instead of $S^3_K$.
\item Throughout the article, we set~$\Lambda:=\Z[t^{\pm 1}]$ for the ring of Laurent polynomials and~$Q:=\Q(t)$ for its field of fractions.
\item For a manifold $X$, we write $\pi_1(X) =\Z$ to mean that there is an isomorphism $\pi_1(X) \cong \Z$, and that we have fixed a choice of such an isomorphism.  For the exterior of a $\Z$-surface, the orientations on $X$ and $\Sigma$ determine an identification of the fundamental group with~$\Z$. The choice of map $\pi_1(X)=\Z$ determines an isomorphism $\Z[\pi_1(X)] \to \Lambda$, which we use to define homology and cohomology with $\Lambda$ coefficients.
\item We write~$p \mapsto \overline{p}$ for the involution on~$\Lambda$ induced by~$t \mapsto t^{-1}$.
Given a~$\Lambda$-module~$H$, we write~$\overline{H}$ for the~$\Lambda$-module whose underlying abelian group is~$H$ but with module structure given by~$p \cdot h=\overline{p}h$ for~$h \in H$ and~$p \in \Lambda$.
\item We write $H^*:=\overline{\Hom_\Lambda(H,\Lambda)}$.
\item For any ring~$R$, elements of~$R^n$ are considered as column vectors.
\item Let~$f \colon X \to Y$ be a morphism in a category~$\mathcal{C}$, and let~$F \colon \mathcal{C} \to \mathcal{D}$ be a contravariant functor. Assume that~$F(f)$ is invertible.  We shall often write the induced morphism in~$\mathcal{D}$ as~$f^* = F(f)$,  and denote its inverse by~$f^{-*} := (f^*)^{-1} = F(f)^{-1}$.
\end{enumerate}

\section{The union of forms along isometries of their boundary linking forms}
\label{sec:Union}

We develop the theory required for computing the intersection form of a closed~$4$-manifold obtained as the union of two compact~$4$-manifolds along their common boundary.
Here is a summary of this section.
In Subsection~\ref{sub:Boundary}, we review the boundary linking form of a Hermitian form.
In Subsection~\ref{sub:Isometries}, we discuss isometries of Hermitian forms and linking forms.
In Subsection~\ref{sub:Gluing}, we study the union of two Hermitian forms over~$\Lambda:=\Z[t^{\pm 1}]$.
In Subsection~\ref{sub:LagrangianComplement}, we provide a condition for this union to admit a metaboliser.

\subsection{The boundary of Hermitian forms and their isometries}
\label{sub:Boundary}
We fix our terminology on Hermitian forms, linking forms and their boundaries.
References include~\cite[Section~3.4]{RanickiExact} and~\cite[Section 6]{CrowleySixt}.
\medbreak
A \emph{Hermitian form} over~$\Lambda$ is a pair~$(H,\lambda)$, where~$H$ is a free~$\Lambda$-module and~$\lambda \colon H \times H \to \Lambda$ is a sesquilinear Hermitian pairing.
Here by sesquilinear, we mean that~$\lambda(px,qy)=p\lambda(x,y)\overline{q}$ for all~$x,y \in H$ and all~$p,q \in \Lambda$.
By Hermitian, we mean that~$\lambda(y,x) = \overline{\lambda(x,y)}$ for all~$x,y \in H$.
The \emph{adjoint} of~$\lambda$ is the~$\Lambda$-linear map
$\widehat{\lambda} \colon H \to \overline{\Hom_\Lambda(H,\Lambda)} =: H^*$ such that~$\widehat{\lambda}(y)(x)=\lambda(x,y)$.
A Hermitian form is \textit{nondegenerate} if its adjoint is injective and \emph{nonsingular} if its adjoint is an isomorphism.
The \emph{standard hyperbolic form} is
\[H^+(\La) := \bigg(\La \oplus \La, \begin{pmatrix}
  0 & 1 \\ 1 & 0
\end{pmatrix}\bigg).\]

\begin{remark}
\label{rem:Torsion}
If~$(H,\lambda)$ is a nondegenerate Hermitian form, then~$\coker(\widehat{\lambda})$ is a torsion~$\Lambda$-module.
This can be seen by tensoring the following exact sequence with the field of fractions~$Q:=\Q(t)$~of~$\Lambda$:
$$ 0 \to H \stackrel{\widehat{\lambda}}{\hookrightarrow} H^* \to \coker(\widehat{\lambda}) \to 0. ~$$
\end{remark}
A \emph{linking form}~$(T,\beta)$ over~$\Lambda$ consists of a torsion~$\Lambda$-module~$T$ together with a sesquilinear Hermitian form~$\beta \colon T \times T \to Q/\Lambda$.
The \emph{adjoint} of~$\beta$ is the~$\Lambda$-linear map
$\widehat{\beta} \colon T \to \overline{\Hom_\Lambda(T,Q/\Lambda)}$ such that~$\widehat{\beta}(y)(x)=\beta(x,y)$.
A linking form is \textit{nondegenerate} if its adjoint is injective and \emph{nonsingular} if its adjoint is an isomorphism.

\begin{definition}
\label{def:BoundaryLinkingForm}
The \emph{boundary linking form} of a nondegenerate Hermitian form~$(H,\lambda)$ over~$\Lambda$ is the linking form~$(\coker(\widehat{\lambda}),\partial \lambda)$, where~$\partial \lambda$ is defined as
\begin{align*}
 \partial \lambda \colon \coker(\widehat{\lambda}) \times \coker(\widehat{\lambda}) &\to Q/\Lambda  \\
 ([x],[y]) &\mapsto  \frac{1}{s}(y(z)),
\end{align*}
where, since~$\coker(\widehat{\lambda})$ is~$\Lambda$-torsion, there exists an~$s \in \Lambda$ and an~$z \in H$ such that~$sx=\widehat{\lambda}(z)$.
\end{definition}

It is not difficult to show that~$\partial \lambda$ is independent of the choices involved, is sesquilinear and Hermitian.
We conclude with two remarks that we will use throughout this section.
\begin{remark}
\label{rem:TensorQ}
Let~$(H,\lambda)$ be a Hermitian form, and set~$H_Q:=H \otimes_\Lambda Q$.
Since~$H$ is free, we can identify~$\overline{\Hom_\Lambda(H,\Lambda)} \otimes_\Lambda Q$ with~$H_Q^*:=\overline{\Hom_Q(H_Q,Q)}$.
As stated in Remark \ref{rem:Torsion}, if~$(H,\lambda)$ is nondegenerate, then~$\coker(\widehat{\lambda})$ is a torsion~$\Lambda$-module and therefore~$\widehat{\lambda}_Q:=\widehat{\lambda} \otimes_\Lambda \id_Q: H_Q \to H_Q^*$ is a nonsingular Hermitian form over~$Q$.
\end{remark}

Using this remark, we describe an equivalent definition of the boundary linking form.

\begin{remark}
\label{rem:AlternativeDefinitionBoundaryLinking}
The boundary linking form of a nondegenerate Hermitian form~$(H,\lambda)$ can be described as~$\partial \lambda([x],[y])=y (\widehat{\lambda}_Q^{-1}(x))$ for~$[x],[y] \in \coker(\widehat{\lambda})$.
Choose a basis~$\mathbf{b}=(e_i)_{i=1}^n$ for~$H$ and endow~$\overline{\Hom_\Lambda(H,\Lambda)}$ with the dual basis~$\mathbf{b}^*$.
If~$A_{ij}=\lambda(e_i,e_j)$ is a Hermitian matrix representing~$\lambda$, then~$\overline{A}$ is a matrix for~$\widehat{\lambda}$ and~$\partial \lambda([x],[y])=y (\widehat{\lambda}_Q^{-1}(x))=(\overline{A}^{-1}x)^T\overline{y}=x^TA^{-1}\overline{y}$.
For future reference, note that since $A$ is Hermitian, we have $\overline{A}=A^T$.
 \end{remark}

\subsection{Isometries of forms}
\label{sub:Isometries}
We discuss isometries of Hermitian forms and linking forms.
References include~\cite[Section 3.4]{RanickiExact} and~\cite[Section 6]{CrowleySixt}.
\medbreak

Let~$(H_0,\lambda_0)$ and~$(H_1,\lambda_1)$ be Hermitian forms over~$\Lambda$.
A~$\Lambda$-linear isomorphism~$F \colon H_0 \to H_1$ is an \emph{isometry} if~$ \lambda_1(F(x),F(y))=\lambda_0(x,y)$ for all~$x,y \in H_0$.
Let~$\Iso(\lambda_0,\lambda_1)$ denote the set of isometries between the Hermitian forms~$(H_0,\lambda_0)$ and~$(H_1,\lambda_1)$, and let~$\Aut(\lambda) := \Iso(\lambda,\lambda)$ denote the group of self-isometries of a Hermitian form~$(H,\lambda)$.

We make the analogous definitions for linking forms.
Let~$(T_0,\beta_0)$ and~$(T_1,\beta_1)$ be two linking forms over~$\Lambda$.
A~$\Lambda$-linear isomorphism~$h \colon T_0 \to T_1$ is an \emph{isometry} if~$ \beta_1(h(x),h(y))=\beta_0(x,y)$ for all~$x,y \in T_0$.
We let~$\Iso(\beta_0,\beta_1)$ denote the set of isometries between the linking forms~$(T_0,\beta_0)$ and~$(T_1,\beta_1)$, and let~$\Aut(\beta) := \Iso(\beta,\beta)$ denote the group of self-isometries of a linking form~$(T,\beta)$.

An isomorphism~$F \colon H_0 \to H_1$ induces an isomorphism~$F^{-*}:= (F^*)^{-1} \colon H_0^* \to H_1^*$ .
If additionally, the isomorphism~$F \in \Iso(\lambda_0,\lambda_1)$ is an isometry, then~$F^{-*}$ descends to an isomorphism
$$\partial F:=F^{-*} \colon \coker(\widehat{\lambda}_0)\to \coker(\widehat{\lambda}_1).$$
For later use, note that~$F \colon (H_0,\lambda_0) \to (H_1,\lambda_1)$ is an isometry if and only if~$\widehat{\lambda}_0=F^*\widehat{\lambda}_1F$.
Next, we verify that~$\partial F$ is an isometry of the boundary linking forms.

\begin{lemma}
\label{lem:BoundaryIsometry}
If~$F \colon (H_0,\lambda_0) \to (H_1,\lambda_1)$ is an isometry of nondegenerate Hermitian forms, then~$\partial F$ is an isometry of linking forms:
$$\partial F \colon  (\coker(\widehat{\lambda}_0),\partial \lambda_0) \to (\coker(\widehat{\lambda}_1),\partial \lambda_1).$$
This construction provides a map
$$ \partial \colon \Iso(\lambda_0,\lambda_1) \to \Iso(\partial \lambda_0,\partial \lambda_1)$$
which is a homomorphism on automorphism groups.
\end{lemma}
\begin{proof}
Given~$[x],[y] \in \coker(\widehat{\lambda}_0)$, there exists~$s \in \Lambda$ and~$z \in H_0$ such that~$sx=\widehat{\lambda}_0(z)$ .
Since~$F$ is an isometry, it follows that~$sF^{-*}(x)=F^{-*}\widehat{\lambda}_0(z)=\widehat{\lambda}_1F(z)$.
We can now conclude that~$\partial F$ is an isometry because
$$\partial \lambda_1 (\partial F([x]),\partial F([y]))
=\partial \lambda_1 (F^{-*}([x]),F^{-*}([y]))
=\frac{1}{s}F^{-*}(y)F(z)
=\frac{1}{s}y(z)
=\partial \lambda_0([x],[y]).$$
The last assertion follows from the equality~$(G \circ F)^{-*}=G^{-*} \circ F^{-*}$.
\end{proof}

Using Lemma~\ref{lem:BoundaryIsometry}, we are led to the following definition.

\begin{definition}
\label{def:BoundaryOfAnIsometry}
The \emph{boundary of an isometry}~$F \colon (H_0,\lambda_0) \to (H_1,\lambda_1)$ of nondegenerate Hermitian forms is the isometry of linking forms
$$ \partial F \colon \partial (H_0,\lambda_0) \to \partial (H_1,\lambda_1).~$$
\end{definition}

\subsection{The union of nondegenerate Hermitian forms}
\label{sub:Gluing}

We define the union of two Hermitian forms~$(H_0,\lambda_0)$ and~$(H_1,\lambda_1)$ over~$\Lambda$ along an isometry of their boundary linking forms.
The definition is inspired by~\cite[Chapter~3]{CrowleyThesis}, which was concerned with forms over the integers.
\medbreak

We describe the main construction of this subsection.
In what follows, for~$i=0,1$, we use~$\pi_i \colon H_i^* \to \coker(\widehat{\lambda}_i)$ to denote the canonical projections.

\begin{construction}
\label{construction:Gluing}
Let~$(H_0,\lambda_0)$ and~$(H_1,\lambda_1)$ be two nondegenerate Hermitian forms over~$\Lambda$, and
let~$h \colon (\coker(\widehat{\lambda}_0),\partial \lambda_0) \to (\coker(\widehat{\lambda}_1),\partial \lambda_1)$ be an isometry of their boundary linking forms.
Consider the pair~$(H,\lambda)$ with
\begin{align*}
&H:=\ker \big(h \pi_0-\pi_1 \colon H_0^* \oplus H_1^* \to \coker(\widehat{\lambda}_1) \big)  \\
 &\lambda \left( \begin{pmatrix} x_0 \\ x_1 \end{pmatrix}, \begin{pmatrix} y_0 \\ y_1  \end{pmatrix} \right)
 =\frac{1}{s_0}y_0(z_0)-\frac{1}{s_1}y_1(z_1) \in Q,
 \end{align*}
where since~$\coker(\widehat{\lambda}_i)$ is torsion, there exists~$s_i \in \Lambda$ and~$z_i \in H_i$ such that~$s_ix_i=\widehat{\lambda}_i(z_i)$.
Since the Hermitian forms~$\lambda_0$ and~$\lambda_1$ are nondegenerate, it is not difficult to prove that the pairing~$\lambda$ does not depend on the choice of~$s_0,s_1,z_0,z_1$.
\end{construction}

The next proposition establishes some facts about the pairing~$(H,\lambda)$.

\begin{proposition}
\label{prop:PropertiesAlgebraicGluing}
The pair~$(H,\lambda)$ from Construction~\ref{construction:Gluing} satisfies the following properties.
\begin{enumerate}
\item For~$(x_0,x_1),(y_0,y_1) \in H$, the pairing~$\lambda$ can equivalently be defined as
$$ \lambda \left( \begin{pmatrix} x_0 \\ x_1 \end{pmatrix}, \begin{pmatrix} y_0 \\ y_1  \end{pmatrix} \right)
=y_0(\widehat{\lambda}_{0,Q}^{-1}(x_0))-y_1(\widehat{\lambda}_{1,Q}^{-1}(x_1)).$$
\item If we choose bases for~$H_0,H_1$ and dual bases for~$H_0^*,H_1^*$ and let~$A_0,A_1$ be Hermitian matrices representing~$\lambda_0,\lambda_1$, then
$$  \lambda \left( \begin{pmatrix} x_0 \\ x_1 \end{pmatrix}, \begin{pmatrix} y_0 \\ y_1  \end{pmatrix} \right)
=x_0^TA_0^{-1}\overline{y}_0-x_1^TA_1^{-1}\overline{y}_1.
  ~$$
\item The pairing~$\lambda$ is sesquilinear, Hermitian and takes values in~$\Lambda$.
\item The following two maps, which we abridge by~$\widehat{\lambda}_0,\widehat{\lambda}_1$, are injective:
$$ (H_0,\lambda_0) \stackrel{\bsm \widehat{\lambda}_0\\0 \esm}{\hookrightarrow} (H,\lambda) \stackrel{\bsm 0 \\\widehat{\lambda}_1\esm}{\hookleftarrow}  (H_1,-\lambda_1).$$
Furthermore, these maps satisfy~$\widehat{\lambda}_0(H_0)^\perp=\widehat{\lambda}_1(H_1)$.
\end{enumerate}
\end{proposition}

\begin{proof}
To prove the first assertion, notice that since~$\lambda_{i,Q}$ is a nonsingular pairing for~$i=0,1$, we can write~$\widehat{\lambda}_{i,Q}^{-1}(x_i)=z_i/s_i$ for some~$z_i \in H$ and some~$s_i \in \Lambda$.
Consequently,~$\widehat{\lambda}_i(z_i)=s_ix_i$ and we have~$\frac{1}{s_i}y_i(z_i)=y_i\left(\widehat{\lambda}_{i,Q}^{-1}(x_i)\right)$, as desired.
This proves the first assertion.
The second assertion now follows as in Remark~\ref{rem:AlternativeDefinitionBoundaryLinking}.

We prove the third assertion.
The fact that~$\lambda$ is sesquilinear and Hermitian follows from the second assertion and the fact that~$A_0^{-1}$ and~$A_1^{-1}$ are Hermitian.
Next, using the definition of the boundary linking forms~$\partial \lambda_0$ and~$\partial \lambda_1$ from Definition~\ref{def:BoundaryLinkingForm}, we have
$$ \lambda \left( \begin{pmatrix} x_0 \\ x_1 \end{pmatrix}, \begin{pmatrix} y_0 \\ y_1  \end{pmatrix} \right)= \partial \lambda_0 ([x_0],[y_0])-\partial \lambda_1 ([x_1],[y_1]) \ \ \ \ \operatorname{mod} \Lambda.~$$
By definition of~$H$, we have~$h[x_0]=x_1$ and $h[y_0]=y_1$. Then since~$h$ is an isometry, this latter expression vanishes:
\begin{align*}
 \partial \lambda_0 ([x_0],[y_0])-\partial \lambda_1 ([x_1],[y_1])
&\equiv \partial \lambda_0 ([x_0],[y_0])-\partial \lambda_1 (h[x_0],h[y_0])  \\
&\equiv \partial \lambda_0 ([x_0],[y_0])-\partial \lambda_0 ([x_0],[y_0]) \equiv 0 \in Q/\Lambda.
\end{align*}
It follows that~$\lambda$ takes values in~$\Lambda$, concluding the proof of the third assertion.
We now prove the fourth assertion.
First, we check that~$\widehat{\lambda}_0$ is an isometric embedding; the proof for~$\widehat{\lambda}_1$ is identical.
Since the pairings are nondegenerate, the maps are injective and given~$z_0,z_0' \in H_0$, we have
$$\lambda \left( \begin{pmatrix} \widehat{\lambda}_0(z_0) \\ 0 \end{pmatrix}, \begin{pmatrix} \widehat{\lambda}_0(z_0') \\ 0  \end{pmatrix}\right)
=\widehat{\lambda}(z_0')(z_0)=\lambda_0(z_0,z_0') .~$$
It remains to check that~$\widehat{\lambda}_0(H_0)^\perp=\widehat{\lambda}_1(H_1)$.
The inclusion~$\widehat{\lambda}_0(H_0)^\perp \supseteq \widehat{\lambda}_1(H_1)$ is clear and so we prove the inclusion~$\widehat{\lambda}_0(H_0)^\perp \subseteq\widehat{\lambda}_1(H_1)$.
Assume that~$(x_0,x_1) \in H$ (with~$sx_0=\widehat{\lambda}_0(z_0)$) satisfies for all~$a_0 \in H_0$
$$0
=\lambda \left( \begin{pmatrix} x_0 \\ x_1  \end{pmatrix},\begin{pmatrix} \widehat{\lambda}_0(a_0) \\ 0 \end{pmatrix} \right)
=\frac{1}{s}\widehat{\lambda}_0(a_0)(z_0)
=\frac{1}{s}\lambda_0(z_0,a_0) \in \Lambda.$$
As~$\lambda_0$ is nondegenerate, this implies that~$z_0=0$ and therefore~$x_0=0$.
But since~$(x_0,x_1) \in H$, we also have~$[x_1]=h[x_0]=0$ and so~$x_1=\widehat{\lambda}_1(z_1)$ for some~$z_1\in H_1$.
 We therefore conclude that~$\widehat{\lambda}_0(H_0)^\perp=\widehat{\lambda}_1(H_1)$, establishing the proposition.
\end{proof}

Note that Proposition~\ref{prop:PropertiesAlgebraicGluing} does not contain a statement about the~$\Lambda$-module~$H$ being free.
Since we defined Hermitian forms over free~$\Lambda$-modules, a \emph{Hermitian module}~$(H,\lambda)$ will refer to a pair consisting of a finitely generated~$\Lambda$-module~$H$ and a sesquilinear Hermitian pairing~$\lambda \colon H \times H \to \Lambda$; we drop the requirement that~$H$ be free.

\begin{definition}
\label{def:Gluing}
Let~$(H_0,\lambda_0)$ and~$(H_1,\lambda_1)$ be two nondegenerate Hermitian forms over~$\Lambda$, and
let~$h \colon (\coker(\widehat{\lambda}_0),\partial \lambda_0) \to (\coker(\widehat{\lambda}_1),\partial \lambda_1)$ be an isometry of their boundary linking forms.
The \emph{union} of~$(H_0,\lambda_0)$ and~$(H_1,\lambda_1)$ along~$h$ is the Hermitian module~$(H_0 \cup_h H_1,\lambda_0\cup_h -\lambda_1)$ described in Construction~\ref{construction:Gluing} and Proposition~\ref{prop:PropertiesAlgebraicGluing}:
\begin{align*}
&H_0 \cup_h H_1:=\ker \big(h \pi_0-\pi_1 \colon H_0^* \oplus H_1^* \to \coker(\widehat{\lambda}_1) \big)  \\
 &\lambda_0\cup_h -\lambda_1 \left( \begin{pmatrix} x_0 \\ x_1 \end{pmatrix}, \begin{pmatrix} y_0 \\ y_1  \end{pmatrix} \right)
 =\frac{1}{s_0}y_0(z_0)-\frac{1}{s_1}y_1(z_1) \in \Lambda,
 \end{align*}
where since~$\coker(\widehat{\lambda}_i)$ is torsion, there exists~$s_i \in \Lambda$ and~$z_i \in H_i$ such that~$s_ix_i=\widehat{\lambda}_i(z_i)$.
\end{definition}

\begin{remark}
\label{rem:NonSing}
We make a brief remark about the nonsingularity of~$(H_0 \cup_h H_1,\lambda_0\cup_h -\lambda_1)$.
Proposition~\ref{prop:PropertiesAlgebraicGluing} contains no statement about~$(H,\lambda):=(H_0 \cup_h H_1,\lambda_0\cup_h -\lambda_1)$ being nonsingular.
While this holds over the integers~\cite[Lemma 3.6]{CrowleyThesis} and in the topological setting of Proposition~\ref{prop:UnionIsUnion}, we will not prove it in the algebraic generality of this section.
Note that if~$\widehat{\lambda} \colon H \xrightarrow{\cong} \overline{\Hom_\Lambda(H,\Lambda)}=:H^*$ is nonsingular, then~$H$ must be free, since for any finitely generated~$\Lambda$-module~$H$, the dual~$H^*$ is free~\cite[Lemma 2.1]{BorodzikFriedlOnTheAlgebraic}.
\end{remark}

Given a~$\Lambda$-module~$H$, we set~$H^*:=\overline{\Hom_\Lambda(H,\Lambda)}$.
If~$H$ is free, then there is a~$\Lambda$-linear evaluation isomorphism~$\ev \colon H \to H^{**}$ given by~$\ev_x(\varphi)=\overline{\varphi(x)}$ where~$x \in H$ and~$\varphi \in H^*$.
The next remark uses~$\ev$ to describe the adjoint of~$\lambda_0 \cup_h-\lambda_1$.

\begin{remark}
\label{rem:AdjointOfUnion}
Let~$(H_0,\lambda_0)$ and~$(H_1,\lambda_1)$ be two nondegenerate Hermitian forms over~$\Lambda$, and let~$h \in \Iso (\partial \lambda_0,\partial \lambda_1)$ be an isometry.
The adjoint of~$\lambda:=\lambda_0\cup_h -\lambda_1$ is given by
\begin{align*}
\widehat{\lambda} \colon H_0 \cup_h H_1 &\to (H_0 \cup_h H_1)^* \\
\begin{pmatrix}
y_0 \\ y_1
\end{pmatrix} &\mapsto
 \begin{pmatrix}
\ev &  \ev
\end{pmatrix}
 \begin{pmatrix}
\widehat{\lambda}_{0,Q}^{-1} & 0 \\
 0 &-\widehat{\lambda}_{1,Q}^{-1}
\end{pmatrix}
\begin{pmatrix}
y_0 \\ y_1
\end{pmatrix}.
\end{align*}
This follows by combining the definition of~$\ev$ with the second and third items of Proposition~\ref{prop:PropertiesAlgebraicGluing}.
Indeed for~$(x_0,y_0) \in H_0 \cup_h H_1$, we have
\begin{align*}
 \lambda \left( \begin{pmatrix} x_0 \\ x_1 \end{pmatrix}, \begin{pmatrix} y_0 \\ y_1  \end{pmatrix} \right)
&=\overline{ \lambda \left( \begin{pmatrix} y_0 \\ y_1 \end{pmatrix}, \begin{pmatrix} x_0 \\ x_1  \end{pmatrix} \right) }
=\overline{x_0(\widehat{\lambda}_{0,Q}^{-1}(y_0))}-\overline{x_1(\widehat{\lambda}_{1,Q}^{-1}(y_1))}  \\
&=
\begin{pmatrix} \ev_{\widehat{\lambda}_{0,Q}^{-1}(y_0)} & \ev_{-\widehat{\lambda}_{1,Q}^{-1}(y_1)}  \end{pmatrix}
\begin{pmatrix}
x_0 \\ x_1
\end{pmatrix}
=\widehat{\lambda}
\begin{pmatrix}
y_0 \\ y_1
\end{pmatrix}
\begin{pmatrix}
x_0 \\ x_1
\end{pmatrix}.
\end{align*}
\end{remark}

We conclude this section by describing how the union interacts with boundary isometries.
The proof follows immediately from the definitions and is left to the reader.

\begin{proposition}
\label{prop:IsometriesOfUnion}
Let~$(H_0,\lambda_0)$ and~$(H_1,\lambda_1)$ be Hermitian forms,  and let~$h \in \Iso(\partial \lambda_0,\partial \lambda_1)$ be an isometry.
If~$F \in \Iso(\lambda_0,\lambda_1)$ is an isometry, then
$F^{-*} \oplus \id$ induces an isometry
$$ (H_0,\lambda_0) \cup_h (H_1,-\lambda_1) \to (H_1,\lambda_1) \cup_{h \circ  \partial F^{-1}} (H_1,-\lambda_1).$$
\end{proposition}

\subsection{Lagrangian complements}
\label{sub:LagrangianComplement}
This subsection provides a criterion for a union of Hermitian forms to be metabolic. It motivates the notion of a compatible pair that will be introduced in the next subsection.
The idea for this criterion stems from work of Kreck~\cite[Proposition 8]{KreckSurgeryAndDuality} and Crowley-Sixt~\cite[Theorem 5.11]{CrowleySixt} and their work on the monoid $\ell_{2q+1}(\Z[\pi])$.
\medbreak

A \emph{Lagrangian} for a nonsingular Hermitian form~$(H,\lambda)$ is a direct summand~$L \subseteq H$ such that~$L^\perp=L$.
A nonsingular Hermitian form that admits a Lagrangian is called \emph{metabolic}.
A \emph{Lagrangian complement} for a half rank direct summand~$V \subseteq (H,\lambda)$ is a Lagrangian of~$(H,\lambda)$ such that~$L \oplus V=H$.

In the next proposition, for an isomorphism~$F \colon H_0 \to H_1$ we consider the graph \[\Gamma_{F^{-*}} = \{(x,F^{-*}(x)) \mid x \in H_0^*\} \subseteq H_0^* \oplus H_1^*.\]
Since, by definition of $\partial F$, we have~$\partial F(x) - F^{-*}(x) = F^{-*}(x) - F^{-*}(x) =0$, we deduce the inclusions~$\Gamma_{F^{-*}} \subseteq \ker(\partial F \circ \pi_0-\pi_1) = H_0 \cup_{\partial F} H_1 \subseteq H_0^* \oplus H_1^*$.

\begin{proposition}
\label{prop:IsElementary}
Let~$F \colon (H_0,\lambda_0) \to (H_1,\lambda_1)$ be a~$\Lambda$-isometry such that~$\lambda_0 \cup_{\partial F} -\lambda_1$ is nonsingular.
The graph~$\Gamma_{F^{-*}} \subseteq H_0 \cup_{\partial F} H_1$ is a Lagrangian complement for~$\widehat{\lambda}_0(H_0)$:
$$H_0 \cup_{\partial F} H_1=\widehat{\lambda}_0(H_0) \oplus \Gamma_{F^{-*}}.$$
\end{proposition}
\begin{proof}
Set~$\lambda:=\lambda_0 \cup_{\partial F}-\lambda_1$.
We show that~$\Gamma_{F^{-*}}$ is a Lagrangian.
To see that~$\Gamma_{F^{-*}} \subseteq \Gamma_{F^{-*}}^\perp$, we must show that for all~$x_0,y_0 \in H_0^*$,
\begin{align*}
\lambda \left(
 \begin{pmatrix} x_0 \\ F^{-*}(x_0) \end{pmatrix} ,
\begin{pmatrix} y_0 \\ F^{-*}(y_0) \end{pmatrix}
\right)=0.
\end{align*}
Pick~$s_0 \in \Lambda$ and~$z_0 \in H_0$ such that~$s_0x_0=\widehat{\lambda}_0(z_0)$.
We apply~$F^{-*}$ to both sides of this equation.
Since~$F$ is an isometry, we have~$F^*\widehat{\lambda}_1F=\widehat{\lambda}_0$, and we therefore obtain~$s_0F^{-*}(x_0)=\widehat{\lambda}_1(F(z_0))$.
Using the definition of~$\lambda$, we now obtain the desired conclusion:
$$ \lambda \left(
 \begin{pmatrix} x_0 \\ F^{-*}(x_0) \end{pmatrix} ,
\begin{pmatrix} y_0 \\ F^{-*}(y_0) \end{pmatrix}
\right)
=\frac{1}{s_0}y_0(z_0)-\frac{1}{s_0}F^{-*}(y_0)(F(z_0))
=0.~$$
Next, we show that~$\Gamma_{F^{-*}}^\perp \subseteq \Gamma_{F^{-*}}$.
We therefore assume that~$(x_0,x_1) \in H_0 \cup_{\partial F} H_1$ satisfies, for all~$y_0 \in H_0^*$, the equation
$$\lambda \left(
\begin{pmatrix} x_0 \\ x_1 \end{pmatrix} ,
 \begin{pmatrix} y_0 \\ F^{-*}(y_0) \end{pmatrix}
\right)=0~$$
and we must show that~$x_1=F^{-*}(x_0)$.
Pick~$s_0,s_1 \in \Lambda, z_0 \in H_0$ and~$z_1 \in H_1$ so that~$s_0x_0=\widehat{\lambda}_0(z_0)$ and~$s_1 x_1=\widehat{\lambda}_1(z_1)$.
We have~$s_0F^{-*}(x_0)=\widehat{\lambda}_1(F(z_0)).$
Use consecutively the definition of~$\lambda$ and the definition of pullback to obtain
\begin{align*}
\label{eq:NonSingular}
0=\lambda \left(
\begin{pmatrix} x_0 \\ x_1 \end{pmatrix} ,
 \begin{pmatrix} y_0 \\ F^{-*}(y_0) \end{pmatrix}
\right)
=y_0\left(\frac{z_0}{s_0}\right)-F^{-*}(y_0)\left(\frac{z_1}{s_1}\right)
=y_0\left(\frac{z_0}{s_0}-F^{-1}\left(\frac{z_1}{s_1}\right)\right).
\end{align*}
Since this equation holds for all~$y_0$, we deduce that \[\frac{z_0}{s_0}=F^{-1}\left(\frac{z_1}{s_1}\right).\]
Using that~$F$ is an isometry, as well as the definitions of~$s_0,s_1,z_0$ and~$z_1$, we obtain the desired equation:
$$ x_0
=\widehat{\lambda}_0\left(\frac{z_0}{s_0}\right)
=\widehat{\lambda}_0\left(F^{-1}\left(\frac{z_1}{s_1}\right)\right)
=F^*\left(\widehat{\lambda}_1\left(\frac{z_1}{s_1}\right)\right)=F^*(x_1).$$
To conclude the proof that~$\Gamma_{F^{-*}}$ is a Lagrangian, it remains to show that it is a direct summand of~$H:=\ker (\partial F \circ  \pi_0 -\pi_1)$.
In fact, we will show that
$$\widehat{\lambda}_0(H_0) \oplus \Gamma_{F^{-*}}=H.$$
First we establish the inclusion~$\widehat{\lambda}_0(H_0) + \Gamma_{F^{-*}} \subseteq H$.
To do so, we must prove that the sum of arbitrary elements~$(\widehat{\lambda}_0(z_0),0) \in \widehat{\lambda}_0(H_0)$ and~$(y,F^{-*}(y)) \in \Gamma_{F^{-*}}$ belongs to~$H$.
In other words, we must show that~$\partial F[\widehat{\lambda}_0(z_0)+y]=[F^{-*}(y)]$.
This uses the fact that~$F^{-*}$ induces~$\partial F$ (by Definition~\ref{def:BoundaryOfAnIsometry}) and the fact that, since~$F$ is an isometry,~$F^{-*} \circ  \widehat{\lambda}_0(z_0)=\widehat{\lambda}_1(F(z_0))$ belongs to~$\operatorname{im}(\widehat{\lambda}_1)$ and so the class~$[F^{-*} \circ  \widehat{\lambda}_0(z_0)]$ must vanish in the quotient~$\coker(\widehat{\lambda}_1)$:
$$\partial F[\widehat{\lambda}_0(z_0)+y]=[F^{-*} \widehat{\lambda}_0(z_0)]+[F^{-*}(y)]=[F^{-*}(y)].$$
Next, we prove the reverse inclusion, namely that~$H \subseteq \widehat{\lambda}_0(H_0) + \Gamma_{F^{-*}}$.
To show this, we must write~$(x,y) \in H$ as a sum of an element in~$\widehat{\lambda}_0(H_0)$ with an element in~$\Gamma_{F^{-*}}$.
Consider the decomposition
$$ (x,y)=(x-F^{*}(y),0)+(F^{*}(y),y).$$
The second term certainly lies in~$\Gamma_{F^{-*}}$, so we need to argue that the first term lies in~$\widehat{\lambda}_0(H_0)$.
Since~$(x,y) \in H$, we know that~$\partial F [x]=[y]$ in~$\coker(\widehat{\lambda}_1)$.
It follows that~$F^{-*}(x)-y \in H_1^*$ belongs to~$\operatorname{im}(\widehat{\lambda}_1)$.
Since~$F$ is an isometry, this is equivalent to saying that~$x-F^{*}(y) \in \operatorname{im}(\widehat{\lambda}_0)$, as desired.

Having proved that~$\widehat{\lambda}_0(H_0) + \Gamma_{F^{-*}}=H$, it remains to check that~$\widehat{\lambda}_0(H_0) \cap \Gamma_{F^{-*}}=0$ (so that the sum is direct): if~$(\widehat{\lambda}_0(x_0),0)=(x,F^{-*}(x))$ belongs to this intersection, then~$F^{-*}(x)=0$ and therefore~$0=x=\widehat{\lambda}_0(x_0)$.
This concludes the proof of Proposition~\ref{prop:IsElementary}.
\end{proof}

%

\section{The intersection form of the union of two~$4$-manifolds}
\label{sec:UnionIsUnion}

We start our study of 4-manifolds with infinite cyclic fundamental group and of their~$\Lambda$--intersection forms.
First, recall the following definition from the introduction.

\begin{definition}
\label{def:RibbonBoundary}
A~$4$-manifold~$M$ has \emph{ribbon boundary} if~$\partial M$ is nonempty and path-connected, and the map~$\pi_1(\partial M) \to \pi_1(M)$ induced by the inclusion is surjective.
\end{definition}

Here is a summary of Section~\ref{sec:UnionIsUnion}.
In Subsection~\ref{sub:BoundaryIntersectionForm}, we show that if~$M$ is a~$4$-manifold with infinite cyclic fundamental group whose boundary is ribbon and has torsion Alexander module, then the boundary linking form~$\partial \lambda_M$ of~$\lambda_M$ is isometric to minus the Blanchfield form of~$\partial M$.
In Subsection~\ref{sub:HomeoIso}, we describe when homeomorphisms of 3-manifolds induce isometries of the Blanchfield pairing.
In Subsection~\ref{sub:UnionIsUnion} we show that the~$\Lambda$--intersection form of a union of two manifolds can be expressed using the algebraic union from Subsection~\ref{sub:Gluing}.
In Subsection~\ref{sub:Compatible}, we introduce the notion of a compatible pair and prove Theorem~\ref{thm:UnionIsS1S3ConnectSumSimplyConnectedIntro} from the introduction.

\subsection{The boundary of the intersection form}
\label{sub:BoundaryIntersectionForm}
If a~$4$-manifold~$M$ with~$\pi_1(M)=\Z$ has ribbon boundary and~$\Lambda$-torsion Alexander module, then~$\partial M$ is endowed with a  sesquilinear Hermitian nonsingular linking form over~$\Lambda$, namely the Blanchfield form
$$\Bl_{\partial M} \colon H_1(\partial M;\Lambda) \times H_1(\partial M;\Lambda) \to Q/\Lambda.~$$
This subsection establishes that the boundary linking form of the~$\Lambda$--intersection form of~$M$ is isometric to minus the Blanchfield form of~$\partial M$.
Since results of this form are known (see e.g.~\cite[Theorem 2.6]{BorodzikFriedl1}), we only outline the argument so as to fix some notations for later~use.
\medbreak
We collect some homological facts about~$4$-manifolds with infinite cyclic fundamental group.

\begin{lemma}
\label{lem:Homology}
Let~$M$ be a~$4$-manifold with~$\pi_1(M)=\Z$.
When~$\partial M$ is nonempty, we assume it is ribbon and has torsion Alexander module~$H_1(\partial M;\Lambda)$.
The following assertions hold:
\begin{enumerate}
\item $H_0(M;\Lambda)\cong \Z$ and, if~$\partial M \neq \emptyset$, then~$H_2(\partial M;\Lambda)\cong \Z$;
\item the~$\Lambda$-modules~$H_1(M;\Lambda), H_1(M,\partial M;\Lambda)$ and~$H_3(M;\Lambda)$ all vanish;
\item the~$\Lambda$-modules~$H_2(M;\Lambda)$ and~$H_2(M,\partial M;\Lambda)$ are free;
\item when~$\partial M=\emptyset$, the~$\Lambda$--intersection form~$\lambda_M$ is nonsingular, whereas for~$\partial M \neq \emptyset$,~$\lambda_M$ is nondegenerate, and any matrix representing it presents the Alexander module of~$\partial M$.
\end{enumerate}
\end{lemma}

\begin{proof}
Throughout this proof we use that since the manifold has fundamental group~$\Z$, homology with~$\Lambda$-coefficients can be computed as the homology of the universal cover.
Since universal covers are~$1$-connected, we immediately deduce that~$H_0(M;\Lambda)\cong \Z$ and~$H_1(M;\Lambda)=0$.
If the boundary is nonempty, the long exact sequence of the pair shows that $H_1(M,\partial M;\Lambda)=0$; here we used that~$\partial M$ is connected, as assumed in Definition~\ref{def:RibbonBoundary}.
For later use, we also note that~$H_0(M,\partial M;\Lambda)=0$, also owing to the fact that~$\partial M$ is connected.
As~$H_1(\partial M;\Lambda)$ is torsion, we have~$\Hom_\Lambda(H_1(\partial M;\Lambda),\Lambda)=0$, and thus~$H_2(\partial M;\Lambda) \cong \operatorname{Ext}_\Lambda^1(H_0(\partial M;\Lambda),\Lambda) \cong \Z$ by duality and the universal coefficient spectral sequence~\cite[Theorem 2.3]{LevineKnotModules}.
Since $H_i(M,\partial M;\Lambda)=0$ for~$i=0,1$, duality and the universal coefficient spectral sequence (UCSS for short) also show that~$H_3(M;\Lambda)=H^1(M,\partial M;\Lambda)=0$.
Thus we have proved the first two assertions; we now prove the third.
Poincar\'e duality and the UCSS imply that
\begin{align*}
&H_2(M;\Lambda) \cong H^2(M,\partial M;\Lambda) \cong \overline{\Hom_\Lambda(H_2(M,\partial M;\Lambda),\Lambda)}, \\
& H_2(M,\partial M;\Lambda) \cong H^2(M;\Lambda) \cong \overline{\Hom_\Lambda(H_2(M;\Lambda),\Lambda)}.
\end{align*}
Since the dual of a finitely generated~$\Lambda$-module is free~\cite[Lemma 2.1]{BorodzikFriedlOnTheAlgebraic}, we deduce that these second homology modules are free over~$\Lambda$.
This proves the third assertion and also establishes the fourth, namely that in the closed case, the~$\Lambda$--intersection form is nonsingular.
Finally, when the boundary is nonempty, the map~$\Z \cong H_2(\partial M;\Lambda) \to H_2(M;\Lambda)$ is the zero map (since~$H_2(M;\Lambda)$ is free) and therefore the intersection form is nondegenerate; it presents~$H_1(\partial M;\Lambda)$ because we established that~$H_1(M;\Lambda)=0$.
\end{proof}

In the case with nonempty connected boundary, we fix bases for the free~$\Lambda$-modules~$H_2(M;\Lambda)$ and~$H_2(M,\partial M;\Lambda)$.

\begin{remark}
\label{rem:Bases}
Let~$M$ be a~$4$-manifold with~$\pi_1(M)=\Z$ whose boundary is ribbon and has torsion Alexander module.
Fix a basis~$\mathbf{b}$ for~$H:=H_2(M;\Lambda)$, endow the dual~$H^*:=\overline{\Hom_\Lambda(H,\Lambda)}$ with the dual basis~$\mathbf{b}^*$, and equip~$H_2(M,\partial M;\Lambda)$ with the basis~$\PD \circ \ev^{-1} (\mathbf{b}^*)$ coming from the isomorphisms~$\ev \colon H^2(M;\Lambda) \to H^*$ and ~$\PD \colon H^2(M;\Lambda) \to  H_2(M,\partial M;\Lambda)$.
A short computation shows that if the matrix~$A$ is defined as~$A_{ij}:=\lambda_M(b_i,b_j)$, then~$A^T=\overline{A}$ is a matrix for~$\widehat{\lambda}_M \colon H\to H^*$; recall Remark~\ref{rem:Torsion}.
The same conclusion holds for the map~$H \to H_2(M,\partial M;\Lambda)$ induced by the inclusion~\cite[Section~5.2]{ConwayFriedlToffoli}.
As mentioned in the fourth item of Lemma~\ref{lem:Homology},
the connecting homomorphism~$\delta$ in the long exact sequence of the pair~$(M,\partial M)$, together with Poincar\'e duality and the evaluation map, determines a map
\begin{align*} D_M \colon \coker(\widehat{\lambda}_M)=\coker(A^T)  &\xrightarrow{\cong} H_1(\partial M;\Lambda) \\
[x]  &\mapsto \delta \circ \PD \circ \ev^{-1}(x),
\end{align*}
where $x\in H_2(M;\Lambda)^*$.
The map $D_M$ is well-defined and an isomorphism. That it is well-defined follows from a diagram chase in the next diagram, in which all homology has $\Lambda$ coefficients.  The left two vertical maps are isomorphisms, and therefore by the five lemma so is $D_M$.  To help the reader to parse the basis choices above, we also indicate the matrices representing the maps on the left of the diagram:
\begin{equation}\label{eqn:delta-induces-map}
\xymatrix{
0\ar[r]&H_2(M) \ar[r]^-{\widehat{\lambda}_M}_-{A^T} \ar[d]^{\Id}_{I} & H_2(M)^* \ar[r] \ar[d]^{\PD \circ \ev^{-1}}_{I} & \coker(\widehat{\lambda}_M) \ar[r] \ar[d]^{D_M} & 0 \\
0\ar[r]& H_2(M) \ar[r]^-{j}_-{A^T}  & H_2(M,\partial M) \ar[r]^-{\delta}  & H_1(\partial M) \ar[r] & 0.  }
\end{equation}
\end{remark}

Next, we briefly recall the definition of the Blanchfield pairing, referring to~\cite{ConwayFriedlToffoli, FriedlPowell} as references in which the conventions are identical to ours.

\begin{definition}
\label{rem:Blanchfield}
Let~$N$ be a closed~$3$-manifold with an epimorphism~$\varphi \colon \pi_1(N) \twoheadrightarrow \Z$ such that the resulting Alexander module is~$\Lambda$-torsion.
The adjoint of the Blanchfield pairing is defined by the composition
\begin{align*}
H_1(N;\Lambda)
 \xrightarrow{\PD^{-1},\cong} H^2(N;\Lambda)
  \xrightarrow{BS^{-1},\cong} H^1(N;Q/\Lambda)
    \xrightarrow{\ev,\cong} \overline{\Hom_\Lambda(H_1(N;\Lambda),Q/ \Lambda)}
\end{align*}
of the inverse of Poincar\'e duality, the inverse of a Bockstein homomorphism, and the evaluation map.
The Blanchfield pairing is a nonsingular linking form, so in particular it is sesquilinear and Hermitian.
\end{definition}

Now we can prove the main result of this subsection.

\begin{proposition}
\label{prop:BoundaryLinkingForm3Manifold}
If~$M$ is a~$4$-manifold with~$\pi_1(M)=\Z$ whose boundary is ribbon and has torsion Alexander module. The map $D_M$ defined in Remark~\ref{rem:Bases} induces an isometry
$$D_M \colon \partial (H_2(M;\Lambda),\lambda_M) = (\coker(\widehat{\lambda}_M),\partial \lambda_M) \xrightarrow{\cong} (H_1(\partial M;\Lambda),-\Bl_{\partial M}).~$$
\end{proposition}

\begin{proof}
Choose bases for~$H_2(M;\Lambda)$ and~$H_2(M,\partial M;\Lambda)$ as in Remark~\ref{rem:Bases}, and use the notation of that remark.
The same argument as in~\cite[Section~5]{ConwayFriedlToffoli} shows that the following diagram (with exact columns) commutes:
$$
\xymatrix@R0.5cm@C2cm{
0\ar[d]&0\ar[d] \\
H_2(M;\Lambda) \times H_2(M;\Lambda) \ar[r]^-{(x,y)\mapsto -x^T A \overline{y}}\ar[d]^{A^T \times A^T}&\Lambda \ar[d]  \\
H_2(M,\partial M;\Lambda) \times H_2(M,\partial M;\Lambda) \ar[r]^-{(x,y)\mapsto -x^T A^{-1} \overline{y}}\ar[d]^{\delta \times \delta } &Q \ar[d]  \\
H_1(\partial M;\Lambda) \times H_1(\partial M;\Lambda) \ar[d]\ar[r]^-{\Bl_{\partial M}}&Q/\Lambda\ar[d] \\
0&0.
}
$$
In particular, the isomorphism $D_M \colon \coker(\widehat{\lambda}_M) = \coker(A^T)  \xrightarrow{\cong} H_1(\partial M;\Lambda)$ induced by $\delta$ and discussed in Remark~\ref{rem:Bases} is an isometry between the linking forms $(\coker(\widehat{\lambda}_M),\partial \lambda_M)$ and~$(H_1(\partial M;\Lambda),-\Bl_{\partial M})$.
Here we are using the characterisation of~$\partial \lambda_M$ from Remark~\ref{rem:AlternativeDefinitionBoundaryLinking}.
\end{proof}

The next remark records that the presence of a minus sign in Proposition~\ref{prop:BoundaryLinkingForm3Manifold} is immaterial once one passes to isometry groups, and defines the bijection $D^\#$ which identifies the two isometry groups.

\begin{remark}
\label{rem:MinusAndIdentifIso}
Given linking forms~$(H_0,\beta_0)$ and~$(H_1,\beta_1)$, observe that canonically we have $\Iso(\beta_0,\beta_1)=\Iso(-\beta_0,-\beta_1)$; the presence of the minus sign has no effect.
In particular if $M_0$ and $M_1$ are~$4$-manifolds with $\pi_1(M_i)=\Z$ for $i=0,1$, whose boundaries are ribbon and have torsion Alexander modules, then we can use the maps $D_i := D_{M_i}$ from Proposition~\ref{prop:BoundaryLinkingForm3Manifold} to obtain an identification
\begin{align*}
 D^\# \colon \Iso (\partial \lambda_{M_0},\partial \lambda_{M_1}) &\xrightarrow{\cong} \Iso (\Bl_{\partial M_0},\Bl_{\partial M_1})\\
\wt{f} & \mapsto D_{1} \circ \wt{f} \circ D_{0}^{-1}.
\end{align*}
\end{remark}

\subsection{Homeomorphisms and isometries}
\label{sub:HomeoIso}

In this short subsection, we provide a condition for a homeomorphism between 3-manifolds to induce an isometry of Blanchfield forms.

\begin{proposition}
\label{prop:HomeoInduceIsometry}
Let~$Y_0^3,Y_1^3$ be~$3$-manifolds equipped with epimorphisms~$\varphi_i \colon \pi_1(Y_i) \twoheadrightarrow \Z$ and assume that the resulting Alexander modules~$H_1(Y_i;\Lambda)$ are~$\Lambda$-torsion for~$i=0,1.$
If an orientation-preserving homotopy equivalence~$f \colon Y_0^3 \to Y_1^3$ satisfies~$\varphi_1 \circ f_* =\varphi_0$ on~$\pi_1(Y_0)$, then it induces an isometry between the Blanchfield forms:
$$ f_* \colon (H_1(Y_0;\Lambda), \Bl_{Y_0})\to (H_1(Y_1;\Lambda), \Bl_{Y_1}).$$
\end{proposition}

\begin{proof}
The homotopy equivalence~$f$ induces~$\Lambda$-isomorphisms~$f_* \colon H_*(Y_0;\Lambda) \to~H_*(Y_1;\Lambda)$ (and similarly on cohomology) because we assumed that~$\varphi_1 \circ f_* =\varphi_0$.
The fact that~$f$ has degree one, together with the naturality of the UCSS and the Bockstein homomorphism ensure that~$f_*$ intertwines the Blanchfield pairings, concluding the proof of the proposition.
\end{proof}

For pairs~$(Y_0,\varphi_0)$ and~$(Y_1,\varphi_1)$ as in the statement of Proposition~\ref{prop:HomeoInduceIsometry}, we consider those orientation-preserving homeomorphisms that intertwine the~$\varphi_i$:
$$ \Homeo_\varphi^+(Y_0,Y_1)=\lbrace f \in \Homeo^+(Y_0,Y_1) \mid \varphi_1 \circ f_* =\varphi_0 \rbrace.$$
In the case where~$(Y_0,\varphi_0)=(Y_1,\varphi_1)$, we simply write~$\Homeo_\varphi^+(Y_0)$ for the resulting group.

\subsection{Algebraic unions and topological unions}
\label{sub:UnionIsUnion}

In this subsection, we show that under favourable conditions, the~$\Lambda$--intersection form of a union of two~$4$-manifolds can be expressed using the algebraic union of Hermitian forms from Subsection~\ref{sub:Gluing}.
For that however, given two manifolds~$M_0,M_1$ with~$\pi_1(M_i) = \Z$, we need to verify that~$\pi_1(M_0 \cup_f -M_1) = \Z$.
Here and from now on, when $\pi_i(M_i)=\Z$, we take $\varphi_i \colon \pi_1(\partial M_i) \to \pi_1(M_i)=\Z$ to be the map induced by the inclusion.
Thus, when we write $\Homeo^+_{\varphi}(\partial M_0,\partial M_1)$, it is with respect to these inclusion induced maps.
\medbreak

\begin{lemma}
\label{lem:UnionHaspi1Z}
If~$f \in \Homeo^+_{\varphi}(\partial M_0,\partial M_1)$ is a homeomorphism where~$M_0,M_1$ are 4-manifolds with ribbon boundaries and~$\pi_1(M_i)= \Z$, then~$\pi_1(M_0 \cup_f -M_1) = \Z$.
\end{lemma}
\begin{proof}
 For the sake of applying van Kampen's theorem, for~$i=0,1$ we write~$\pi_1(M_i)=\Z \langle t_i \rangle$, so that~$\pi_1(M_0 \cup_f -M_1)$ is generated by~$t_0,t_1$.
Consider the following commutative diagram:
\[
\xymatrix@C+0.2cm{
\pi_1(\partial M_0) \ar[rrrr]^-{f_*,\cong}\ar@{->>}@/^1pc/[rrd]_-{\varphi_0} \ar@{->>}[d]^{\iota_0} &&&& \pi_1(\partial M_1) \ar@{->>}[d]^{\iota_1} \ar@{->>}@/_1pc/[lld]^-{\varphi_1}  \\
\pi_1(M_0) \ar@{=}[r]& \Z\langle t_0 \rangle \ar[r]_{\psi_0,\cong}& \Z\langle t \rangle& \ar[l]^{\psi_1,\cong} \Z\langle t_1 \rangle& \ar@{=}[l] \pi_1(M_1).
}
\]
Here, for $j=0,1$, $\iota_j \colon \pi_1(\partial M_j) \to \pi_1(M_j)$ is the inclusion induced map and $\psi_j \colon t_j \mapsto t$.
In the group~$\pi_1(M_0 \cup_f -M_1)$, the relations identify~$\iota_1f_*(g)$ with~$\iota_0(g)$ for every~$g \in \pi_1(\partial M_0)$.
Since~$\iota_0$ is surjective, there is a~$\mu \in \pi_1(\partial M_0)$ with~$\iota_0(\mu)=t_0$.
We deduce that in~$\pi_1(M_0 \cup_f -M_1)$, the relations identify~$t_0=\varphi_0(\mu)$ with~$\iota_1(f_*(\mu))=t_1^n$ for some~$n$.
Under the identification given by the~$\psi_i$ and using that~$\varphi_0(\mu)=\varphi_1(f_*(\mu))$, we deduce that~$n=1$.
Thus $\pi_1(M_0 \cup_f -M_1)$ is cyclic.
For every~$g \in\pi_1(\partial M_0)$, we have~$\iota_0(g)=t_0^m$ for some~$m$, and the same reasoning as above shows that~$\iota_0(g)=t_0^m$ is identified with~$\iota_1(f_*(g))=t_1^m$ in~$\pi_1(M_0 \cup_f -M_1)$.
It follows that~$\pi_1(M_0 \cup_f -M_1)$ is isomorphic to~$\langle t_0,t_1 \mid t_0=t_1 \rangle\cong \Z$.
\end{proof}

Next, we prove the main result of this section, after setting up some notation.
Continuing with the setting of Lemma~\ref{lem:UnionHaspi1Z}, we write $M:=M_0 \cup_f -M_1$.
Note that $\widehat{\lambda}_M \colon H_2(M;\Lambda) \to H_2(M;\Lambda)^*$ is an isomorphism because $\partial M=\emptyset$; recall the fourth item of Lemma~\ref{lem:Homology}. This endows $H_2(M;\Lambda)^*$ with a Hermitian form, which we shall denote $\lambda_M^{-1}$. 


For $i=0,1$, we set~$\widehat{\lambda}_i:=\widehat{\lambda}_{M_i}$, and as above we write $D_i := D_{M_i}$.
Combining Propositions~\ref{prop:BoundaryLinkingForm3Manifold} and~\ref{prop:HomeoInduceIsometry}, we have an isometry
\[\wt{f}_* \colon \coker(\widehat{\lambda}_0) \xrightarrow{D_0} H_1(\partial M_0;\Lambda) \xrightarrow{f_*} H_1(\partial M_1;\Lambda) \xrightarrow{D_1^{-1}} \coker(\widehat{\lambda}_1),\]
which satisfies that $D^\#(\wt{f}_*) = f_*$.

\begin{proposition}
\label{prop:UnionIsUnion}
Let~$M_0$ and~$M_1$ be two~$4$-manifolds with $\pi_1(M_i) =\Z$ for $i=0.1$, whose boundaries are ribbon and have torsion Alexander modules.
 Let~$f \in \Homeo_{\varphi}^+(\partial M_0,\partial M_1)$ be a homeomorphism, set~$M:=M_0\cup_f-M_1$, and let $\iota_j \colon M_i \to M$ be the inclusion.
The map $\bsm  \iota_0^* \\ \iota_1^* \esm \colon H_2(M;\Lambda)^* \to H_2(M_0;\Lambda)^* \oplus H_2(M_1;\Lambda)^*$ induces an isometry
\begin{equation}
\label{eq:UnionIdentif}
(H_2(M;\Lambda)^*,\lambda_M^{-1}) \xrightarrow{\cong} (H_2(M_0;\Lambda) \cup_{\widetilde{f}_*} H_2(M_1;\Lambda),\lambda_{M_0} \cup_{\widetilde{f}_*} -\lambda_{M_1}).
\end{equation}
Furthermore, the map $\widehat{\lambda}_M$ induces an isometry $(H_2(M;\Lambda),\lambda_M) \to (H_2(M;\Lambda)^*,\lambda_M^{-1})$ which, using the identification from~\eqref{eq:UnionIdentif},
makes the following diagram commute:
\begin{equation}
\label{eq:MapsIntoUnion}
\xymatrix@R0.6cm{
(H_2(M_0;\Lambda),\lambda_{M_0})\ar@{^{(}->}[r]^{\iota_0}\ar@{_{(}->}[rd]^-{\widehat{\lambda}_{M_0}}&(H_2(M;\Lambda),\lambda_{M}) \ar[d]^{\widehat{\lambda}_M}_\cong&(H_2(M_1;\Lambda),-\lambda_{M_1}) \ar@{_{(}->}[l]_{\iota_1}\ar@{^{(}->}[ld]_-{\widehat{\lambda}_{M_1}} \\
&(H_2(M_0;\Lambda) \cup_{\widetilde{f}_*} H_2(M_1;\Lambda),\lambda_{M_0} \cup_{\widetilde{f}_*} -\lambda_{M_1}).&
}
\end{equation}
In particular,~$H_2(M_0;\Lambda) \cup_{\widetilde{f}_*} H_2(M_1;\Lambda)$ is a free~$\Lambda$-module and~$\lambda_{M_0} \cup_{\widetilde{f}_*} -\lambda_{M_1}$ is nonsingular.
\end{proposition}

\begin{proof}
Lemma~\ref{lem:UnionHaspi1Z} implies that~$\pi_1(M)=\Z$.
Consider the following commutative diagram in which the first and third rows are exact and~$\Lambda$ coefficients are understood:
\begin{equation}
\label{eq:BigDiagramUnion}
\xymatrix@C1cm{
0 \ar[r]& H_2(M_0) \oplus H_2(M_1) \ar[r]^-{(\iota_0 \ \iota_1)}\ar[d]^{j_0 \oplus j_1}& H_2(M) \ar[r]\ar[dd]^{\PD^{-1}_M}& H_1(\partial M_1) \ar[r]&0 \\
 & H_2(M_0,\partial M_0) \oplus H_2(M_1,\partial M_1) \ar[d]^{\PD^{-1}_{M_0} \oplus -\PD^{-1}_{M_1}}&  & & \\
H^2(\partial M_1) & H^2(M_0) \oplus H^2(M_1) \ar[l]_-{f^{-*} \incl_0^*-\incl_1^*}\ar[d]_\cong^{\ev \oplus \ev}& H^2(M) \ar[l]_-{\bsm \iota_0^* \\ \iota_1^* \esm }\ar[d]_\cong^{\ev}&\ar[l] 0 & \\
& H_2(M_0)^* \oplus H_2(M_1)^* & H_2(M)^*. \ar[l]_-{\bsm \iota_0^* \\ \iota_1^* \esm }&  &
}
\end{equation}
We justify the three zeros that appear.
In the first row, the rightmost zero comes from the fact that~$H_1(M_i;\Lambda)=0$ for~$i=0,1$ by the second item of Lemma~\ref{lem:Homology}.
For the leftmost zero of the first row, the first and third items of Lemma~\ref{lem:Homology} respectively imply that~$H_2(\partial M_i;\Lambda)=\Z$ is~$\Lambda$-torsion and that~$H_2(M_i;\Lambda)$ is free.
 It follows that the inclusion induced homomorphism~$H_2(\partial M_i;\Lambda) \to H_2(M_i;\Lambda)$ is zero.
The same argument explains the appearance of the zero in the third row:~$H^2(M;\Lambda)$ is free and~$H^1(\partial M_1;\Lambda)$ is~$\Lambda$-torsion.
The commutativity of the middle square follows by applying~\cite[Lemma 8.2]{Bredon} with~$K=M_0$ and~$L=M_1$, as well as applying excision to the pairs~$(M,M \setminus M_i)$.

Our next aim is to simplify the diagram in~\eqref{eq:BigDiagramUnion}.
We write $\pi_i \colon H_2(M_i;\Lambda)^* \to \coker(\widehat{\lambda}_i)$ for the canonical projection.

\begin{claim}
\label{claim:ForNotationUnion}
The isometry  $\widetilde{f}_*\colon\coker(\widehat{\lambda}_0) \xrightarrow{\cong} \coker(\widehat{\lambda}_1)$
 fits in the diagram
\begin{equation}
\xymatrix@C1.5cm{
0 \ar[r]&H_2(M_0;\Lambda) \oplus H_2(M_1;\Lambda)  \ar[r]^-{\bsm \iota_0 & \iota_1 \esm}\ar[d]^{\widehat{\lambda}_0 \oplus -\widehat{\lambda}_1}& H_2(M;\Lambda)\ar[d]_\cong^{\widehat{\lambda}_M} \ar[r]& H_1(\partial M_1;\Lambda) \\
\coker(\widehat{\lambda}_1) &H_2(M_0;\Lambda)^* \oplus H_2(M_1;\Lambda)^* \ar[l]_-{\widetilde{f}_*\pi_0-\pi_1}& H_2(M;\Lambda)^*\ar[l]_-{\bsm
 \iota_0^* \\ \iota_1^*
\esm} &\ar[l]0,
}
\label{eq:DiagramForUnionClaim}
\end{equation}
where the rows are exact and the square is commutative.
\end{claim}
\begin{proof}
The two vertical compositions in \eqref{eq:BigDiagramUnion} are by definition $\widehat{\lambda}_0 \oplus -\widehat{\lambda}_1$ and $\widehat{\lambda}_M$ respectively. Thus the exactness of the top row and the commutativity of the square follow from~\eqref{eq:BigDiagramUnion}. It therefore remains to establish the exactness of the bottom row in~\eqref{eq:DiagramForUnionClaim}.
A second look at~\eqref{eq:BigDiagramUnion} shows a similar exact sequence with cohomology instead of duals and $H^2(\partial M_1;\Lambda)$ instead of $\coker(\widehat{\lambda}_1)$.
We shall produce the following commuting square that can be appended to the bottom left of \eqref{eq:BigDiagramUnion}. The claim will then follow immediately.
\begin{equation}
  \xymatrix @C+1.2cm {H^2(\partial M_1;\Lambda) \ar[d]^{G_1}_{\cong}  & H^2(M_0;\Lambda) \oplus H^2(M_1;\Lambda) \ar[l]_-{f^{-*} \incl_0^*-\incl_1^*}\ar[d]_\cong^{\ev \oplus \ev} \\
\coker(\widehat{\lambda}_1) &H_2(M_0;\Lambda)^* \oplus H_2(M_1;\Lambda)^*. \ar[l]_-{\widetilde{f}_*\pi_0-\pi_1}
}
\label{eq:new-bottom-left-square}
\end{equation}
We define isomorphisms $G_i$, for $i=0,1$, as the composition:
\[G_i := D_i^{-1} \circ \PD \colon H^2(\partial M_i;\Lambda) \xrightarrow{\cong} H_1(\partial M_i;\Lambda) \xrightarrow{\cong} \coker(\widehat{\lambda}_i).\]
The map $G_i$ agrees with the rightmost vertical downwards composition (inverting $D_i$) in the next diagram.  Again $\Lambda$ coefficients are understood, and the rows are exact:
\begin{equation}
\xymatrix{
0\ar[r]&H^2(M_i,\partial M_i) \ar[r]^-{j_i^*} \ar[d]^{\cong}_{PD} & H^2(M_i) \ar[r]^-{\incl_i^*} \ar[d]_{\PD}^{\cong} & H^2(\partial M_i) \ar[r] \ar[d]_{\PD}^{\cong}  & 0 \\
0\ar[r]&H_2(M_i) \ar[r]^-{j_i}   & H_2(M_i,\partial M_i) \ar[r]^{\delta_i}  & H_1(\partial M_i)  \ar[r] & 0\\
0\ar[r]&H_2(M_i) \ar[r]^-{\widehat{\lambda}_i} \ar[u]_{\Id}^= & H_2(M_i)^* \ar[r]^-{\pi_i} \ar[u]_{\PD \circ \ev^{-1}}^{\cong} & \coker(\widehat{\lambda}_i) \ar[r] \ar[u]_{D_i}^{\cong} & 0.
}
\label{eq:diagram-for-sans-basis-1}
\end{equation}
The diagram commutes by the definition of $\widehat{\lambda}_i$ as $\ev \circ PD^{-1} \circ j_i$, and using standard naturality of Poincar\'{e} duality; for the bottom right square, one uses the definition of $D_i$ from Remark~\ref{rem:Bases}.
Via $G_i$ and $\ev$, the concatenation of the two squares on the right identifies the maps $\pi_i$ and $\incl^*_i$. That is, for $i=0,1$:
\begin{equation}\label{eqn:G_i-ev-identify-maps}
  G_i \circ \incl^*_i = \pi_i \circ \ev \colon H^2(M_i;\Lambda) \to \coker(\widehat{\lambda}_i).
\end{equation}
This shows that the bottom left square in the following diagram commutes. This diagram again has $\Lambda$ coefficients and vertical maps isomorphisms:
\begin{equation}
\xymatrix{
& H_1(\partial M_0;\Lambda) \ar[r]^-{f_*}_-{\cong} \ar[d]^{\PD^{-1}}_{\cong} & H_1(\partial M_1;\Lambda) \ar[d]^{\PD^{-1}}_{\cong} \\
H^2(M_0) \ar[r]^-{\incl_0^*} \ar[d]_{\ev}^{\cong} & H^2(\partial M_0) \ar[r]^-{f^{-*}}_-{\cong} \ar[d]_{G_0}^{\cong}  & H^2(\partial M_1) \ar[d]^{\cong}_{G_1} \\
H_2(M_0)^* \ar[r]^-{\pi_0} 
 & \coker(\widehat{\lambda}_0) \ar[r]^-{\wt{f}_*}_{\cong} 
& \coker(\widehat{\lambda}_1).
}
\label{eq:diagram-for-sans-basis-2}
\end{equation}
Commutativity of the top right square follows from naturality of Poincar\'{e} duality and the fact that all four maps are isomorphisms.
Using this and that $D_i = \PD \circ G_i^{-1}$ for $i=0,1$, starting with the definition of $\wt{f}_* \colon \coker(\widehat{\lambda}_0) \to \coker(\widehat{\lambda}_1)$  we obtain:
\[\wt{f}_* = D_1^{-1} \circ f_* \circ D_0 = G_1 \circ \PD^{-1} \circ f_* \circ \PD \circ G_0^{-1} = G_1 \circ f^{-*} \circ G_0^{-1}.\]
In other words, the bottom right square of \eqref{eq:diagram-for-sans-basis-2} commutes.

%
%
Therefore the bottom of diagram \eqref{eq:diagram-for-sans-basis-2}
yields $G_1\circ f^{-*}\circ \incl_0^*=\widetilde{f}_*\circ\pi_0\circ\ev$.
 Subtracting \eqref{eqn:G_i-ev-identify-maps} with $i=1$, we obtain the desired commutative square~\eqref{eq:new-bottom-left-square}.
The diagram in~\eqref{eq:DiagramForUnionClaim} then follows from the combination of~\eqref{eq:BigDiagramUnion} and~\eqref{eq:new-bottom-left-square}, completing the proof of the claim.
\end{proof}

Now we write $T:=H_1(\partial M_1;\Lambda)$, as well as $H_i:=H_2(M_i;\Lambda)$ for $i=0,1$.
We also set~$V:=H_2(M;\Lambda)^*$ and identify~$H_2(M;\Lambda)$ with the double dual~$H_2(M;\Lambda)^{**}=V^{*}$ via the evaluation isomorphism~$\ev$.
With these notations, the diagram from Claim~\ref{claim:ForNotationUnion} leads to the following commutative diagram in which both rows are exact:
\begin{equation}
\label{eq:DiagramForUnion}
\xymatrix@C1.5cm{
0 \ar[r]&H_0 \oplus H_1  \ar[r]^-{\bsm \ev & \ev \esm}\ar[d]^{\widehat{\lambda}_0 \oplus -\widehat{\lambda}_1}& V^{*}\ar[d]_\cong^{\widehat{\lambda}_M} \ar[r]& T \\
\coker(\widehat{\lambda}_1) &H_0^* \oplus H_1^* \ar[l]_-{\widetilde{f}_*\pi_0-\pi_1}& V\ar[l] &\ar[l] 0.
}
\end{equation}
Identify~$V=H_2(M;\Lambda)^*$ with~$\ker(\widetilde{f}_* \pi_0-\pi_1)$ and view it as a subspace of~$H_0^*\oplus H_1^*$.
The diagram in~\eqref{eq:DiagramForUnion} therefore yields the equation
\begin{equation}
\label{eq:NearlyUnionIsUnion}
\begin{pmatrix}
\widehat{\lambda}_0 & 0 \\
0 &-\widehat{\lambda}_1
\end{pmatrix}= \widehat{\lambda}_M \circ  \begin{pmatrix}
 \ev & \ev
\end{pmatrix}.
\end{equation}
Since Lemma~\ref{lem:Homology} implies that~$H_2(M;\Lambda)$ is a free~$\Lambda$-module and~$H_1(M;\Lambda)=0$, we may identify~$H^2(M;Q)$ with~$H^2(M;\Lambda) \otimes_\Lambda Q$, and similarly for~$H^2(M_i,\partial M_i;Q) \cong H^2(M_i,\partial M_i;\Lambda) \otimes_\Lambda Q$ and~$H^2(M_i;Q) \cong H^2(M_i;\Lambda) \otimes_\Lambda Q$ for~$i=0,1$.
As a consequence, we may identify the tensored up intersection pairings~$\lambda_M \otimes_\Lambda \id_Q$ and~$\lambda_i \otimes_\Lambda \id_Q$ with the nonsingular~$Q$-valued intersections pairings on~$H_2(M;Q)$ and~$H_2(M_i;Q)$.
Since~$\widehat{\lambda}_M$ is invertible by Lemma~\ref{lem:Homology}, we deduce from~\eqref{eq:NearlyUnionIsUnion} that
\begin{equation*}
\widehat{\lambda}_M^{-1}
=  \begin{pmatrix}
 \ev & \ev
\end{pmatrix} \circ \begin{pmatrix}
\widehat{\lambda}_{0,Q}^{-1} & 0 \\
0 &-\widehat{\lambda}_{1,Q}^{-1}
\end{pmatrix}.
\end{equation*}
As noted in Remark~\ref{rem:AdjointOfUnion}, this is precisely the adjoint of ~$\lambda_0 \cup_{\widetilde{f}_*} -\lambda_1$.

We have thus proved that the inclusions $\iota_0$ and $\iota
_1$ induce an isometry~$(H_2(M;\Lambda)^*,\lambda_M^{-1}) \cong (H_0 \cup_{\widetilde{f}_*} H_1,\lambda_0 \cup_{\widetilde{f}_*} -\lambda_1)$.
Now, since~$\lambda_M^*=\lambda_M$, a quick verification shows that the isomorphism~$\widehat{\lambda}_M \colon H_2(M;\Lambda) \to H_2(M;\Lambda)^*$ induces an isometry between the forms~$(H_2(M;\Lambda),\lambda_M)$ and~$(H_2(M;\Lambda)^*,\lambda_M^{-1})$.
Indeed, writing~$H:=H_2(M;\Lambda)$, this follows from the following commutative diagram:
$$
\xymatrix{
H \ar[r]^{\widehat{\lambda}_M}_\cong \ar[d]^{\widehat{\lambda}_M}& H^*  \ar[d]^{\widehat{\lambda}_M^{-1}} \\
H^* & H. \ar[l]_\cong^{\widehat{\lambda}_M^*=\widehat{\lambda}_M}
}
$$
The commutativity of the diagram in~\eqref{eq:MapsIntoUnion} follows from the commutativity of~\eqref{eq:DiagramForUnion}: we have shown that~$(H_0 \cup_{\widetilde{f}_*} H_1,\lambda_0 \cup_{\widetilde{f}_*} -\lambda_1) \cong (H_2(M;\Lambda),\lambda_M)$ and, recalling that~$(\ev,\ev)=(\iota_0,\iota_1)$, the diagram in~\eqref{eq:DiagramForUnion} gives~$\widehat{\lambda}_{M}\circ \iota_0(x_0,0)=\widehat{\lambda}_{0}(x_0,0)$ leading to the commutativity of the left hand triangle of~\eqref{eq:MapsIntoUnion}; the reasoning for the right hand triangle is identical.
\end{proof}

\subsection{Compatible pairs}
\label{sub:Compatible}

Throughout this section,~$M_0$ and $M_1$ are~$4$-manifolds with $\pi_1(M_i)=\Z$ for $i=0,1$, whose boundaries are ribbon and have torsion Alexander modules.
This section introduces a notion of compatibility for an isometry~$F \colon H_2(M_0;\Lambda) \to H_2(M_1;\Lambda)$ and a homeomorphism~$f \colon \partial M_0 \to \partial M_1$.
We then prove Theorem~\ref{thm:UnionIsS1S3ConnectSumSimplyConnectedIntro} from the introduction.
\medbreak

 Since the boundaries are ribbon, and we have identifications $\pi_1(M_i)=\Z$, the inclusions~$\partial M_i \hookrightarrow M_i$ induce surjections~$\varphi_i \colon \pi_1(M_i) \twoheadrightarrow \Z$ for~$i=0,1$.
Recall from Subsection~\ref{sub:HomeoIso} that we set
$$ \Homeo_{\varphi}^+(\partial M_0,\partial M_1)=\lbrace f \in \Homeo^+(\partial M_0,\partial M_1) \ | \ \varphi_1 \circ f_*=\varphi_0 \rbrace.$$
Proposition~\ref{prop:HomeoInduceIsometry} ensures that a homeomorphism~$f \in \Homeo_{\varphi}^+(\partial M_0,\partial M_1)$ induces an isometry~$f_*\in \Iso (\Bl_{\partial M_0},\Bl_{\partial M_1})$ of the Blanchfield forms.
On the other hand, recall from Lemma~\ref{lem:BoundaryIsometry} that an isometry~$F \in \Iso(\lambda_{M_0},\lambda_{M_1})$ of the~$\Lambda$--intersection forms of~$M_0$ and~$M_1$ induces an isometry~$\partial F=F^{-*}$ of the boundary linking forms.
Using Proposition~\ref{prop:BoundaryLinkingForm3Manifold}, we identify these boundary linking forms with minus the corresponding Blanchfield forms, so that using Remark~\ref{rem:MinusAndIdentifIso} we may consider~$D^\#(\partial F) \in \Iso (\Bl_{\partial M_0},\Bl_{\partial M_1})$.

\begin{definition}
\label{def:Compatible}
An orientation-preserving homeomorphism~$f \in \Homeo_{\varphi}^+(\partial M_0,\partial M_1)$ is \emph{compatible} with an isometry~$F \in \Iso(\lambda_{M_0},\lambda_{M_1})$ if, using the identification $D^\#$ from Remark~\ref{rem:MinusAndIdentifIso}, we have:
$$D^\#(\partial F)=f_* \colon (H_1(\partial M_0;\Lambda), \Bl_{\partial M_0})\to (H_1(\partial M_1;\Lambda), \Bl_{\partial M_1}).$$
In this case, the pair~$(f,F)$ is called a \emph{compatible pair}.
\end{definition}

The next proposition shows that the existence of a compatible pair is a necessary condition for a homeomorphism~$f \in \Homeo^+_{\varphi}(\partial M_0,\partial M_1)$ to extend to a homeomorphism~$M_0 \to~M_1$.

\begin{proposition}
\label{prop:necessary}
Let~$f \in \Homeo_{\varphi}^+(\partial M_0,\partial M_1)$ be a homeomorphism, and let~$F \in \Iso(\lambda_{M_0},\lambda_{M_1})$ be an isometry.
If~$f$ extends to an orientation-preserving homeomorphism~$\Phi \colon M_0 \xrightarrow{\cong} M_1$ that induces~$F$, then~$(f,F)$ is a compatible pair.
\end{proposition}

\begin{proof}
Write $\lambda_i := \lambda_{M_i}$ and $D_{i} := D_{M_i}$ for $i=0,1$.
The homeomorphism~$\Phi$ induces an isomorphism~$\Phi_* \colon \pi_1(M_0) \to \pi_1(M_1)$. By the assumption on $f$, the isomorphisms~$\pi_1(M_i) = \Z$ are such that~$\Phi_*$ intertwines them. This follows because in the next diagram both triangles, the top square, and the large outside square commute, and it then follows that the bottom square commutes too.
\[\xymatrix @R-0.3cm{\pi_1(\partial M_0) \ar[r]^-{f_*} \ar[d] \ar@/_2pc/[dd]_{\varphi_0} & \pi_1(\partial M_1) \ar[d] \ar@/^2pc/[dd]^{\varphi_1}  \\  \pi_1(M_0) \ar[r]^-{\Phi_*} \ar@{=}[d] & \pi_1(M_1) \ar@{=}[d]  \\ \Z \ar[r]^{=} & \Z.}\]

By assumption $\Phi$ induces the maps $F \colon H_2(M_0;\Lambda) \xrightarrow{\cong} H_2(M_1;\Lambda)$ and $F^{-*} \colon H_2(M_0;\Lambda)^* \xrightarrow{\cong} H_2(M_1;\Lambda)^*$.
By Lemma~\ref{lem:BoundaryIsometry} the map $F^{-*}$ induces an isometry
\[\partial F \colon (\coker(\widehat{\lambda}_0),\partial \lambda_0) \xrightarrow{\cong} (\coker(\widehat{\lambda}_1),\partial \lambda_1).\]
We assert that $\partial F$ fits into a commuting diagram of isometries:
\begin{equation}\label{eq:diagram-identifying-two-isometries}
  \xymatrix{\coker(\widehat{\lambda}_0) \ar[r]^-{\partial F}_-{\cong} \ar[d]^{D_0}_-{\cong} & \coker(\widehat{\lambda}_1) \ar[d]^{D_1}_-{\cong} \\
  H_1(\partial M_0;\Lambda) \ar[r]^{f_*}_-{\cong} & H_1(\partial M_1;\Lambda). }
\end{equation}
The maps $D_i$ are isometries by Proposition~\ref{prop:BoundaryLinkingForm3Manifold}, and the map $f_*$ in the bottom row is an isometry by Proposition~\ref{prop:HomeoInduceIsometry}.
By definition of $D^\#$, commutativity of the diagram~\eqref{eq:diagram-identifying-two-isometries} is equivalent to $D^{\#}(\partial F) = f_*$, i.e. to the fact that $(f,F)$ is a compatible pair.
To show that~\eqref{eq:diagram-identifying-two-isometries} commutes, we use the next diagram, in which all homology groups have $\Lambda$ coefficients and the rightmost horizontal maps are all surjections. The map $\partial F \colon \coker(\widehat{\lambda}_0) \to \coker(\widehat{\lambda}_1)$ is by definition the map induced by the upper part of the diagram.
\[\begin{tikzcd}[column sep={6em,between origins},row sep=3em]
&
H_2(M_0)
  \arrow[rr,"\widehat{\lambda}_0"]
  \arrow[dl,"F_*"',"\cong"]
  \arrow[dd,"\Id" near start,"="' near start]
&&
H_2(M_0)^*
  \arrow[dd,"\PD \circ \ev^{-1}" near start,"\cong"' near start]
  \arrow[dl,"F^{-*}"',"\cong"]
  \arrow[rr,twoheadrightarrow]
&&
\coker(\widehat{\lambda}_0)
\arrow[dl,"\partial F"',"\cong"]
\arrow[dd,"D_0","\cong"']
\\
H_2(M_1)
  \arrow[rr,"\widehat{\lambda}_1" near end,crossing over]
  \arrow[dd,"\Id" near start,"="' near start]
&&
H_2(M_1)^*
\arrow[rr,crossing over,twoheadrightarrow]
&&
\coker(\widehat{\lambda}_1)
&
\\
&
H_2(M_0)
  \arrow[rr,"j_0" near start]
  \arrow[dl,"F_*"',"\cong"]
&&
H_2(M_0,\partial M_0)
  \arrow[dl,"F_*"',"\cong"]
  \arrow[rr,twoheadrightarrow]
&&
H_1(\partial M_0).
\arrow[dl,"f_*"',"\cong"]
\\
H_2(M_1)
  \arrow[rr,"j_1"]
&&
H_2(M_1,\partial M_1)
\arrow[rr,twoheadrightarrow]
\arrow[from=uu,crossing over,"\PD \circ \ev^{-1}" near start,"\cong"' near start]
&&
H_1(\partial M_1)
\arrow[from=uu,crossing over, "D_1" near start,"\cong"' near start]
&
\end{tikzcd}\]
The squares in the left-most cube commute either trivially, by the definition of $\widehat{\lambda}_i$, or by naturality of Poincar\'{e} duality.  The potentially contentious point is the latter justification: naturality in fact says that $F_* \circ PD_{M_0} \circ F^* = PD_{M_1}$. But since $F^*$ is an isomorphism, with inverse $F^{-*}$, it follows that the middle vertical square commutes as shown. A straightforward diagram chase now shows as desired that the rightmost vertical square, which is equivalent to~\eqref{eq:diagram-identifying-two-isometries}, also commutes. This completes the proof that $(f,F)$ is a compatible pair.
\end{proof}


The next result shows that the existence of a compatible pair~$(f,F)$ imposes strong restrictions on the topology of~$M:= M_0 \cup_f -M_1$.
This is Theorem~\ref{thm:UnionIsS1S3ConnectSumSimplyConnectedIntro} from the introduction.

\begin{theorem}
\label{thm:UnionIsS1S3ConnectSumSimplyConnected}
Let~$(f,F)$ be a compatible pair. If~$M_0$ and~$M_1$ are not spin, assume that the Kirby-Siebenmann invariants satisfy~$\ks(M_0)=\ks(M_1) \in~\Z/2$.
Then there is an orientation-preserving homeomorphism
$$M:= M_0 \cup_f -M_1 \cong S^1 \times S^3 \#_{i=1}^{a} S^2 \times S^2 \#_{j=1}^b S^2 \wt{\times} S^2$$
for some~$a$,~$b$ with~$a+b=b_2(M_0)$. If~$M_0$ and~$M_1$ are spin then~$b=0$.

Moreover it can be assumed that the~$\Lambda$-isometry induced by this homeomorphism takes the Lagrangian~$\widehat{\lambda}_M^{-1}(\Gamma_{F^{-*}}) \subseteq H_2(M;\Lambda)$ to the Lagrangian generated by~$\{ [\{ \pt \} \times S^2 ]  \}_{i=1}^{b_2(M_0)}$.
\end{theorem}

\begin{proof}
Lemma~\ref{lem:UnionHaspi1Z} establishes that~$\pi_1(M)=\Z$.
Proposition~\ref{prop:UnionIsUnion} proves that~$\lambda_M$ is isometric to~$\lambda_{M_0} \cup_{\widetilde{f}_*} -\lambda_{M_1}.$
We have~$\lambda_{M_0} \cup_{\widetilde{f}_*} -\lambda_{M_1}=\lambda_{M_0} \cup_{\partial F} -\lambda_{M_1}$ because the pair~$(f,F)$ is compatible.
It follows from Proposition~\ref{prop:IsElementary} that~$\lambda_{M_0} \cup_{\partial F} -\lambda_{M_1}$ is metabolic with~$\Gamma_{F^{-*}}$ as a Lagrangian.
We deduce that~$M$ is a closed~$4$-manifold with~$\pi_1(M)=\Z$ and metabolic~$\Lambda$--intersection form.
The (nonsingular) intersection form of $M$ (which can be obtained from $\lambda_M$ by setting $t=1$) is therefore isometric to
\[\begin{pmatrix}
  0 & 1 \\ 1 & 0
\end{pmatrix}^{\oplus a} \oplus \begin{pmatrix}
  0 & 1 \\ 1 & 1
\end{pmatrix}^{\oplus b}\]
for some~$a,b$.
As~$\ks(M_0)=\ks(M_1)$, by additivity of the Kirby-Siebenmann invariant~\cite[Theorem~8.2]{FriedlNagelOrsonPowell} under gluing it follows that~$\ks(M)=0$.
Apart from the statement that~$b=0$ for~$M_i$ spin, the theorem now follows from the classification of closed~$4$-manifolds with infinite cyclic fundamental group due to~\cite[Theorem~10.7A~(2)]{FreedmanQuinn} and~\cite{Stong-Wang-2000}: every isometry between the~$\Lambda$-intersection forms of two closed, oriented~$4$-manifolds with fundamental group~$\Z$ can be realised by a homeomorphism.

Now assume that~$M_0$ and~$M_1$ are spin.
To show that~$b=0$, it is sufficient to show that~$\lambda_M$ is hyperbolic, as the result will  then once again follow from~\cite[Theorem~10.7A~(2)]{FreedmanQuinn} and~\cite{Stong-Wang-2000}.
We already argued that~$\lambda_M$ is isometric to~$\lambda_{M_0} \cup_{\partial F} -\lambda_{M_1}$ because the pair~$(f,F)$ is compatible.
Applying Proposition~\ref{prop:IsometriesOfUnion} to the isometry~$h=\partial F$, this latter form is isometric to~$\lambda_{M_1} \cup_{\id} -\lambda_{M_1}$.
The identity certainly belongs to~$\Homeo_\varphi^+(\partial M_1,\partial M_1)$ and so a second application of Proposition~\ref{prop:UnionIsUnion} ensures that
$\lambda_{M_1} \cup_{\id} -\lambda_{M_1}$ is isometric to~$\lambda_{M_1 \cup_{\id} -M_1}$.
Summarising, we have the following sequence of isometries:
\begin{equation}
\label{eq:SpinUnion}
\lambda_M \cong \lambda_{M_0 \cup_{\widetilde{f}_*} -M_1} \cong \lambda_{M_0} \cup_{\partial F} -\lambda_{M_1} \cong \lambda_{M_1} \cup_{\id} -\lambda_{M_1} \cong \lambda_{M_1 \cup_{\id} -M_1}.
 \end{equation}
Since~$M_1$ is spin, the double~$DM_1:=M_1 \cup_{\id} -M_1$ is also spin.
As~$\pi_2(DM_1)=H_2(DM_1;\Lambda)$ is free (by the third item of Lemma~\ref{lem:Homology} and Lemma~\ref{lem:UnionHaspi1Z})  and~$DM_1$ is spin, it follows from~\cite[Remark 1.7]{KasprowskiPowellTeichner} that~$\lambda_{DM_1}$ is even, meaning that $\lambda_{DM_1}=q+\overline{q}$ for some sesquilinear form $q$. Since $\lambda_{DM_1}$ and $\lambda_M$ are isometric by~\eqref{eq:SpinUnion}, it follows that $\lambda_M$ is even.
Now since $\lambda_M$ is an even metabolic nonsingular Hermitian form on a finitely generated free $\Lambda$-module, it is hyperbolic by~\cite[Corollary 3.7.3]{KnusQuadratic} with~$\Gamma_{F^{-*}}$ as a Lagrangian.
\end{proof}

We make the Lagrangian~$\widehat{\lambda}_M^{-1}(\Gamma_{F^{-*}}) \subseteq H_2(M;\Lambda)$ somewhat more concrete.
In the setting of Theorem~\ref{thm:UnionIsS1S3ConnectSumSimplyConnected}, combine the inclusion induced maps as
$$ \iota =
\begin{pmatrix}  \iota_0 & \iota_1\end{pmatrix}
\colon H_2(M_0;\Lambda) \oplus H_2(M_1;\Lambda) \to H_2(M;\Lambda).~$$
We record the following fact; compare with~\cite[Proposition~4.2]{Boyer} in the simply connected case.

\begin{lemma}
\label{rem:ContainsGraph}
The Lagrangian ~$\widehat{\lambda}_M^{-1}(\Gamma_{F^{-*}}) \subseteq H_2(M;\Lambda)$ contains the graph~$\Gamma_{-F}$ of the isometry~$-F \colon H_2(M_0;\Lambda) \to H_2(M_1;\Lambda)$.
\end{lemma}

\begin{proof}
Given~$x_0 \in H_2(M_0;\Lambda)$, we have
$$ \iota \begin{pmatrix} x_0 \\ -F(x_0) \end{pmatrix}
 =\widehat{\lambda}_M^{-1}\begin{pmatrix}
 \widehat{\lambda}_{M_0}(x_0) \\
\widehat{\lambda}_{M_1}(F(x_0)) \end{pmatrix}
 =\widehat{\lambda}_M^{-1}\begin{pmatrix}
 \widehat{\lambda}_0(x_0)  \\
F^{-*}(\widehat{\lambda}_{M_0}(x_0))
 \end{pmatrix}.
~$$
Here, the first equality holds by the commutativity of the diagram in~\eqref{eq:DiagramForUnion}, while the second holds because~$F$ is an isometry.
It follows that~$\Gamma_{-F} \subseteq \widehat{\lambda}_M^{-1}(\Gamma_{F^{-*}})$ as desired.
\end{proof}

We introduce some notation needed to provide a sufficient criterion to produce compatible pairs.
Under the identification~$\Aut (\partial \lambda_{M_1})=\Aut (\Bl_{\partial M_1})$ as in Remark~\ref{rem:MinusAndIdentifIso},
there is a left action
\begin{align*}
\left( \Homeo_{\varphi}^+(\partial M_1) \times \Aut(\lambda_{M_1}) \right) \times \Aut (\Bl_{\partial M_1}) &\to \Aut (\Bl_{\partial M_1}) \\
(f,F) \cdot h &:=f_* \circ h \circ \partial F^{-1}.
\end{align*}
The next proposition gives a criterion to find a compatible pair.

\begin{proposition}
\label{prop:FindCompatible}
For $M_0$ and $M_1$ as above, if there exist~$f \in \Homeo_{\varphi}^+(\partial M_0,\partial M_1)$ and~$F \in \Iso(\lambda_{M_0},\lambda_{M_1})$, and if the orbit set
~$$\Aut(\Bl_{\partial M_1})/\Homeo^+_{\varphi}(\partial M_1) \times \Aut(\lambda_{M_1})$$
 is trivial, then a compatible pair~$(f',F')$ exists.
\end{proposition}
\begin{proof}
The composition~$f_* \circ \partial F^{-1}$ is an isometry of~$\partial \lambda_{M_1} \cong \Bl_{\partial M_1}$.
Using the assumption, there exists a homeomorphism~$g \in \Homeo^+_{\varphi}(\partial M_1)$ and an isometry~$G \in \Aut(\lambda_{M_1})$ such that we have~$f_* \circ \partial F^{-1} =g_*^{-1} \circ \partial G$.
Equivalently, we can write~$(\widetilde{g \circ f})_*=\partial (G \circ F)$ and consequently the pair~$(f',F'):=(g \circ f,G \circ F)$ is compatible.
\end{proof}

\section{Partial classification of compact~$4$-manifolds with fundamental group~$\Z$}
\label{sec:Proof}

In this section, we prove Theorem~\ref{thm:Boyer} from the introduction.  The proof of this result was inspired by \cite{Boyer} and \cite{KreckSurgeryAndDuality}.

\begin{theorem}
\label{thm:MainTechnical}
Let~$M_0$ and~$M_1$ be two~$4$-manifolds with $\pi_1(M_i)=\Z$ for $i=0,1$, whose boundaries are ribbon and have torsion Alexander modules.
Let~$f \in \Homeo_{\varphi}^+(\partial M_0,\partial M_1)$ be an orientation-preserving homeomorphism and let~$F \in \Iso(\lambda_{M_0},\lambda_{M_1})$ be an isometry.
If~$M_0$ and~$M_1$ are not spin, assume that the Kirby-Siebenmann invariants satisfy~$\ks(M_0)=\ks(M_1) \in \Z/2$.
Then the following assertions are equivalent:
\begin{enumerate}
\item the pair~$(f,F)$ is compatible;
\item the homeomorphism~$f$ extends to an orientation-preserving homeomorphism
$$ \Phi \colon M_0 \xrightarrow{\cong} M_1$$
inducing the given isometry~$F \colon H_2(M_0;\Lambda) \cong H_2(M_1;\Lambda)$.
\end{enumerate}
\end{theorem}

\begin{proof}
Proposition~\ref{prop:necessary} shows that~$(2)$ implies~$(1)$.
To prove the converse, given a compatible pair~$(f,F)$, we must therefore establish the existence of a homeomorphism~$\Phi$.
Under this assumption, and in the non-spin case assuming that the Kirby-Siebenmann invariants coincide, we saw in Theorem~\ref{thm:UnionIsS1S3ConnectSumSimplyConnected} that for some~$a,b$ such that~$a+b=b_2(M_0)$, there is a homeomorphism
$$M_0 \cup_f -M_1 \cong S^1 \times S^3 \#_{i=1}^{a} S^2 \times S^2 \#_{i=1}^b S^2 \widetilde{\times} S^2.$$
Define~$M:=M_0 \cup_f -M_1$.
With this notation, Theorem~\ref{thm:UnionIsS1S3ConnectSumSimplyConnected} additionally states that this homeomorphism takes the Lagrangian~$\widehat{\lambda}^{-1}_M(\Gamma_{F^{-*}}) \subseteq H_2(M;\Lambda)$ to the Lagrangian generated by the spheres~$\{[  \{\pt\} \times S^2] \}_{i=1}^{b_2(M_0)}$.
Consider the following~$5$-dimensional null-bordism for~$M$
\[W:=S^1 \times D^4 \natural_{i=1}^{a} S^2 \times D^3 \natural_{i=1}^{b} S^2 \widetilde{\times} D^3,\]
where~$S^2 \widetilde{\times} D^3$ denotes the twisted linear~$D^3$ bundle over~$S^2$.
We prove that~$W$ is an~$h$-cobordism; it will then automatically be an~$s$-cobordism as~$\operatorname{Wh}(\Z)=0$.
As~$H_i(W,M_0;\Lambda)=0$ for~$i \neq 2,3$, we must show that~$H_2(W,M_0;\Lambda)=0$ for~$i=2,3$.

Use~$D^3_{\frac{1}{2}} \subseteq D^3$ to denote a~$3$-ball of smaller radius inside~$D^3$.
Choosing the obvious set of generators for~$H_2(W;\Lambda)$ and representing them by embedded spheres, we obtain
$$U:= \natural_{i=1}^{a} S^2 \times D^3_{\frac{1}{2}} \natural_{i=1}^{b} S^2 \widetilde{\times} D^3_{\frac{1}{2}} \subseteq W.$$
Note that~$H_*(U;\Lambda)=H_*(W;\Lambda)$ and~$H_*(\partial U;\Lambda)= H_*(M;\Lambda)$.
In particular, under this identification, we have
\begin{equation}
\label{eq:TheFisTheGraph}
\ker(H_2(M;\Lambda) \xrightarrow{j} H_2(W;\Lambda)) = \im (H_3(U,\partial U;\Lambda) \to H_2(\partial U;\Lambda)) = \widehat{\lambda}_M^{-1}(\Gamma_{F^{-*}}).
\end{equation}
In what follows, we identify $H_3(U,\partial U;\Lambda)$ with its image in $H_2(\partial U;\Lambda)$ allowing ourselves for instance to write $H_3(U,\partial U;\Lambda) \subseteq H_2(\partial U;\Lambda)$.
We also think of the adjoint~$\widehat{\lambda}_M$ of the intersection form on~$M$ as having~$H_2(\partial U;\Lambda)$ as its domain.
\begin{claim}
\label{claim:quasi-formationPrelude}
The connecting homomorphism
\[\partial_0 \colon H_3(W\setminus \mathring{U},M_0 \sqcup \partial U;\Lambda) \to H_2(M_0;\Lambda),\]
arising from the long exact sequence of the triple~$(W \sm \mathring{U},M_0 \sqcup \partial U,\partial U)$,
 is an isomorphism.
\end{claim}
\begin{proof}[Proof of Claim~\ref{claim:quasi-formationPrelude}]
It suffices to note that by excision and by our choice of~$U$ and~$W$, we have~$H_i(W \setminus \mathring{U},\partial U;\Lambda) \cong H_i(W,U;\Lambda)=0$ for~$i=2,3$.
\end{proof}

 \begin{claim}
 \label{claim:quasi-formation}
The following diagram commutes:
~$$
 \xymatrix{
H_2(M_0;\Lambda) \ar@{^{(}->}[r]^-{\widehat{\lambda}_0} & ((H_2(M_0;\Lambda)\cup_{\partial F} H_2(M_1;\Lambda),\lambda_0 \cup_{\partial F}-\lambda_1)  & \Gamma_{F^{-*}}\ar@{_{(}->}[l] \\
 H_3(W\setminus \mathring{U},M_0 \sqcup \partial U;\Lambda)  \ar[u]_{\partial_0}^\cong \ar@{^{(}->}[r] & (H_2(\partial U;\Lambda),\lambda_{\partial U}) \ar[u]^\cong_{\widehat{\lambda}_M}  & H_3(U,\partial U;\Lambda). \ar@{_{(}->}[l]\ar[u]^\cong_{\widehat{\lambda}_M}
 }$$
\end{claim}
\begin{proof}[Proof of Claim~\ref{claim:quasi-formation}]
The commutativity of the right square follows from~\eqref{eq:TheFisTheGraph}, and so we focus on the left square.
Consider the long exact sequence of the triples~$(W \setminus \mathring{U},M_0 \sqcup \partial U,\partial U)$ and~$(W \setminus \mathring{U},M_0 \sqcup \partial U,M_0)$.
The portions of interest can be seen in the two horizontal rows of the following diagram, where~$\Lambda$ coefficients are understood:
\begin{equation}
\label{eq:KeyDiagram}
 \xymatrix@R0.3cm@C0.3cm{
 H_3(W \setminus \mathring{U},\partial U)=0 \ar[r]& H_3(W \setminus \mathring{U},M_0 \sqcup \partial U)  \ar[r]^-{\partial_0}& H_2(M_0) \ar[d]^{\iota_0} \ar[rr] \ar@{^{(}->}[dr]^{\widehat{\lambda}_0}&& *+[l]{0=H_2(W \setminus \mathring{U},\partial U)}  \\
 &&H_2(M)\ar[r]^(.4){\widehat{\lambda}_M,\cong}&*+[r]{H_2(M_0) \cup_{\partial F} H_2(M_1)}&\\
H_3(W \setminus \mathring{U},M_0)=0 \ar[r]& H_3(W \setminus \mathring{U},M_0 \sqcup \partial U)  \ar[r]\ar[uu]^=& H_2(\partial U) \ar[rr]\ar[ru]_-{\widehat{\lambda}_{M},\cong} \ar[u]^=&& *+[l]{H_2(W \setminus \mathring{U},M_0).}}
\end{equation}
The claim will follow from the central portion of the diagram, once we explain all its features and its commutativity.
The zeros in the first row are a consequence of Claim~\ref{claim:quasi-formationPrelude}.
The bottom leftmost zero is stated in Kreck's work~\cite[page 734]{KreckSurgeryAndDuality} but we outline a proof in this setting.
The exact sequence of the triple~$(W \setminus \mathring{U},M_1 \sqcup \partial U,M_1)$,
together with a Mayer-Vietoris argument, give rise to an isomorphism~$H_1(W\setminus \mathring{U},M_1 \sqcup \partial U;\Lambda) \cong H_0(\partial U;\Lambda)=~\Lambda$.
Similarly, using the long exact sequence of the triple~$(W \setminus \mathring{U},M_1\sqcup \partial U, \partial U)$,
one can deduce that~$H_2(W\setminus \mathring{U},M_1 \sqcup \partial U;\Lambda)=0$.
Since we also have~$H_0(W\setminus \mathring{U},M_1 \sqcup \partial U;\Lambda)=0$, Poincar\'e duality and the UCSS imply that
$$ H_3(W\setminus \mathring{U},M_0;\Lambda)
=H^2(W\setminus \mathring{U},M_1 \sqcup \partial U;\Lambda)=0.
$$
To establish the claim, it only remains to show that the diagram in~\eqref{eq:KeyDiagram} is commutative.
The middle square clearly commutes.
Proposition~\ref{prop:UnionIsUnion} establishes the commutativity of the triangles but with~$H_2(M_0) \cup_{\widetilde{f}_*} H_2(M_1)$ in place of~$H_2(M_0) \cup_{\partial F} H_2(M_1)$.
However, since~$(f,F)$ is a compatible pair, these two modules are isomorphic.
This concludes our explanation of the diagram in~\eqref{eq:KeyDiagram} and concludes the proof of Claim~\ref{claim:quasi-formation}.
\end{proof}

Write $A$ and $B$ for the images of~$H_3(U,\partial U;\Lambda)$ and~$H_3(W \setminus \mathring{U};M_0 \sqcup \partial U;\Lambda)$ in~$H_2(\partial U;\Lambda)$.
As Proposition~\ref{prop:IsElementary} implies that~$H_2(M_0;\Lambda) \cup_{\partial F} H_2(M_1;\Lambda)=\Gamma_{F^{-*}} \oplus H_2(M_0;\Lambda)$, we deduce from Claim~\ref{claim:quasi-formation} that
$$ H_2(\partial U;\Lambda)=A \oplus B.$$
Assemble the long exact sequence of the triple~$(W \setminus \mathring{U},M_0 \sqcup \partial U,M_0)$ and the long exact sequence of the triple~$(W,W \setminus \mathring{U},M_0)$ into the following diagram, based on that from~\cite[p.~738]{KreckSurgeryAndDuality}, in which~$\Lambda$-coefficients are understood:
$$
\xymatrix@R0.5cm@C0.5cm{
&&&0\ar[d] \\
&&&B \cong H_3(W \setminus \mathring{U},M_0 \sqcup \partial U) \ar[d] \\
&&A\cong H_3(U,\partial U) \ar[r]\ar[d]^\cong& H_2(\partial U) \ar[d]\\
0\ar[r]& H_3(W,M_0) \ar[r]& H_3(W,W \setminus \mathring{U}) \ar[r]& H_2(W\setminus \mathring{U},M_0)\ar[r]\ar[d]& H_2(W,M_0) \ar[r]& 0. \\
&&&0&&
}
$$
Here, the left vertical isomorphism comes from excision.
Since~$H_2(\partial U;\Lambda)=A \oplus B$, the right-down composition~$A \to H_2(W\setminus \mathring{U},M_0;\Lambda)$ is an isomorphism. It follows that the central map in the long row is an isomorphism, and therefore~$H_i(W,M_0;\Lambda)=0$ for~$i=2,3$. Thus~$W$ is a relative~$h$-cobordism, and therefore an~$s$-cobordism, as desired.

Since~$\Z$ is a good group, by the topological~$s$-cobordism theorem~\cite[Theorem 7.1A]{FreedmanQuinn},~$M_0$ and~$M_1$ are homeomorphic, via a homeomorphism~$\Phi$ that extends~$f$.
It remains to show that~$\Phi$ induces the isometry~$F$ on~$H_2(-;\Lambda)$.
The inclusions of~$M_0,M_1$ into~$M=M_0 \cup_f -M_1 \subseteq W$ give rise to the following homomorphism:
$$ H_2(M_0) \oplus H_2(M_1) \xrightarrow{\iota=\bsm \iota_0 & \iota_1 \esm} H_2(M) \xrightarrow{j} H_2(W).$$
Set~$j_i=j \circ \iota_i$; these are isomorphisms because~$W$ is an~$s$-cobordism.
By definition~$j_1^{-1}j_0$ is the isometry induced by the homeomorphism~$\Phi$.
We noted that~$H_2(U;\Lambda)=H_2(W;\Lambda)$ as well as~$H_2(\partial U;\Lambda)=H_2(M;\Lambda)$.
We also noted in~\eqref{eq:TheFisTheGraph} that~$\ker(j)=\widehat{\lambda}_M^{-1}(\Gamma_{F^{-*}})$.
Using Lemma~\ref{rem:ContainsGraph}, we deduce that~$\iota(\Gamma_{-F}) \subseteq \ker(j)$.
This inclusion implies that for all~$x_0 \in H_2(M_0;\Lambda)$,
$$ 0=j\iota
\begin{pmatrix}
x_0 \\ -F(x_0)
\end{pmatrix}
=j\iota_0(x_0) -j\iota_1(F(x_0))
=j_0(x_0)-j_1(F(x_0)).
$$
Since the~$j_i$ are isomorphisms, we obtain the desired conclusion:~$ j_1^{-1}j_0(x_0)=F(x_0).$
This concludes the proof of Theorem~\ref{thm:MainTechnical}.
\end{proof}

\section{Knotted surfaces in simply-connected~$4$-manifolds.}
\label{sec:Knotted}

The goal of this section is to use Theorem~\ref{thm:MainTechnical} to prove Theorems~\ref{thm:WithBoundaryIntro} and~\ref{thm:Unknotting4ManifoldIntro} from the introduction.
Recall from our conventions that~$X$ refers to a closed, simply-connected, oriented~$4$-manifold, while~$N=X\setminus \mathring{D}^4$, and~$\Sigma$ always denotes a $\Z$-surface of genus $g$ either embedded in~$X$ or properly embedded in~$N$ (recall that a \emph{$\Z$-surface} refers to an oriented surface whose knot group is~$\Z$.)

This section is organised as follows.
In Subsection~\ref{sub:FactsZSurfaces}, we collect some initial facts concerning $\Z$-surfaces.
Subsection~\ref{sub:IsometriesBlanchfield} then shows that any isometry of~$\Bl_{\Sigma_{g,1} \times S^1}$ can be realised by a homeomorphism of~$\Sigma_{g,1} \times S^1$ (allowing us to construct a compatible pair between $\Z$-surfaces exteriors~$N_{\Sigma_0}$ and~$N_{\Sigma_1}$).
In Subsections~\ref{sub:SurfacesManifoldWithBoundary} and~\ref{subsection:surfaces-in-closed-4-manifolds}, we prove Theorems~\ref{thm:WithBoundaryIntro} and~\ref{thm:Unknotting4ManifoldIntro} from the introduction respectively.

\subsection{Facts about $\Z$-surfaces}
\label{sub:FactsZSurfaces}

\begin{lemma}
\label{lem:Nullhomologous}
If $\Sigma \subseteq N$ is a $\Z$-surface, then it is null-homologous.
\end{lemma}

\begin{proof}
We must show that~$[\Sigma,\partial \Sigma] =0 \in H_2(N,\partial N)$.
The intersection form~$Q_N$ pairs~$H_2(N,\partial N)$ with~$H_2(N)$ nonsingularly because~$N$ is simply-connected.
Thus, it is equivalent to show that~$Q_N([\Sigma,\partial \Sigma],x)=0$ for every class~$x \in H_2(N)$.
Represent such an~$x \in H_2(N)$ by a closed surface~$S \subseteq N$ that intersects~$\Sigma$ transversely in points~$p_1,\ldots,p_n$, so that~$Q_N([\Sigma,\partial \Sigma],x)=\sum_{i=1}^n \varepsilon(p_i)$, where~$\varepsilon(p_i)=\pm 1$.
Now the intersection~$S \cap N_\Sigma$ is a properly embedded surface in~$N_\Sigma$ with oriented boundary (homologous to)~$\sum_{i=1}^n \varepsilon(p_i)\mu_\Sigma$.
This implies that~$\sum_{i=1}^n \varepsilon(p_i)\mu_\Sigma =0\in H_1(N_\Sigma)$.
But now, since the homology group~$H_1(N_\Sigma)=\pi_1(N_\Sigma)=\Z\mu_\Sigma$ is torsion-free, we therefore deduce that~$Q_N([\Sigma,\partial \Sigma],x)=\sum_{i=1}^n \varepsilon(p_i)=0$, establishing that~$\Sigma$ is null-homologous.
\end{proof}

It follows easily that a closed $\Z$-surface $\Sigma \subseteq X$ is also null-homologous in $H_2(X)$.
For the next lemma, recall that orientable 2-plane bundles over compact, orientable surfaces are classified up to isomorphism by their (relative) Euler number.

\begin{lemma}\label{lem:trivial-normal-bundle}
  Every $\Z$-surface has trivial normal bundle. In the case of nonempty boundary this holds moreover relative to the Seifert framing on the boundary.
\end{lemma}

\begin{proof}
  In the case that the boundary is nonempty, the surface is oriented and therefore the normal bundle is trivial.
    However we need to show that this holds relative to the Seifert framing, and we need to show that the normal bundle is likewise trivial in the closed case.
    Continuing with the nonempty boundary case, cap off the $\Z$-surface $\Sigma$ with a Seifert surface $F$ in $S^3$.  Form a push-off by pushing $F$ off itself in $S^3$ and extending this to a normal push-off in $N$ along a generic section that agrees with the push-off of $F$ along $\partial F = \partial \Sigma$.

    Now the argument is the same in the closed case and in the nonempty boundary case (in the closed case take a push-off of $\Sigma$ using a generic section of the normal bundle). Now in both cases, by the proof of  Lemma~\ref{lem:Nullhomologous}, the push-off intersects $\Sigma$ algebraically trivially.  Therefore the normal bundle of $\Sigma$ has vanishing Euler number, relative to the Seifert framing in the case of nonempty boundary.  Thus the normal bundle of $\Sigma$ is trivial, again relative to the Seifert framing in the case of nonempty boundary.
\end{proof}

The boundary of the exterior of~$N_\Sigma$ is homeomorphic to~$M_{K,g}:=E_K \cup_\partial (\Sigma_{g,1} \times S^1)$, where~$\Sigma_{g,1}$ is the orientable genus~$g$ surface with one boundary component and~$E_K:=S^3 \setminus \nu (K)$.
We give more details on this identification as we will then make it implicitly throughout the remainder of the paper.

\begin{remark}
\label{rem:IdentifBoundary}
If~$\Sigma \subseteq X$ is a closed oriented surface with~$\pi_1(X_\Sigma)=\Z$,  then the homotopy class of the meridian of~$\Sigma \subseteq X$ is the unique nontrivial primitive class in~$\pi_1(\partial X_\Sigma)$ that bounds a disc in~$\Sigma \times D^2$ and maps to $1 \in \Z$.  Fix a framing of the normal bundle of~$\Sigma$, i.e.\ an identification~$\nu \Sigma \cong \Sigma \times \R^2$ compatible with the orientation, with the property that for each simple closed curve~$\gamma_k \subseteq \Sigma$, we have that ~$\gamma_k \times \{e_1\} \subseteq X \sm \Sigma$ is null-homologous in~$H_1(X \sm \nu \Sigma) \cong \Z$.
Use a choice of an identification~$\Sigma \cong \Sigma_g$ and this condition to fix an identification of the boundary of the exterior with~$\Sigma_g \times S^1$.
Any two choices now differ by an element of the mapping class group of~$\Sigma_g$.
Similarly, if~$\Sigma \subseteq N$ is a properly embedded, oriented surface, then we can identify~$\partial N_\Sigma$ with~$E_K \cup_\partial (\Sigma_{g,1} \times S^1).$
\end{remark}

\begin{lemma}
\label{lem:RibbonBoundary}
The exterior of a $\Z$-surface $\Sigma$ has ribbon boundary.
\end{lemma}

\begin{proof}
Since $N$ is simply-connected, we have
\[\{1\} = \pi_1(N_{\Sigma}) *_{\pi_1(\Sigma \times S^1)} \pi_1(\Sigma \times D^2),\]
which implies that $\pi_1(\Sigma \times S^1) \to \pi_1(N_{\Sigma})$ is onto.
In fact, by the parametrisation of $\Sigma \times S^1$ described in Remark~\ref{rem:IdentifBoundary}, we have that~$\pi_1(\Sigma)$ maps trivially to $\pi_1(N_{\Sigma})$, so the fundamental group of $N_{\Sigma}$ is generated by a meridian of $\Sigma$.
The closed case follows, since the exterior of a closed surface in $X$ can be thought of as the exterior of a surface in $N = X \sm \mathring{D}^4$.
\end{proof}


\begin{lemma}
\label{lem:AlexanderModuleSigmaS1}
The boundary of the exterior of a $\Z$-surface $\Sigma$ has $\Lambda$-torsion Alexander module.
More precisely, the following hold.
\begin{enumerate}
\item If $\Sigma \subseteq X$ is closed, then
\[ H_1(\partial X_\Sigma;\Lambda)=H_1(\Sigma_g \times S^1;\Lambda)=\left( \Lambda/(t-1)\right)^{\oplus 2g}.\]
\item If $\Sigma \subseteq N$ has boundary a knot $K$, then the inclusion induces an isomorphism
\begin{align*}
H_1(E_K;\Lambda) \oplus H_1(\Sigma_{g,1} \times S^1;\Lambda)\cong H_1(\partial N_\Sigma;\Lambda).
\end{align*}
In particular, if $K$ has Alexander polynomial $1$, then $H_1(\partial N_\Sigma;\Lambda)=\left( \Lambda/(t-1)\right)^{\oplus 2g}$.
\end{enumerate}
\end{lemma}

\begin{proof}
Using infinite cyclic covers, the assertion in the closed case is immediate:
\begin{equation*}
H_1(\partial X_\Sigma;\Lambda)=H_1(\Sigma_g \times S^1;\Lambda)=H_1(\widetilde{\Sigma_g \times S^1};\Z)=H_1(\Sigma_g \times \R)=\Z^{2g}=\left( \Lambda/(t-1)\right)^{\oplus 2g}.
\end{equation*}
We prove the second assertion.
Consider the Mayer-Vietoris sequence for the decomposition~$M_{K,g}=E_K \cup_\partial (\Sigma_{g,1} \times S^1)$ with~$\Lambda$ coefficients:
$$ \cdots \to H_1(S^1 \times S^1;\Lambda) \to H_1(E_K;\Lambda)\oplus H_1(\Sigma_{g,1} \times S^1;\Lambda) \to H_1(M_{K,g};\Lambda)\to 0.$$
We have~$H_1(S^1 \times S^1;\Lambda)=H_1(S^1 \times \R)= \Z$ and~$H_1(\Sigma_{g,1} \times S^1;\Lambda)=H_1(\Sigma_{g,1} \times \R)=~\Z^{2g}.$
The map~$H_1(S^1 \times S^1;\Lambda) \to H_1(E_K;\Lambda)$ is the zero map: it sends the generator to the lift of the longitude of~$E_K$; this is null-homologous in the infinite cyclic cover~$E_K^\infty$ of~$E_K$.
Next, the homomorphism~$H_1(S^1 \times S^1;\Lambda) \to H_1(\Sigma_{g,1} \times S^1;\Lambda)$ is also the zero map: it coincides with the map~$H_1(S^1) \to H_1(\Sigma_{g,1})$ sending the generator to~$[\partial \Sigma_{g,1}]$ which vanishes in~$H_1(\Sigma_{g,1})$.
This establishes the second assertion.
\end{proof}

\subsection{Isometries of the Blanchfield pairing}
\label{sub:IsometriesBlanchfield}

Given an orientable genus~$g$ surface~$\Sigma_{g,1}$ with one boundary component, we determine the isometries of the Blanchfield pairing on~$\Sigma_{g,1} \times S^1$ and show that they can be realised by orientation-preserving homeomorphisms.
We then describe the Blanchfield pairing of \[M_{K,g}:=E_K \cup (\Sigma_{g,1} \times S^1),\] where~$E_K:=S^3 \setminus \nu (K)$ is the knot exterior.
\medbreak

To describe the Blanchfield pairing~$\Bl_{\Sigma_{g,1} \times S^1}$, we fix some notation on the symplectic group:
\begin{align*}J_g&:=
\begin{pmatrix}
0&1 \\
-1&0
\end{pmatrix}^{\oplus g}, \\
\Sp_{2g}(\Z)&:=\lbrace A \in M_{2g \times 2g}(\Z) \mid A^TJ_gA=J_g \rbrace.
\end{align*}
Given a genus~$g$ surface~$\Sigma_{g,1}$ with one boundary component, we call a basis of~$H_1(\Sigma_{g,1})$ \emph{symplectic} if the intersection form on~$\Sigma_{g,1}$ with respect to this basis is represented by~$J_g$.
We also consider the map~$\varphi \colon \pi_1(\Sigma_{g,1} \times S^1) \twoheadrightarrow \Z$ induced by the projection on the second coordinate and recall that
$$ \Homeo^+_\varphi(\Sigma_{g,1} \times S^1)=\lbrace f \in \Homeo^+(\Sigma_{g,1} \times S^1) \ | \  f_* \circ \varphi=\varphi \rbrace.$$
The next proposition identifies the isometries of~$\Bl_{\Sigma_{g,1} \times S^1}$ with~$\Sp_{2g}(\Z)$ and uses this fact to show that the orbit set~$\Aut(\Bl_{\Sigma_{g,1} \times S^1})/\Homeo_\varphi^+(\Sigma_{g,1} \times S^1)$ is trivial.

\begin{proposition}
\label{prop:IsometriesBlanchfield}
A choice of a symplectic basis for~$H_1(\Sigma_{g,1})$ gives rise to an identification
$$ \Aut(\Bl_{\Sigma_{g,1} \times S^1})=\Sp_{2g}(\Z).$$
Every element of~$\Aut(\Bl_{\Sigma_{g,1} \times S^1})$ can be realised by a homeomorphism in~$\Homeo_\varphi^+(\Sigma_{g,1} \times S^1)$:
$$ \Aut(\Bl_{\Sigma_{g,1} \times S^1})/\Homeo_\varphi^+(\Sigma_{g,1} \times S^1)=\lbrace \id \rbrace.$$
Any such homeomorphism can be assumed to be of the form~$j \times \id_{S^1}$, for some~$j \in \Homeo^+(\Sigma_{g,1})$ that fixes the boundary of~$\Sigma_{g,1}$ pointwise.
\end{proposition}

\begin{proof}
Let~$h$ be an isometry of the Blanchfield pairing~$\Bl_{\Sigma_{g,1} \times S^1}.$
Since we saw in Lemma~\ref{lem:AlexanderModuleSigmaS1} that~$H_1(\Sigma_{g,1} \times S^1;\Lambda)=H_1(\Sigma_{g,1})$, the isometry~$h$ can be thought of as a~$\Z$-linear map.
We assert that it preserves the intersection form on the surface~$\Sigma_{g,1}$.
Pick a symplectic basis for~$H_1(\Sigma_{g,1})$ so that its intersection form is represented by~$J_g$.
By \cite[Corollary~1.2]{FriedlPowell}, the Blanchfield form on the fibred manifold~$\Sigma_{g,1} \times S^1$ is given by
$$ (v,w) \mapsto v^T (t^{-1}-1)^{-1}J_g w \in \Lambda_S/\Lambda.$$
Here,~$\Lambda_S=\Z[t^{\pm 1},(t-1)^{-1}]$ denotes the ring obtained from~$\Lambda=\Z[t^{\pm 1}]$ by inverting~$(t-1)$.
Since~$h$ can be thought of as a~$\Z$-linear map, we represent it by a matrix~$H$ with coefficients in~$\Z$.
As~$h$ is an isometry, we have the following equalities of matrices with coefficients in~$\Lambda_S/\Lambda$:
$$ (t^{-1}-1)^{-1} H^T J_gH=J_g (t^{-1}-1)^{-1}.$$
Since $H$ and~$J_g$ takes values in~$\Z$, the coefficients of the matrix~$H^TJ_gH-J_g$ take values in~$(t-1)\Lambda$.
On the other hand, $H^TJ_gH-J_g$ has coefficients in~$\Z$, so $H^TJ_gH-J_g =0$, and we deduce that~$H^TJ_gH=J_g$.
This shows that~$H$ is a symplectomorphism, proving the first assertion.

We prove the second assertion.
We may realise any~$h \in \Sp_{2g}(\Z)$ as the map on $H_1(\Sigma_{g,1})$ induced by an orientation-preserving homeomorphism~$j \colon \Sigma_{g,1} \to \Sigma_{g,1}$ that fixes the boundary pointwise~\cite[Section~2.1 and the discussion following Theorem 6.4]{FarbMargalit}.
Cross~$j$ with the identity on~$S^1$ to obtain an orientation-preserving homeomorphism~$j \times \id \colon \Sigma_{g,1} \times S^1 \to \Sigma_{g,1} \times S^1$ that belongs to~$\Homeo_\varphi^+(\Sigma_{g,1} \times~S^1)$.
\end{proof}

The next proposition describes the Blanchfield pairing of~$M_{K,g}=E_K \cup (\Sigma_{g,1} \times S^1)$ and its automorphism group; here recall that~$E_K=S^3 \setminus \nu (K)$ is the knot exterior.

\begin{proposition}
\label{prop:AutomorphismInTheBoundaryCase}
Given a knot~$K \subseteq S^3$,
 there is an isometry~$\Bl_{M_{K,g}}\cong \Bl_{E_K} \oplus \Bl_{\Sigma_{g,1} \times S^1}$ and
$$\Aut(\Bl_{M_{K,g}}) \cong \Aut(\Bl_{E_K}) \oplus \Aut(\Bl_{\Sigma_{g,1} \times S^1}).$$
\end{proposition}

\begin{proof}
We proved in Lemma~\ref{lem:AlexanderModuleSigmaS1} that the inclusion induces a $\Lambda$-isomorphism $H_1(E_K;\Lambda)\oplus H_1(\Sigma_{g,1} \times S^1;\Lambda) \cong H_1(M_{K,g};\Lambda).$
The isometry $\Bl_{M_{K,g}} \cong \Bl_{E_K} \oplus \Bl_{\Sigma_{g,1} \times S^1}$ now follows from~\cite[Theorem~1.1]{FriedlLeidyNagelPowell}.

We must now show that~$\Aut(\Bl_{E_K} \oplus \Bl_{\Sigma_{g,1} \times S^1})=\Aut(\Bl_{E_K}) \oplus \Aut(\Bl_{\Sigma_{g,1} \times S^1})$.
The order of~$H_1(M_K;\Lambda)$ is the Alexander polynomial~$\Delta_K$ (which is not divisible by~$(t-1)$) and the order of~$H_1(\Sigma_{g,1} \times S^1;\Lambda)$ is~$(t-1)^{2g}$.
An automorphism~$h \in  \Aut(\Bl_{E_K}) \oplus \Aut(\Bl_{\Sigma_{g,1} \times S^1})$ can be written as~$h=\bsm h_{11}&h_{12} \\ h_{21}&h_{22} \esm$.
Since the orders of~$H_1(M_K;\Lambda)$ and~$H_1(\Sigma_{g,1} \times S^1;\Lambda)$ are coprime, we deduce that~$h_{12}=h_{21}=0$.
\end{proof}

\subsection{Surfaces in manifolds with boundary}
\label{sub:SurfacesManifoldWithBoundary}

The aim of this subsection is to prove Theorem~\ref{thm:WithBoundaryIntro} from the introduction.
Given a knot~$K \subseteq S^3$, thanks to Proposition~\ref{prop:AutomorphismInTheBoundaryCase}, we can write automorphisms of~$\Bl_{M_{K,g}}$ as~$h=h_K \oplus h_\Sigma$ with~$h_K \in \Aut(\Bl_{E_K})$ and~$h_\Sigma \in \Aut(\Bl_{\Sigma_{g,1} \times S^1})$.

Let us recall the construction of an isotopy, as introduced in the introduction.
Let $f_K \colon E_K \to E_K$ be an orientation-preserving homeomorphism that is the identity on~$\partial E_K$. Extend $f_K$ via the identity on $\nu K$ to an orientation-preserving self-homeomorphism $\wt{f}_K$ of $S^3$.  The mapping class group of $S^3$ is trivial, so there is an isotopy  $\Theta(f_K)_t \colon S^3 \to S^3$ between the extension and the identity, such that $\Theta(f_K)_0 = \Id$ and $\Theta(f_K)_1 = \wt{f}_K$.
We can now prove Theorem~\ref{thm:WithBoundaryIntro}.

\begin{theorem}
\label{thm:WithBoundary}
Let $\Sigma_0,\Sigma_1 \subseteq N$ be two $\Z$-surfaces of genus $g$ with boundary $K$.
Suppose there is an isometry~$F \colon \lambda_{N_{\Sigma_0}} \cong \lambda_{N_{\Sigma_1}}$, and write~$\partial F=h_K \oplus h_\Sigma$.
\begin{itemize}
\item  If~$h_K$ is induced by an orientation-preserving homeomorphism~$f_K \colon E_K \to E_K$ that is the identity on~$\partial E_K$, then~$f_K$ extends to an orientation-preserving homeomorphism of pairs
$$(N,\Sigma_0) \xrightarrow{\cong} (N,\Sigma_1)$$
inducing the given isometry~$F \colon H_2(N_{\Sigma_0};\Lambda) \cong H_2(N_{\Sigma_1};\Lambda)$.
\item
If in addition~$N=D^4$, then for any choice of isotopy $\Theta(f_K)$, the surfaces~$\Sigma_0$ and~$\Sigma_1$ are topologically ambiently isotopic via an ambient isotopy of $D^4$ extending $\Theta(f_K)$.
\end{itemize}
\end{theorem}


%
%

\begin{proof}
By assumption, the boundary of the isometry~$F$ is~$\partial F=h_K \oplus h_\Sigma$ for some~$h_K \in \Aut (\Bl_{E_K})$ and some~$h_\Sigma \in \Aut (\Bl_{\Sigma_{g,1} \times S^1})$.
By Proposition~\ref{prop:IsometriesBlanchfield}, we can realise~$h_\Sigma$ by an orientation-preserving homeomorphism of the form~$j \times \id_{S^1}$ with~$j \colon \Sigma_{g,1} \to \Sigma_{g,1}$ an orientation-preserving homeomorphism that fixes the boundary pointwise.
Moreover,~$h_K$ is realised by an orientation-preserving homeomorphism~$f_K \colon E_K \to E_K$ that fixes the boundary $\partial E_K$.
Gluing these two homeomorphisms together, we obtain a homeomorphism
$$f:=f_K \cup (j \times \id_{S^1})  \colon M_{K,g} \to M_{K,g}.$$
Since~$f$ is the identity on~$\partial E_K \subseteq M_{K,g}$, we deduce that~$f \in \Homeo_\varphi^+(M_{K,g})$.
It follows that~$(f,F)$ is a compatible pair.
The~$N_{\Sigma_i}$ are~$4$-manifolds whose boundaries are ribbon (recall Lemma~\ref{lem:RibbonBoundary}) and have torsion Alexander modules for~$i=0,1$ (by Lemma~\ref{lem:AlexanderModuleSigmaS1}).
Since~$\Sigma_{i} \times D^2$ is a smooth manifold,~$\ks(\Sigma_{i} \times D^2)=0$. Thus for~$i=0,1$, by additivity of~$\ks$ under gluing we have $\ks(N_{\Sigma_i}) = \ks(N_{\Sigma_i}) + \ks(\Sigma_i \times D^2) =  \ks(N)$.
In particular~$\ks(N_{\Sigma_0}) = \ks(N_{\Sigma_1})$.
We may therefore apply Theorem~\ref{thm:MainTechnical} to extend~$f$ to an orientation-preserving homeomorphism
$\Phi \colon N_{\Sigma_0} \to N_{\Sigma_1}$ that induces~$F$.
As a consequence, we obtain the required homeomorphism of pairs:
$$\Phi':=\Phi \cup (j \times \id_{D^2}) \colon (N,\Sigma_0) \to (N,\Sigma_1).$$
This establishes the first assertion.

For the second assertion, let~$N=D^4$ and let $\Theta(f_K)_t \colon S^3 \to S^3$ be an isotopy of self-homeomorphisms of $S^3$ with $\Theta(f_K)_0 =\Id_{S^3}$  and $\Theta(f_K)_1$ equal to the extension $\wt{f}_K$ of $f_K$ by the identity to all of $S^3$.
Cone $\Theta(f_K)$ to obtain an isotopy $C(\Theta(f_K)_t) \colon D^4 \to D^4$ with $C(\Theta(f_K)_0) = \Id_{D^4}$ and $C(\Theta(f_K)_1)$ equal to the cone $C(\wt{f}_K)$ of $\wt{f}_K$.

Next, note that $C(\wt{f}_K)$ and $\Phi' \colon D^4 \to D^4$ are two homeomorphisms of $D^4$ that restrict to the same homeomorphism of $S^3$ on the boundary.
Therefore $C(\wt{f}_K)^{-1} \circ \Phi' \colon D^4 \to D^4$ is a homeomorphism restricting to $\Id_{S^3}$ on the boundary. By the Alexander trick, $C(\wt{f}_K)^{-1} \circ \Phi'$ is isotopic rel.\ boundary to the identity, via an isotopy $G_t \colon D^4 \to D^4$ with $G_0 = \Id_{D^4}$ and $G_1 = C(\wt{f}_K)^{-1} \circ \Phi'$. Note that $G_t|_{S^3} = \Id_{S^3}$ for all $t \in [0,1]$.
Then let
\[H_t := C(\Theta(f_K)_t) \circ G_t \colon D^4 \to D^4.\]
This is an isotopy with
\begin{align*}
  H_0 &= \Id_{D^4} \circ \Id_{D^4} = \Id_{D^4} \\
  H_1 &= C(\wt{f}_K) \circ C(\wt{f}_K)^{-1} \circ \Phi' = \Phi'.
\end{align*}
In addition, for every $t \in [0,1]$ we have that $H_t|_{S^3} = \Theta(f_K)_t \circ G_t|_{S^3} = \Theta(f_K)_t$.
Since $\Id_{D^4}(\Sigma_0) = \Sigma_0$ and $\Phi'(\Sigma_0) = \Sigma_1$, the isotopy $H_t$ is a topological ambient isotopy between~$\Sigma_0$ and~$\Sigma_1$, and it extends $\Theta(f_K)$, as required.
\end{proof}

We also note the following corollary, in which we give a slightly relaxed criterion, but without precise control on the isotopy of $K$. It is sometimes easier to construct homeomorphisms of $E_K$ that do not fix the boundary $\partial E_K$, but rather only fix the basepoint, so this could be a useful variation.

\begin{corollary}\label{corollary:realising-boundary-isos}
Let $K \subseteq S^3$ be a knot and fix a basepoint in $\partial E_K$.   Let $f \colon S^3 \to S^3$ be an orientation-preserving, basepoint-preserving homeomorphism with $f(K) = K$ as oriented knots, $f(E_K) = E_K$. Then $f$ induces an isomorphism $h \colon H_1(E_K;\Lambda) \to H_1(E_K;\Lambda)$. Suppose that $\Sigma_0$ and $\Sigma_1$ are two $\Z$-surfaces in $N = X \sm \mathring{D}^4$ with boundary $K \subseteq S^3 = \partial N$, and let $F \colon \lambda_{N_{\Sigma_0}} \cong \lambda_{N_{\Sigma_1}}$ be an isometry with $\partial F = h$. Then $\Sigma_0$ and $\Sigma_1$ are related by a homeomorphism of pairs, and if $N=D^4$ then they are ambiently isotopic.
\end{corollary}

\begin{proof}
Since $f$ is orientation-preserving as a map from $S^3$ to $S^3$, and as a map $K \to f(K)$, it restricts to an orientation-preserving map on $\partial E_K$ that preserves the homotopy class of the zero-framed longitude and therefore preserves the orientation of the meridian.  Thus $f$ commutes with the map $\pi_1(E_K) \to \Z$ inducing the $\Lambda$-coefficients, and therefore induces a map $h \colon H_1(E_K;\Lambda) \to H_1(E_K;\Lambda)$ as claimed; recall Proposition~\ref{prop:HomeoInduceIsometry}.
  Let $\theta_t \colon S^3 \to S^3$ be an isotopy from the identity of $S^3$ to $f$. Let $S^3 \times [0,1] \hookrightarrow N$ be a collar on the boundary of $N$ with $S^3 \times \{0\}$ mapping to $\partial N$. Let $\Theta_t \colon N \to N$ be the isotopy that is the identity outside $S^3 \times [0,1]$ and defined by $\Theta_t(x,s) = \theta_{t(1-s)}(x)$. This performs $\theta_t$ on the boundary, and tapers it in the collar.  Let $\Sigma_0' := \Theta_1(\Sigma_0)$, and note that $\Sigma_0'$ is ambiently isotopic (not necessarily rel.\ boundary) to $\Sigma_0$.
  In addition, $\partial (\Theta_1)_* = \partial F = h$ as isomorphisms of $H_1(E_K;\Lambda)$.  Therefore $G:= F \circ (\Theta_1)_*^{-1} \colon \lambda_{N_{\Sigma_0'}} \cong \lambda_{N_{\Sigma_1}}$ is an isometry with $\partial G = \id$.  By Theorem~\ref{thm:WithBoundaryIntro},  $\Sigma_0'$ and $\Sigma_1$ are related by a homeomorphism of pairs restricting to the identity on the boundary, and if $N=D^4$ then they are ambiently isotopic rel.\ boundary.  The corollary now follows immediately.
\end{proof}

\subsection{Surfaces in closed manifolds}\label{subsection:surfaces-in-closed-4-manifolds}

Now we recover the statements in the closed setting, proving Theorem~\ref{thm:Unknotting4ManifoldIntro} from the introduction.
As discussed in the introduction, for a closed, simply-connected~$4$-manifold~$X$ there is a classification of self-homeomorphisms of~$X$ in terms of their action on~$H_2(X)$.  The hypothesis of Theorem~\ref{thm:Unknotting4ManifoldIntro} is in terms of an isometry between the~$\La$-intersection forms of the surface exteriors~$X_{\Sigma_0}$ and~$X_{\Sigma_1}$.  In order to carefully state Theorem~\ref{thm:Unknotting4ManifoldIntro}, first we need to understand the relationship between such an isometry and isomorphisms of the second homology~$H_2(X)$ of the ambient~$4$-manifold.

\begin{lemma}
\label{lem:InducedIsometryOfH2(X)}
Let~$\Sigma_0,\Sigma_1 \subseteq X$ be closed $\Z$-surfaces of the same genus.
\begin{enumerate}
\item An isometry~$F \colon \lambda_{X_{\Sigma_0}} \cong \lambda_{X_{\Sigma_1}}$ induces an isomorphism~$F_\Z \colon H_2(X) \xrightarrow{\cong} H_2(X)$.
\item\label{item:induced-isom-item-2} If the isometry~$F$ is induced by a homeomorphism of pairs~$\Phi' \colon (X,\Sigma_0) \to (X,\Sigma_1)$, then~$F_\Z=\Phi_*'$; in particular,~$F_\Z$ is an isometry of the standard intersection form~$Q_{X}$.
\end{enumerate}
\end{lemma}

\begin{proof}
We claim that~$H_2(X_{\Sigma_i};\Lambda) \otimes_{\Lambda} \Z \cong H_2(X_{\Sigma_i})$.
To prove this, we use the universal coefficient spectral sequence with second page~$E_{p,q}^2=\operatorname{Tor}_p^\Lambda(H_q(X_{\Sigma_i};\Lambda),\Z)$ and which converges to~$H_*(X_{\Sigma_i})$.
The claim now follows promptly from the fact that~$H_1(X_{\Sigma_i};\Lambda)=0$ and
$$\operatorname{Tor}_2^\Lambda(H_0(X_{\Sigma_i};\Lambda),\Z)=\operatorname{Tor}_2^{\Z[\Z]}(\Z,\Z)=H_2(\Z;\Z)=H_2(S^1)=0.$$
The claim implies that the intersection form on $H_2(X_{\Sigma_i})$ is obtained from the intersection form on $H_2(X_{\Sigma_i};\Lambda)$ by tensoring down. Therefore~$F$ induces an isometry~$F \otimes_\Lambda \id_{\Z} \colon H_2(X_{\Sigma_0}) \to H_2(X_{\Sigma_1})$.
A Mayer-Vietoris argument then yields the following commutative diagram, where the dotted map labelled~$F_\Z$
is defined as the unique homomorphism that makes the right hand square commutes (both existence and uniqueness follow from a short diagram chase):
\begin{equation}
\label{eq:InducedIsomOfX}
\xymatrix@R0.5cm{
0 \ar[r] & \Z^{2g} \ar[r]\ar[d]^{(F \otimes_\Lambda \id_{\Z})|} & H_2(X_{\Sigma_0}) \ar[r]_{p_0}\ar[d]^{F \otimes_\Lambda \id_{\Z}}  & H_2(X) \ar[r] \ar@{..>}[d]^{F_\Z} &0  \\
0 \ar[r] & \Z^{2g} \ar[r] & H_2(X_{\Sigma_1}) \ar[r]^{p_1} & H_2(X) \ar[r]&0.
}
\end{equation}
%
The second assertion is now immediate:
 the homeomorphism~$\Phi'$ induces an isometry of~$(H_2(X),Q_X)$ that satisfies~$\Phi_*' \circ p_0=p_1 \circ (F \otimes_\Lambda \id_{\Z})$, so $F_\Z=\Phi'_*$ is also an isometry.
\end{proof}

We move on to the proof of Theorem~\ref{thm:Unknotting4ManifoldIntro}, whose statement we recall here for the benefit of the reader.

\begin{theorem}
\label{thm:Unknotting4Manifold}
Let~$\Sigma_0,\Sigma_1 \subseteq X$ be two closed $\Z$-surfaces of the same genus.
\begin{enumerate}
\item\label{item:closed-case-1-maintext}
If the intersection forms~$\lambda_{X_{\Sigma_0}}$ and~$\lambda_{X_{\Sigma_1}}$ are isometric via an isometry~$F$, then there is an orientation-preserving homeomorphism of pairs
$$\Phi \colon (X,\Sigma_0) \xrightarrow{\cong} (X,\Sigma_1)$$
inducing the given isometry~$\Phi_* = F \colon H_2(X_{\Sigma_0};\Lambda) \cong H_2(X_{\Sigma_1};\Lambda)$.
\item\label{item:closed-case-2-maintext}
The isometry~$F$ induces an isometry~$F_\Z \colon H_2(X) \to H_2(X)$ of the standard intersection form~$Q_X$ of~$X$ by Lemma~\ref{lem:InducedIsometryOfH2(X)}. The surfaces $\Sigma_0$ and~$\Sigma_1$ are topologically ambiently isotopic if and only if $F_\Z=\id$.
 \end{enumerate}
\end{theorem}


\begin{proof}
After an ambient isotopy, assume that~$\Sigma_0$ and $\Sigma_1$ coincide on a disc~$D^2 \subseteq \Sigma_0 \cap \Sigma_1$.
Assume that the normal bundles also coincide over this~$D^2$.
Consider the preimage~$\mathring{D}^2 \times \R^2 \subseteq \nu \Sigma_i$.
This is homeomorphic to an open 4-ball~$\mathring{D}^4$.
Remove this~$(\mathring{D}^4,\mathring{D}^2)$ from~$(X,\Sigma_i)$ to obtain~$(N,\widetilde{\Sigma}_i)$, with~$\partial \widetilde{\Sigma}_i=\Sigma_i \cap \partial N$ the unknot~$K$ in~$S^3$.
Then the exterior of~$\Sigma_i$ in~$X$ equals the exterior of~$\widetilde{\Sigma}_i$ in~$N$.

Since~$X_{\Sigma_i}=N_{\widetilde{\Sigma}_i}$, the~$\Lambda$-intersection forms are unchanged and the isometry~$F$ also induces an isometry~$F \colon \lambda_{N_{\widetilde{\Sigma}_0}} \cong \lambda_{N_{\widetilde{\Sigma}_1}}$.
Write~$\partial F=h_K \oplus h_\Sigma$ as in Theorem~\ref{thm:WithBoundary}.
The unknot~$K \subseteq S^3$ has trivial Alexander module, so~$h_K=\id$ and so~$h_K$ is realised by the homeomorphism~$f_K = \id \colon E_K \to E_K$.
Theorem~\ref{thm:WithBoundary} provides a rel.\ boundary homeomorphism of pairs~$\Phi' \colon (N,\widetilde{\Sigma}_0) \to (N,\widetilde{\Sigma}_1)$ that induces the isometry~$F$ on the~$\Lambda$-homology of the surface exteriors.
We recover the required homeomorphism of pairs~$\Phi \colon (X,\Sigma_0) \to (X,\Sigma_1)$ by gluing~$\Phi'$ with the identity homeomorphism~$(D^4,D^2) \to (D^4,D^2)$.

Lemma~\ref{lem:InducedIsometryOfH2(X)} implies that the isometry~$F$ induces an isometry~$F_\Z$ of the standard intersection form~$Q_X$. In particular, Lemma~\ref{lem:InducedIsometryOfH2(X)}~\eqref{item:induced-isom-item-2} specifies that~$F_\Z$ is induced by the homeomorphism~$\Phi$.
If~$F_{\Z}=\id$, then~$\Phi$ is a self-homeomorphism of~$X$ inducing the identity on~$H_2(X)$, so by \cite[Theorem~1.1]{QuinnIsotopy},~\cite[Theorem~10.1]{FreedmanQuinn},~$\Phi$ is isotopic to the identity. It follows that~$\Sigma_0$ and~$\Sigma_1$ are topologically ambiently isotopic.

On the other hand if $\Sigma_0$ and~$\Sigma_1$ are topologically ambiently isotopic, then the induced homeomorphism between their exteriors extends to a homeomorphism from $X$ to itself that is isotopic to the identity, and so certainly induces the identity map on $H_2(X)$.
\end{proof}

\section{Equivariant intersection forms of surface exteriors}
\label{sec:IntersectionForms}

The goal of this section is to collect some results about the intersection forms of $\Z$-surface exteriors.
The main result of Subsection~\ref{sub:IntersectionFormHyperbolic} shows that up to direct summands with the Hermitian form
 $$ \mathcal{H}_2:=\left( \Lambda^2,
\begin{pmatrix}
0&t-1 \\
t^{-1}-1&0
\end{pmatrix} \right),
$$
the equivariant intersection forms of any two $\Z$-surface exteriors are isometric (Proposition~\ref{prop:StabiliseToUnknot}).
Subsection~\ref{sub:AlexanderPolynomial1} then focuses on the case of $\Z$-surfaces $\Sigma \subseteq N=X\setminus \mathring{D}^4$ with boundary an Alexander polynomial one knot $K \subseteq S^3$.
The main result is Corollary~\ref{cor:IntersectionFormAlexanderPolynomial1} which shows that~$\lambda_{N_\Sigma}$ becomes isometric to~$Q_X \oplus \mathcal{H}_2^{\oplus g}$ after adding sufficiently many copies of~$\mathcal{H}_2$, where~$Q_X$ is the~$\Z$-valued intersection form of~$X$ (note that $Q_X=Q_N$).

\subsection{Equivariant intersection forms of $\Z$-surface exteriors}
\label{sub:IntersectionFormHyperbolic}

The goal of this section is to study the~$\Lambda$--intersection form of $\Z$-surface exteriors up to stabilisations by $\mathcal{H}_2$.
\medbreak
We start by describing the~$\Lambda$--intersection form of unknotted surfaces in~$S^4$.

\begin{lemma}
\label{lem:Unknot}
If~$\Sigma \subseteq S^4$ is an unknotted oriented closed genus~$g$ surface, then the~$\Lambda$--intersection form of~$S^4_\Sigma$ is isometric to~$\mathcal{H}_2^{\oplus g}$.
\end{lemma}

\begin{proof}
We carry out the computation for the standardly embedded genus~$g$ surface~$\Sigma \subseteq S^4$.
Use~$U \subseteq S^3$ to denote the unknot.
Slice~$(S^4,\Sigma)$ along an equatorial~$(S^3,U)$ to obtain a decomposition~$(S^4,\Sigma)=(D^4,\mathring{\Sigma}) \cup (D^4,D)$, with~$D \subseteq D^4$ a disc bounding~$U$ and~$\mathring{\Sigma} \subseteq D^4$ a punctured unknotted surface in~$D^4$.

Note that~$S^4_\Sigma$ and~$D^4_{\mathring{\Sigma}}:=D^4 \setminus \nu \mathring{\Sigma}$ are homeomorphic, since we can assume that the removed~$\mathring{D}^4$ lies in the regular neighbourhood~$\nu \Sigma$ removed from~$S^4$ to form~$S^4_{\Sigma}$.
It follows that~$H_2(S^4_\Sigma;\Lambda)=H_2(D^4_{\mathring{\Sigma}};\Lambda)$ and that the~$\Lambda$--intersection forms agree.


\begin{figure}[!htb]
\centering
\captionsetup[subfigure]{}
\begin{subfigure}[t]{.35\textwidth}
\centering
\labellist
    \small
    \pinlabel {$0$} at 105 590
        \pinlabel {$0$} at 205 590
            \pinlabel {$0$} at 395 590
                \pinlabel {$0$} at 495 590
    \endlabellist
\includegraphics[scale=0.3]{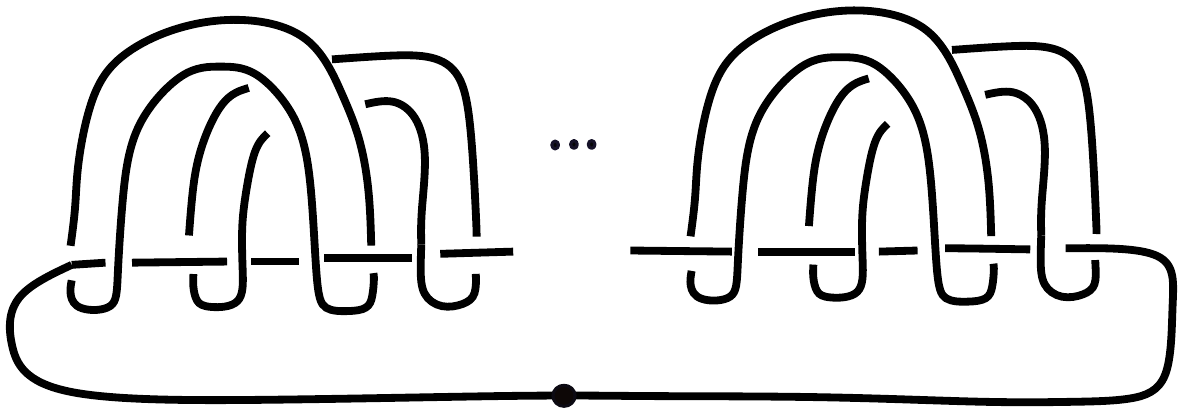}
\end{subfigure}
 \hspace{2.5cm}
\begin{subfigure}[t]{.35\textwidth}
\centering
\labellist
    \small
    \pinlabel {$0$} at 75 590
        \pinlabel {$0$} at 175 590
    \pinlabel {$0$} at 378 590
        \pinlabel {$0$} at 478 590
    \endlabellist
\includegraphics[scale=0.3]{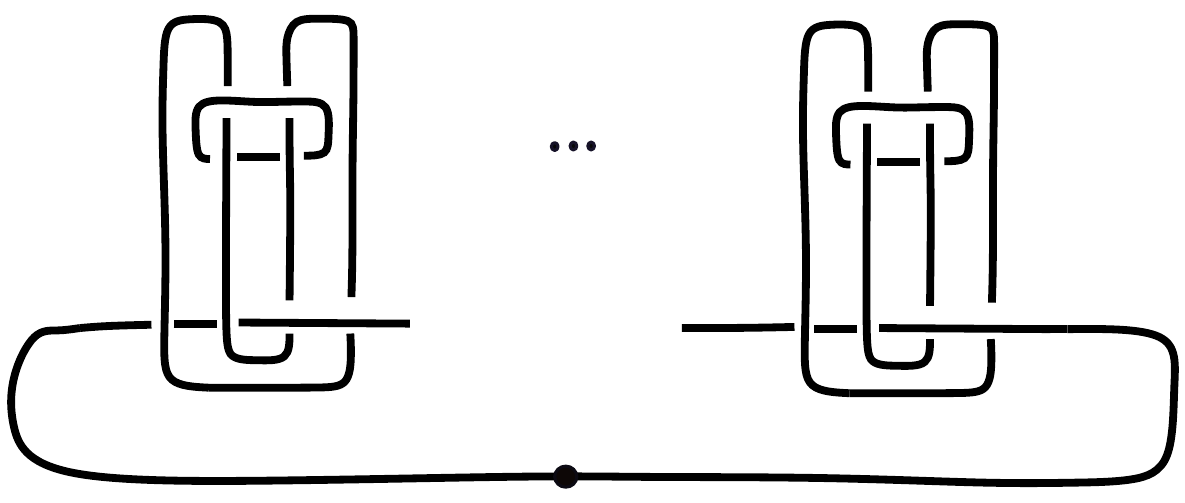}
\end{subfigure}
\caption{Two handle diagrams for the exterior of a standardly embedded genus~$g$ surface~$F \subseteq D^4$ with boundary the unknot in~$S^3$.}
\label{fig:HandleDiagram}
\end{figure}
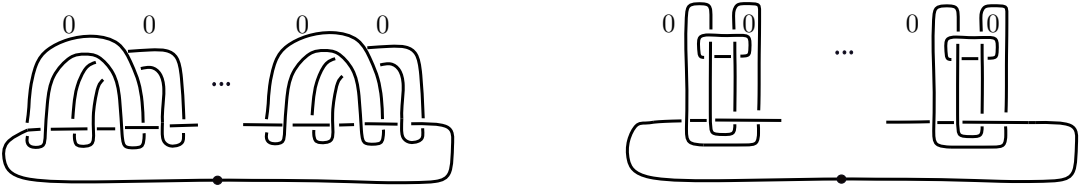
Thus it remains to compute the~$\Lambda$--intersection form of the exterior~$D^4_F$ of a properly embedded unknotted surface~$F \subseteq D^4$.
A handle diagram with a single one handle and~$2g$ two handles for~$D^4_F$ appears in the left hand side of Figure~\ref{fig:HandleDiagram}, produced using the formalism of~\cite[Section~6.2]{GompfStipsicz}.
It can then be isotoped as in the right hand side of Figure~\ref{fig:HandleDiagram}, leading to a handle diagram for the infinite cyclic cover of~$D^4_F$ depicted in Figure~\ref{fig:UnknottedSurfaceCover}.
From this diagram, by taking the union of the cores of the 2-handles with null-homotopies of their attaching curves in the~$4$-ball, we obtain generators of~$\pi_2(D^4_F)=H_2(D^4_F;\Lambda)=\Lambda^{2g}$. The~$\Lambda$--intersection form can be computed via (equivariant) linking numbers, yielding the required result.
\begin{figure}[!htb]
\centering
\labellist
    \small
    \pinlabel {$0$} at 100 592
        \pinlabel {$0$} at 160 592
    \pinlabel {$0$} at 230 592
        \pinlabel {$0$} at 290 592
            \pinlabel {$0$} at 100 480
        \pinlabel {$0$} at 160 480
    \pinlabel {$0$} at 230 480
        \pinlabel {$0$} at 290 480
    \endlabellist
\includegraphics[scale=0.5]{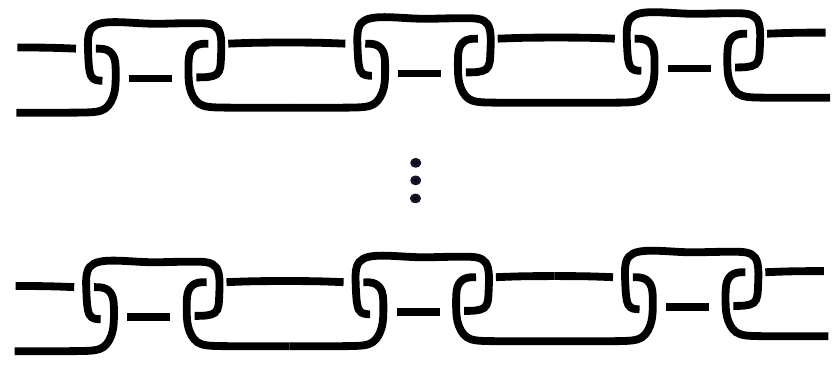}
\caption{A handle diagram for the infinite cyclic cover of~$D^4_F$, where~$F \subseteq D^4$ is an unknotted punctured surface.}
\label{fig:UnknottedSurfaceCover}
\end{figure}
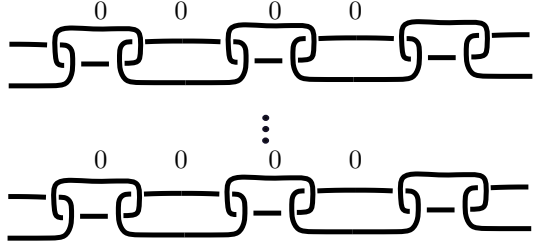
\end{proof}

We recall the concept of a~$1$-handle stabilisation for a surface in a~$4$-manifold.  The following definition was motivated by~\cite{JuhaszZemke}.
Let~$\Sigma \subseteq V$ be a locally flat (connected) surface embedded in a~$4$-manifold~$V$.
Let~$B$ be an embedding of~$D^4$ into~$V$ such that~$\partial B$ intersects~$\Sigma$ transversely in a~$2$-component unlink~$L$ and~$B$ intersects~$\Sigma$ transversely in two discs~$D_0$ and~$D_1$, which can be simultaneously isotoped within~$B$ to lie in~$\partial B$.
Suppose that, for~$i=0,1$, a~$3$-dimensional~$1$-handle~$D^2 \times [0,1]$ is embedded in the interior of~$B$ such that~$D^2 \times \lbrace i \rbrace=D_i$.
The resulting \emph{$1$-handle stabilisation of~$\Sigma$} is defined as
$$\Sigma'=\left( \Sigma \cap (S^4 \setminus B) \right) \cup_L (S^1 \times [0,1]).$$
If there exists a path $\gamma$ in $\Sigma$ between~$(1,0)$ and $(1,1)$ in $S^1 \times [0,1]$ such that $\gamma \cup (\{1\} \times [0,1])$ is a null-homotopic loop in~$V$, via a null-homotopy $h \colon D^2 \to V$ with $h(\mathring{D}^2) \subseteq V \sm \Sigma$, then we call the stabilisation \emph{trivial}.


A detailed discussion of this construction in the locally flat setting can be found in~\cite[Proposition~9.1]{FriedlNagelOrsonPowell}.
The next result describes the effect of trivial~$1$-handle stabilisation on the~$\Lambda$--intersection form of surface exteriors.

\begin{lemma}
\label{lem:Handle}
If~$\Sigma \subseteq N$ is a $\Z$-surface and~$\Sigma'\subseteq N$ is obtained from~$\Sigma$ by a trivial~$1$-handle stabilisation, then
$$\lambda_{N_{\Sigma'}} \cong \lambda_{N_\Sigma} \oplus \mathcal{H}_2.$$
\end{lemma}

\begin{proof}
Since the~$1$-handle stabilisation is trivial, one can write~$(N,\Sigma')=(N,\Sigma) \# (S^4,T^2)$, where~$T^2 \subseteq~S^4$ denotes a standardly embedded torus, and~$\#$ denotes the interior connect sum.
It follows that~$N_{\Sigma'}=N_{\Sigma} \cup S^4_{T^2}$, where the identification takes place along thickened meridians:~$\mu_\Sigma \times D^2$ is identified with~$\mu_{T^2} \times D^2$.
One thus deduces the isomorphism~$\pi_1(N_{\Sigma'})\cong\pi_1(N_\Sigma)=\Z$.
As the coefficient system maps these meridians to~$1 \in \Z$, a straightforward Mayer-Vietoris argument shows that~$H_2(N_{\Sigma'};\Lambda)=H_2(N_\Sigma;\Lambda)\oplus H_2(S^4_{T^2};\Lambda)$, noting that~$H_1(\mu_{\Sigma};\Lambda)=~0$.
It then follows that~$\lambda_{N_{\Sigma'}}=\lambda_{N_\Sigma}\oplus \lambda_{S^4_{T^2}}$.
The result is now a consequence of Lemma~\ref{lem:Unknot}.
\end{proof}

The next proposition shows that the $\Lambda$-intersection forms of any two $\Z$-surfaces exteriors become isometric after adding sufficiently many $\mathcal{H}_2$ summands.

\begin{proposition}
\label{prop:StabiliseToUnknot}
For $\Z$-surfaces $\Sigma_0,\Sigma_1 \subseteq N$ of the same genus with boundary $K$, there exists an integer~$n \geq 0$ and an isometry
\[\lambda_{N_{\Sigma_0}} \oplus \mathcal{H}_2^{\oplus n} \cong \lambda_{N_{\Sigma_1}}\oplus \mathcal{H}_2^{\oplus n}. \]
\end{proposition}

\begin{proof}
Any two null-homologous surfaces in a~$4$-manifold can be made isotopic by enough~$1$-handles stabilisations to each surface~\cite[Theorem 5]{BaykurSunukjian}.
While Baykur and Sunukjian prove this result in the smooth category for surfaces in closed manifolds, it also applies in the topological category for properly embedded surfaces in~$4$-manifolds with boundary; see Theorem~\ref{thm:BS} for a detailed proof in the case at hand.

We apply this to~$\Sigma_0,\Sigma_1.$
This is possible because $\Z$-surfaces are nullhomologous by Lemma~\ref{lem:Nullhomologous}.
Since~$\pi_1(\partial (\nu \Sigma_{i})) \to \pi_1(N_{\Sigma_{i}})=\Z$ is surjective, all~$1$-handles stabilisations can be taken to be trivial~\cite[Lemma~3]{BaykurSunukjian}.
We deduce that after sufficiently many (say~$n$) trivial 1-handles stabilisations,~$\Sigma_0$ becomes isotopic to~$\Sigma_1$, also stabilised~$n$ times.
Using Lemma~\ref{lem:Handle}, each such stabilisation adds an~$\mathcal{H}_2$-summand to the~$\Lambda$--intersection form of the corresponding surface exterior, from which the result follows.
\end{proof}

\subsection{$\Z$-surfaces for Alexander polynomial one knots}
\label{sub:AlexanderPolynomial1}
We now restrict our attention to $\Z$-surfaces with boundary Alexander polynomial one knots.
\medbreak
An Alexander polynomial one knot~$K \subseteq S^3$ bounds a disc~$D \subseteq D^4$ with~$\pi_1(D^4 \setminus D)=\Z$~\cite{FreedmanQuinn}.
We call this disc \emph{a~$\Z$-slice disc of~$K$.}
When~$S^3=\partial N$ with~$N=X \setminus \mathring{D}^4$, we arrange that this disc~$D$ belong to a collar neighborhood~$S^3 \times [0,1] \subseteq N$ of~$S^3=\partial N$.
We use this disc to build a genus~$g$ surface with boundary~$K$.

\begin{definition}
\label{def:Target}
A \emph{genus~$g$ target surface} $\Sigma_g^t$ for an Alexander polynomial one knot $K$ is an embedded surface obtained from a $\Z$-disc $D$ of $K$ by~$g$ trivial~$1$-handle stabilisations.
\end{definition}

Although we do not need this fact in the sequel, an Alexander polynomial one knot~$K \subseteq S^3$ bounds a \emph{unique} $\Z$-disc: this follows either from~\cite[Theorem 1.2]{ConwayPowell} or from Theorem~\ref{thm:WithBoundary} above, because $H_2(D^4_D;\Lambda)=0$.
This is why we do not keep track of the $\Z$-disc~$D$ in the notation for target surfaces.

The next lemma describes the~$\Lambda$--intersection form of a target surface exterior.

\begin{lemma}
\label{lem:Unknot-2}
If $D \subseteq S^3 \times [0,1] \subseteq N$ is a $\Z$-disc, then $\lambda_{N_D}=Q_X$. If $\Sigma_g^t$ is a target surface, then $\lambda_{N_{\Sigma_g^t}} \cong Q_X \oplus \mathcal{H}_2^{\oplus g}$.
\end{lemma}

\begin{proof}
Since~$D$ is properly embedded in a collar neighbourhood of~$\partial N$, $N_D$ is homeomorphic to the interior connected sum~$N_D \cong D^4_{D} \# X$.
The exterior~$D^4_D$ of a~$\Z$-slice disc is aspherical, and therefore we have~$H_i(D^4_D;\Lambda)=0$ for~$i>0$~\cite[Lemma~2.1]{ConwayPowell}.
The first assertion now follows from a straightforward Mayer-Vietoris argument.
Since a target surface is obtained from a~$\Z$-slice disc by~$g$ handle additions, the second assertion follows from the first and Lemma~\ref{lem:Handle}.
\end{proof}


Finally, we describe the~$\Lambda$--intersection form of a $\Z$-surface exterior for an Alexander polynomial one knot $K$.

\begin{corollary}
\label{cor:IntersectionFormAlexanderPolynomial1}
If $\Sigma \subseteq N=X\setminus \mathring{D}^4$ is a genus $g$ $\Z$-surface for an Alexander polynomial one knot $K$, then $\lambda_{N_\Sigma} \oplus \mathcal{H}_2^{\oplus n} \cong  Q_X \oplus \mathcal{H}_2^{\oplus (n+g)}$ for some integer $n \geq 0$.
\end{corollary}

\begin{proof}
Since $K$ has Alexander polynomial one, it bounds a genus $g$ target surface $\Sigma_g^t$.
Proposition~\ref{prop:StabiliseToUnknot} ensures that $\lambda_{N_\Sigma} \oplus \mathcal{H}_2^{\oplus n}$ is isometric to $\lambda_{N_{\Sigma_g^t}} \oplus \mathcal{H}_2^{\oplus n}$ for some $n \geq 0$.
The result now follows from Lemma~\ref{lem:Unknot-2}, thanks to which $\lambda_{N_{\Sigma_g^t}}=Q_X \oplus \mathcal{H}_2^{\oplus g}.$
\end{proof}

\section{Knotted surfaces in $S^4$ and $D^4$}
\label{sec:S4D4}

As we outlined in Subsection~\ref{sub:DeducingIntro}, our results about surfaces in $D^4$ and $S^4$ can be deduced from Theorem~\ref{thm:WithBoundary} once we prove that the surface exteriors have isometric equivariant intersection forms.
In Subsection~\ref{sub:intersection-forms-in-D4}, we show that if $\Sigma \subseteq D^4$ is a genus $g\geq 3$ $\Z$-surface with boundary an Alexander polynomial one knot $K$, then $\lambda_{D^4_\Sigma} \cong \mathcal{H}_2^{\oplus g}$.
In Subsection~\ref{sub:SurfacesS4D4} we then combine this result with Theorem~\ref{thm:WithBoundary} to deduce our results about surfaces in $D^4$ and $S^4$, in particular proving Theorems~\ref{thm:UnknottingIntro} and \ref{thm:D4Intro}.
In Subsection~\ref{sub:PushedIn}, we apply these results to study Seifert surfaces pushed into $D^4$.

\subsection{Intersection forms in $D^4$}\label{sub:intersection-forms-in-D4}
The aim of this subsection is to show that if $\Sigma \subseteq D^4$ is a $\Z$-surface of genus $g \geq 3$ with $\Delta_{\partial \Sigma}\doteq 1$, then $\lambda_{D^4_\Sigma} \cong\mathcal{H}_2^{\oplus g}.$
\medbreak
Before we continue, we need to recall some general theory on~$\eps$-Hermitian forms, and on their~$\eps$-quadratic counterparts; further details can be found in~\cite[Section~1.1]{RanickiExact} and~\cite[Section 2]{CrowleySixt}, where the terminology ``$\eps$-symmetric" is used, as is customary in $L$-theory.
Let~$R$ be a ring with involution and let~$\eps \in R$ be a central unit with~$\eps \ol{\eps} =1$.
Let~$M$ be a finitely generated free left~$R$-module.
 Let~$\Sesq(M)$ denote the abelian group of~$R$-sesquilinear forms on~$M$, meaning that~$b(rx,sy) = r b(x,y) \ol{s}$ for~$b \in \Sesq(M)$,~$x,y \in M$, and~$r,s \in R$.
Define an involution:
\begin{align*}
  &T_{\eps} \colon \Sesq(M) \to \Sesq(M) \\
  &(T_{\eps}b)(y,x) = \eps \ol{b(x,y)}.
\end{align*}
A quick check shows that~$(T^2_{\eps}b)(x,y) = \eps b(x,y) \ol{\eps} = \eps\ol{\eps} b(x,y) = b(x,y)$, so the conditions that~$\eps$ is central and~$\eps \ol{\eps} =1$ are crucial for~$T_{\eps}$ to be an involution.
This enables us to define the \emph{symmetric~$Q$-group}~$Q^{\eps}(M)$ of~$M$ and the \emph{quadratic~$Q$-group}~$Q_{\eps}(M)$ of~$M$ via the exact sequence
\[0 \to Q^{\eps}(M)  \to \Sesq(M) \xrightarrow{1-T_{\eps}} \Sesq(M) \to Q_{\eps}(M) \to 0.\]
In other words,
\[Q^{\eps}(M):= \ker(1-T_{\eps}) \text{ and } Q_{\eps}(M):= \coker(1-T_{\eps}).\]
A sesquilinear form~$b \colon M \times M \to R$ is called
\emph{$\varepsilon$-Hermitian} if it belongs to~$Q^{\eps}(M)$, that is if~$b(y,x)=\varepsilon \overline{b(x,y)}$ for every~$x,y \in M$.
A~$(+1)$-Hermitian form is a Hermitian form in the usual sense, while a~$(-1)$-Hermitian form is a skew-Hermitian form in the usual sense.
A~$\varepsilon$-Hermitian form is called \emph{hyperbolic} if it is isometric to~$H^\varepsilon(R)^{\oplus g}=\left(R^{2g},\bsm 0&1 \\ \varepsilon & 0 \esm^{\oplus g} \right)$ for some~$g \geq~0$.
The reason for introducing this terminology is that we will shortly be concerned with~$(-t)$-Hermitian forms over~$\Lambda$, and their quadratic analogues.

An element~$\psi \in Q_{\eps}(M)$ is called an \emph{$\eps$-quadratic form}.  To an~$\eps$-quadratic form in~$Q_{\eps}(M)$ is associated its \emph{symmetrisation}~$\varphi := (1+T_{\eps})(\psi) \in Q^{\eps}(M)$. Given~$\varphi \in Q^{\eps}(M)$, a quadratic form~$\psi$ with~$(1+T_{\eps})(\psi) = \varphi$ is called a \emph{quadratic refinement of~$\varphi$}.
The symmetrisation is well-defined on equivalence classes in~$Q_{\eps}(M)$ because~$(1+T_{\eps})(1-T_{\eps}) = 1-T_{\eps}^2 = 0$.
Note that quadratic forms are considered up to addition of forms in the image of~$1-T_{\eps}$.
A quadratic form over~$R$ is \emph{hyperbolic} if it is isometric to~$H_\varepsilon(R)^{\oplus g}=\left(R^{2g},\bsm 0&1 \\ 0&0 \esm^{\oplus g} \right)$ for some~$g \geq 0$.
Let us emphasise that subscripts denote quadratic forms, while superscripts denote~$\eps$-Hermitian forms.

Here is the key relationship between~$(-t)$-Hermitian
and $(-t)$-quadratic forms that we shall exploit in the proof of Theorem~\ref{thm:HyperbolicIntersectionForm} below.
\begin{lemma}\label{lem:Q-group-map}
Let~$M$ be a finitely generated free~$\La$-module. Then the map
\[(1+T_{-t}) \colon Q_{-t}(M) \to Q^{-t}(M)\]
is injective.
\end{lemma}

\begin{proof}
To prove the lemma we will show the following facts.
\begin{enumerate}
    \item\label{item:Q-group-lemma-1} For~$\varphi \in Q^{-t}(M)$, the assignment~$\varphi \mapsto (t^{-1}-1) \varphi$ determines a map \[\Omega \colon Q^{-t}(M) \to~Q^{+1}(M).\]
    \item\label{item:Q-group-lemma-2} For~$\psi \in \Sesq(M)$, the assignment~$\psi \mapsto (t^{-1}-1) \psi$ determines an
injective map
  $$\Xi \colon Q_{-t}(M) \to Q_{+1}(M).$$
   \item\label{item:Q-group-lemma-3} The map~$(1+T_{+1}) \colon Q_{+1}(M) \to Q^{+1}(M)$ is injective.
 \item\label{item:Q-group-lemma-4} The diagram
 \[\xymatrix{Q_{-t}(M) \ar[r]^-{\Xi} \ar[d]_{1+T_{-t}}  & Q_{+1}(M) \ar[d]^{1+T_{+1}} \\ Q^{-t}(M) \ar[r]^-{\Omega} & Q^{+1}(M)  }\]
 commutes.
  \end{enumerate}
Then, since the right-down composition is injective, it follows that the left vertical map is injective, as desired. So it suffices to prove these four assertions.

For \eqref{item:Q-group-lemma-1}, suppose that~$\varphi \in Q^{-t}(M) = \ker(1-T_{-t})$, that is~$\varphi = -t\ol{\varphi}^T$. Then
\[\ol{(t^{-1}-1)\varphi}^T = (t-1) \ol{\varphi}^T = (t^{-1}-1)(-t)\ol{\varphi}^T = (t^{-1}-1)\varphi.\]
Therefore~$\Omega(\varphi) = (t^{-1}-1)\varphi \in \ker (1-T_{+1}) = Q^{+1}(M)$ as desired.

To prove \eqref{item:Q-group-lemma-2}, first we show that~$\Xi\colon Q_{-t}(M) \to Q_{+1}(M)$ is well-defined. We need to see that for every sesquilinear form~$\theta \in \Sesq(M)$, the element~$(1-T_{-t})(\theta) = \theta - (-t) \ol{\theta}^T = \theta + t \ol{\theta}^T$ maps to the trivial element in~$Q_{+1}(M)$.
This is indeed the case since
\[(t^{-1}-1)( \theta + t \ol{\theta}^T) = (t^{-1}-1) \theta - (t-1) \ol{\theta}^T = (1-T_{+1})((t^{-1}-1)\theta)\]
is trivial in~$\coker(1-T_{+1}) = Q_{+1}(M)$.  Therefore~$\Xi$ is well-defined.

Now we prove that~$\Xi \colon Q_{-t}(M) \to Q_{+1}(M)$ is injective.
Let~$B$ and~$C$ be two~$(-t)$-quadratic forms in~$Q_{-t}(M)$, and suppose that~$Z:= (t^{-1}-1)(B-C) =0 \in Q_{+1}(M)$.
That is,~$Z \in \im(1-T_{+1})$, so, choosing a basis for~$M\cong \La^n$, there exists a matrix~$X$ over~$\Lambda$ such that
$$(t^{-1}-1)(B-C)=Z = X-\overline{X}^T.$$
We assert that there exists a matrix~$Y$ over~$\Lambda$ such that
\[Z = (t^{-1}-1)Y-(t-1)\overline{Y}^T.\]
Write the entries of the matrix~$Z= X-\overline{X}^T$ as~$z_{ij}$.
Since~$X-\overline{X}^T$ is skew-Hermitian, we have~$\overline{z_{ij}}=-z_{ji}$.
Also~$z_{ij}$ is divisible by~$t^{-1}-1$ for every~$i,j$. For~$i<j$, define the entries~$y_{ij}$ of~$Y$ by the formula~$y_{ij}:=z_{ij}/(t^{-1}-1)$. For~$i>j$ define~$y_{ij}:=0$.
This way~$(t^{-1}-1)y_{ij}-(t-1)\overline{y_{ji}}=z_{ij}$ for~$i \neq j$. It remains to define the diagonal entries of~$Y$.
For each~$i$, since the Laurent polynomial~$z_{ii}$ satisfies~$z_{ii}=-\overline{z}_{ii}$, we have~$z_{ii}(-1)= -z_{ii}(-1)$, and so~$z_{ii}(-1)=0$.
Therefore~$z_{ii}$ is divisible by~$t+1$.
 Since~$t^{-1}-1$ and~$(t+1)$ are coprime, we deduce that~$z_{ii}=(t+1)(t^{-1}-1)q$ for some polynomial~$q$ that satisfies~$q=\overline{q}$.
Now set~$y_{ii}:=q$. We compute:
\[(t^{-1}-1)y_{ii}-(t-1)\overline{y_{ii}}=(t^{-1}-1-(t-1))q=(t^{-1}-t)q= (t+1)(t^{-1}-1)q = z_{ii}.\]
This concludes the proof of the assertion that~$Z = (t^{-1}-1)Y-(t-1)\overline{Y}^T$ for some matrix~$Y$.
Using this assertion, as well as the definitions of~$X$ and~$Y$, we obtain:
\[Z= (t^{-1}-1)(B-C)=X-\overline{X}^T=(t^{-1}-1)Y-(t-1)\overline{Y}^T=(t^{-1}-1)(Y+t\overline{Y}^T).\]
This implies that~$B-C=Y+t\overline{Y}^T = Y-(-t)\ol{Y}^T$, which is zero in~$Q_{-t}(M)$.
Therefore the map~$\Xi\colon Q_{-t}(M) \to Q_{+1}(M)$ is injective as desired.
This completes the proof of~\eqref{item:Q-group-lemma-2}.



Next we prove~\eqref{item:Q-group-lemma-3}, that~$1+T_{+1} \colon Q_{+1}(M) \to Q^{+1}(M)$ is injective.
For an arbitrary finitely generated free module~$M$, the map~$1+T_{\eps}$  fits into the following exact sequence:
\[0 \to \wh{Q}^{-\varepsilon}(M) \to Q_{\eps}(M) \xrightarrow{1+T_{\eps}} Q^{\eps}(M).\]
The group,~$\wh{Q}^{-\varepsilon}(M) := \ker(1+T_{\eps})$ is a \emph{hyperquadratic~$Q$-group}; its purpose is to measure the difference between Hermitian and quadratic forms~\cite{RanickiExact}.
We have \begin{equation}\label{eqn:hyperquadratic}
\wh{Q}^{-\varepsilon}(M) = \ker(1+T_{\eps} \colon Q_{\eps}(M) \to Q^{\eps}(M)) = \frac{\ker(1+T_{\eps})}{\im(1-T_{\eps})} = \frac{\ker(1-T_{-\eps})}{\im(1+T_{-\eps})}.
\end{equation}
We shall use the penultimate description, but include the last equation to show why the~$-\eps$ appears in the notation for the group~$\wh{Q}^{-\varepsilon}(M)$.
Writing~$M=\La^n$, there is an isomorphism~$\wh{Q}^{-\eps}(M) =\wh{Q}^{-\varepsilon}(\La^n) \cong \wh{Q}^{-\varepsilon}(\La)^n$~\cite[Remark 3.4]{CrowleySixt}.
By \eqref{eqn:hyperquadratic} with~$\eps=+1$ and~$M=~\La$, we have~$\wh{Q}^{-1}(\La) =  \ker(1+T_{+1})/\im(1-T_{+1}).$
For~$p(t) \in \La$, if~$p(t)+p(t^{-1})=0$, then~$p(t)$ must be of the form~$p(t) = r(t)-r(t^{-1})$ for some~$r(t) \in \La$.
It follows that~$\wh{Q}^{-1}(\La)=0$.
Thus~$\ker(1+T_{+1} \colon Q_{+1}(M) \to Q^{+1}(M))=0$ and so~$1+T_{+1}$ is indeed injective.
%
 %
 %
This concludes the proof of the third item of the lemma.

For \eqref{item:Q-group-lemma-4}, let $\psi \in Q_{-t}(M)$. Then
\begin{align*}
(1+T_{+1})\circ \Xi(\psi) &= (1+T_{+1})(t^{-1}-1)\psi = (t^{-1}-1) \psi + (t-1)\overline{\psi}^T \\ &= (t^{-1}-1)(\psi - t \overline{\psi}^T) = \Omega \circ (1+T_{-t}) (\psi)
\end{align*}
so the diagram commutes, which proves \eqref{item:Q-group-lemma-4} and completes the proof that $(1+T_{-t}) \colon Q_{-t}(M) \to Q^{-t}(M)$ is injective.
\end{proof}

Next we state the algebraic cancellation result that we will use. The idea behind the proof of Theorem~\ref{thm:HyperbolicIntersectionForm} below will be to engineer a situation in which we can apply cancellation.

\begin{definition}
The \emph{Witt index}~$\operatorname{ind}(H,\theta)$ of an~$\varepsilon$-quadratic form is the largest integer~$k$ such that a subform of~$(H,\theta)$ is isometric to~$H_{\varepsilon}(R)^{\oplus k}$.
\end{definition}

\begin{proposition}
\label{prop:StablyHyperbolicImpliesHyperbolic}
Let~$\varepsilon \in \Lambda$ be a central unit with~$\eps\ol{\eps}=1$, and let~$(H,\theta),(H',\theta')$ be~$\varepsilon$-quadratic forms over~$\Lambda$.
Assume that for some~$n \geq 0$ there is an isometry
\begin{equation}
\label{eq:WantBass}
(H,\theta)\oplus H_\varepsilon(\Lambda)^{\oplus n} \cong (H',\theta') \oplus H_\varepsilon(\Lambda)^{\oplus n}.
\end{equation}
If~$\operatorname{ind}(H',\theta') \geq 3$, then there is an isometry
\begin{equation}
\label{eq:AppliedBass}
(H,\theta) \cong (H',\theta').
\end{equation}
\end{proposition}

\begin{proof}
We apply a result due to Bass~\cite[Corollary~IV.3.6]{Bass}.
Given a ring~$R$ with involution, we write~$R_0=\lbrace \sum_i x_i\overline{x}_i \mid x_i \in R \rbrace$ for the norm subring of~$R$, as well as~$\operatorname{maxspec}(R_0)$ for the set of all maximal ideals of~$R_0$ under the Zariski topology, and~$d_R:=\dim \operatorname{maxspec}(R_0)$.
Detailed definitions of these notions are irrelevant: we need only know that~$d_\Lambda=2$; see e.g.~\cite[Proposition 2.2]{Khan} or \cite[p.~439]{HambletonTeichner}.
Let~$(H,\theta),(H',\theta')$ be quadratic forms over~$R$. If there is an isometry
$$(H,\theta) \oplus H_{\varepsilon}(R)^{\oplus n} \cong (H',\theta') \oplus H_{\varepsilon}(R)^{\oplus n}$$ and if~$\operatorname{ind}((H',\theta') \oplus H_{\varepsilon}(R)) \geq d_{R} + 2$, then there is an isometry~$(H,\theta) \cong (H',\theta')$~\cite[Corollary~IV.3.6]{Bass}.
In particular, this cancellation result holds if~$\operatorname{ind}(H',\theta') \geq d_{R}+1$.
As mentioned above, for~$R=\Lambda$, we have~$d_\Lambda= 2$.   Since
\[\operatorname{ind}((H',\theta') \oplus H_{\varepsilon}(\Lambda)) \geq 3+1=d_\Lambda+2,\]
the result of Bass says that~\eqref{eq:WantBass} implies~\eqref{eq:AppliedBass}.
%
%
\end{proof}

The next theorem is the main result of this subsection.

\begin{theorem}
\label{thm:HyperbolicIntersectionForm}
If~$\Sigma \subseteq D^4$ is a $\Z$-surface of genus $g \geq 3$ whose boundary is an Alexander polynomial one knot~$K \subseteq S^3$, then
$$\lambda_{D^4_\Sigma} \cong \mathcal{H}_2^{\oplus g}.$$
\end{theorem}

\begin{proof}
We outline the strategy of the proof.
Since the knot~$K$ has Alexander polynomial one and $Q_{D^4}=0$, by Corollary~\ref{cor:IntersectionFormAlexanderPolynomial1}, there exists an integer~$n \geq 0$ and an isometry
\begin{equation}
\label{eq:StableIsometryNoCancel2}
\lambda_{D^4_\Sigma} \oplus \mathcal{H}_2^{\oplus n} \cong \mathcal{H}_2^{\oplus g+n}.
\end{equation}
We cannot apply the cancellation result of Proposition~\ref{prop:StablyHyperbolicImpliesHyperbolic} directly, since~$\mathcal{H}_2^{\oplus n}$ is not hyperbolic over~$\Lambda$.  However, by switching the basis elements $\mathcal{H}_2$ is isometric to a multiple of the $(-t)$-hyperbolic form:
\[\begin{pmatrix} 0 & t^{-1}-1 \\ t-1 & 0 \end{pmatrix} = (t^{-1}-1) \cdot \begin{pmatrix} 0 & 1 \\ -t & 0 \end{pmatrix} = (t^{-1}-1) \mathcal{H}^{-t}(\Lambda).\]
If we could remove the $(t^{-1}-1)$ factor, we could apply Proposition~\ref{prop:StablyHyperbolicImpliesHyperbolic}.
The idea for this is to use the relative intersection pairing
$$\lambda_{D^4_\Sigma}^\partial \colon H_2(D^4_\Sigma;\Lambda) \times H_2(D^4_\Sigma,\partial D^4_\Sigma;\Lambda) \to \Lambda$$
instead of the ``absolute'' intersection pairing~$\lambda_{D^4_\Sigma}$.
We will see, using~\eqref{eq:StableIsometryNoCancel2} and  appropriate bases, that the pairing~$\lambda^\partial_{D^4_\Sigma}$ is represented by a~$(-t)$-Hermitian matrix over~$\Lambda$.
We will then apply the cancellation result of Proposition~\ref{prop:StablyHyperbolicImpliesHyperbolic} with~$\varepsilon = -t$ to these forms, before deducing the desired conclusion on the original absolute pairings.
Note that the adjoints of the relative and absolute pairings fit into the following commutative diagram:
\begin{equation}
\label{eq:Triangle}
\xymatrix@R0.5cm{
H_2(D^4_\Sigma;\Lambda) \ar[rr]^-{\widehat{\lambda}}\ar[rd]_-q && \overline{\Hom_\Lambda(H_2(D^4_\Sigma;\Lambda),\Lambda)} \\
&H_2(D^4_\Sigma,\partial D^4_\Sigma;\Lambda) \ar[ru]_-{\widehat{\lambda}^\partial}.&
}
\end{equation}
For the first step of the proof we will choose bases for~$H_2(D^4_\Sigma;\Lambda)$ and $H_2(D^4_\Sigma,\partial D^4_\Sigma;\Lambda)$ and describe the map $q$ with respect to these bases.
First a short rank computation.

\begin{claim}
\label{claim:EulerCharacComputation}
$H_2(D^4_\Sigma;\Lambda) \cong \Lambda^{2g}$.
\end{claim}
\begin{proof}
Since $H_2(D^4_\Sigma;\Lambda)$ is free (by Lemma~\ref{lem:Homology}), we need only show that its rank is $b_1(\Sigma)=2g$.
Using~$H_1(D^4_\Sigma)=\Z$, a Mayer-Vietoris argument shows that~$b_2(D^4_\Sigma)=b_1(\Sigma)$ as well as~$b_3(D^4_\Sigma)=~0$.
Since we also have~$b_4(D^4_\Sigma)=0$ and~$b_0(D^4_\Sigma)=1$, we get~$\chi(D^4_\Sigma)=b_1(\Sigma)$.
The Euler characteristic can also be computed with~$Q=\Q(t)$ coefficients.
By Lemma~\ref{lem:Homology}, we have~$b_i^{Q}(D^4_\Sigma)=0$ for~$i \neq 2$, and therefore~$\rk_\Lambda H_2(D^4_\Sigma;\La) = \chi^Q(D^4_\Sigma) = b_1(\Sigma)$ as asserted.
\end{proof}

\begin{claim}\label{claim:y_i-basis}
There is a choice of bases with respect to which the homomorphism \[q \colon H_2(D^4_\Sigma;\Lambda) \to H_2(D^4_\Sigma,\partial D^4_\Sigma;\Lambda)\] is represented by the matrix $(t-1)\id$.
\end{claim}
\begin{proof}
We
first base~$H_2(D^4_\Sigma,\partial D^4_\Sigma;\Lambda)$
using the fact that the Alexander polynomial of~$K$ is one.
We have an exact sequence as part of the long exact sequence of the pair~$(D^4_\Sigma,\partial D^4_\Sigma)$:
\[H_2(D^4_\Sigma;\La) \xrightarrow{q} H_2(D^4_\Sigma,\partial D^4_\Sigma;\La) \xrightarrow{\delta} H_1(\partial D^4_\Sigma;\La) \to 0.\]
Here~$H_1(D^4_\Sigma;\La) =0$ because~$\pi_1(D^4_\Sigma)\cong \Z$.
In particular, the map~$q$ presents the~$\La$-module~$H_1(\partial D^4_\Sigma;\La)$.
Use Lemma~\ref{lem:Homology}, Lemma~\ref{lem:AlexanderModuleSigmaS1}, Claim~\ref{claim:EulerCharacComputation}, and the fact that~$\Delta_K \doteq 1$ to deduce that
\[H_2(D^4_\Sigma;\La) \cong \Lambda^{2g} \cong H_2(D^4_\Sigma,\partial D^4_\Sigma;\La) \text{ and } H_1(\partial D^4_\Sigma;\La) \cong \bigoplus^{2g}_{i=1} \La/(t-1). \]
Choose a set of generators~$\{\gamma_k\} \subseteq H_1(\partial D^4_\Sigma;\La)$, represented in the form  \[g_k \times \{\pt\} \subseteq  \Sigma \times \{\pt\} \subseteq \Sigma \times S^1 \subseteq \partial D^4_\Sigma,\]
   where~$\{g_k\}$ is a symplectic basis of curves on~$\Sigma$.
Choose a basis~$\{x_i\}$ for~$H_2(D^4_\Sigma,\partial D^4_\Sigma;\La)$ such that~$\delta(x_i) = \gamma_i$ for~$1 \leq i \leq 2g$.
The homology classes~$x_k$ can be represented by embedded surfaces~$S_i \subseteq D^4_\Sigma$ with~$\partial S_i = \gamma_i$ for~$1 \leq i \leq 2g$~\cite[Section~10.3]{FriedlNagelOrsonPowell}.

Next, we base $H_2(D^4_\Sigma;\La)$.
For each~$i =1,\dots,2g$, define a closed surface representing a class~$y_i \in H_2(D^4_\Sigma;\La)$, as follows.  Consider the torus~$g_i \times S^1 = \partial \big(\ol{\nu} \Sigma |_{g_i}\big)$, and surger it along~$\gamma_i$ using the surface~$S_i$.
That is, remove an annular neighbourhood $g_i \times (p,p') \subseteq g_i \times S^1$ of~$\gamma_i = g_i \times \{\pt\}$, leaving $B:= g_i \times (S^1 \sm (p,p'))$.  Then glue $-S_i$ to $g_i \times \{p\}$ and a push-off $S_i'$ of $S_i$ to $g_i \times \{p'\}$.
Call the resulting surface~$T_i$, and set~$y_i := [T_i] \in H_2(D^4_\Sigma;\La)$ for~$1 \leq i \leq 2g$.
Then, for~$i=1,\ldots,2g$, in~$H_2(D^4_\Sigma,\partial D^4_\Sigma ;\La)$ we have
\begin{equation}
\label{eq:Descriptionq}
q(y_i) = q([T_i]) = [-S_i] + [B \cup S_i'] =  -x_i + t x_i = (t-1) x_i.
\end{equation}
Here the~$t$ arises because the annulus $B$ wraps around a meridian of $\Sigma$. This process is illustrated in Figure~\ref{fig:surgering-for-Ti}.

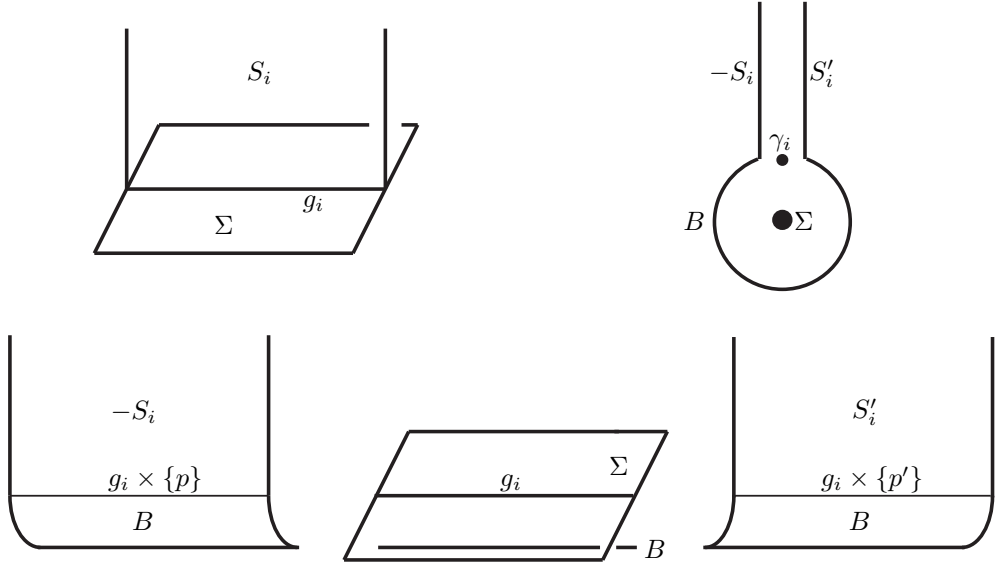
\begin{figure}[!htb]
\centering
\labellist
    \small
    \pinlabel {$\Sigma$} at 350 80
        \pinlabel {$g_i$} at 290 70
    \pinlabel {$S_i$} at 150 300
        \pinlabel {$\Sigma$} at 130 215
        \pinlabel {$g_i$} at 180 227
    \pinlabel {$-S_i$} at 80 110
    \pinlabel {$-S_i$} at 413 300
    \pinlabel {$S_i'$} at 463 300
    \pinlabel {$S_i'$} at 487 110
        \pinlabel {$\Sigma$} at 453 217
        \pinlabel {$\gamma_i$} at 440 261
    \pinlabel {$B$} at 85 49
    \pinlabel {$g_i \times \{p\}$} at 92 72
    \pinlabel {$g_i \times \{p'\}$} at 491 72
    \pinlabel {$B$} at 484 49
    \pinlabel {$B$} at 370 33
    \pinlabel {$B$} at 392 217
    \endlabellist
\includegraphics[scale=0.68]{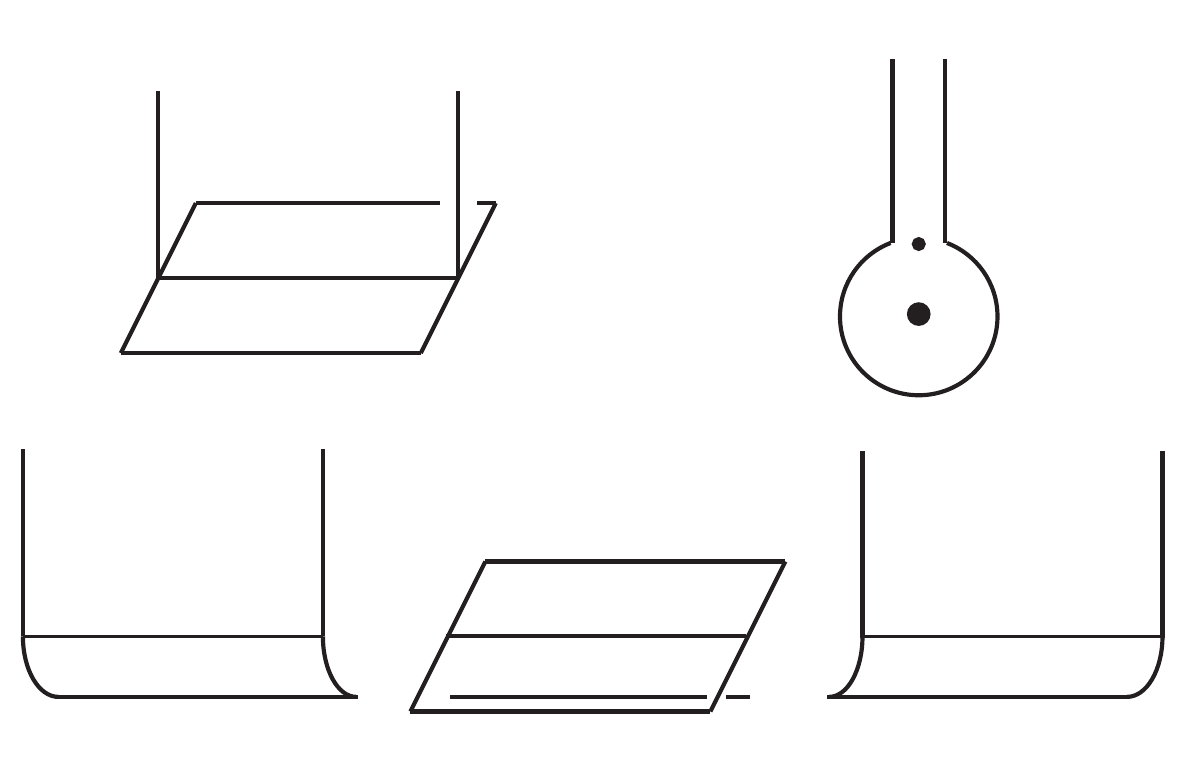}
\caption{Representing $y_i \in H_2(D^4_\Sigma;\Lambda)$ by an immersed surface $T_i$ in $D^4_{\Sigma}$ obtained by surgering the torus $g_i \times S^1$,  using a surface $S_i$ representing the relative homology class $x_i$. The top left picture shows a neighbourhood of a portion of $g_i$, intersected with a carefully chosen 3-dimensional subspace that contains the intersections of $\Sigma$ and $S_i$ with this neighbourhood.  Bottom, a 4-dimensional picture of $T_i$, consisting of three 3-dimensional slices at certain special values of the fourth coordinate (time), in which $\Sigma$, $-S_i$, $S_i'$, and most of $B$ appear. The annulus $B$ joins $-S_i$ and $S_i'$ together; in intermediate time values only a line of $B$ appears, including in the middle slice.  The top right picture shows a cross section of $T_i$, looking along $g_i$.}
\label{fig:surgering-for-Ti}
\end{figure}

Consider the following diagram of short exact sequences
$$
\xymatrix@R0.5cm{
0 \ar[r] & \Lambda^{\oplus 2g} \ar[r]^{(t-1)\id}\ar[d]^y& \Lambda^{\oplus 2g}  \ar[r]^-{\operatorname{proj}}\ar[d]^x&(\Lambda/(t-1))^{\oplus 2g}  \ar[r]\ar[d]^-\gamma& 0 \\
0 \ar[r] & H_2(D^4_\Sigma;\Lambda) \ar[r]^-q& H_2(D^4_\Sigma,\partial D^4_\Sigma;\Lambda) \ar[r]& H_1(\partial D^4_\Sigma;\Lambda) \ar[r]& 0,
}
$$
where the map $y$ (resp.\ $x$) sends the canonical basis $\lbrace e_i \rbrace_{i=1}^{2g}$ of $\Lambda^{2g}$ to $\lbrace y_i \rbrace_{i=1}^{2g}$ (resp.\ to~$\lbrace x_i \rbrace_{i=1}^{2g}$) and $\gamma$ maps the $[e_i]$ to the $\gamma_i$.
This diagram commutes thanks to~\eqref{eq:Descriptionq}.
By construction~$x$ and~$\gamma$ are isomorphisms and therefore so is~$y$ by the five lemma.
Thus $\lbrace y_i \rbrace_{i=1}^{2g}$ forms a basis of $H_2(D^4_\Sigma;\Lambda)$ and $q$ is represented by $(t-1)\id$.
This completes the proof of Claim~\ref{claim:y_i-basis}.
\end{proof}

Recall from~\eqref{eq:StableIsometryNoCancel2} that there is an isometry
$\alpha \colon\lambda_{D^4_{\Sigma}} \oplus \mathcal{H}_2^{\oplus n} \cong \mathcal{H}_2^{\oplus (g+n)}$.
Denote the matrices of~$ \lambda_{D^4_\Sigma}$ and~$ \lambda^\partial_{D^4_\Sigma}$ with respect to the previously described bases by~$L$ and~$L^\partial$ respectively. Also represent the isometry~$\alpha$ by a matrix~$A$ with respect to these bases.
Therefore we have a congruence
\begin{equation}
\label{eq:StableIso}
A^T \left( L \oplus \begin{pmatrix}
0&t-1 \\
t^{-1}-1&0
\end{pmatrix}^{\oplus n}  \right) \overline{A}=\begin{pmatrix}
0&t-1 \\
t^{-1}-1&0
\end{pmatrix}_.^{\oplus (g+n)}
\end{equation}
Our aim is now to factor out a~$(t^{-1}-1)$ from both sides of this equation.
Recall from~\eqref{eq:Triangle} that the adjoints of the pairings~$\lambda_{D^4_\Sigma}$ and~$\lambda_{D^4_\Sigma}^\partial$ are related by~$\widehat{\lambda}_{D^4_\Sigma}^\partial \circ q=\widehat{\lambda}_{D^4_\Sigma}.$
With respect to our bases (and recalling~\eqref{eq:Descriptionq}), for every~$x,y \in H_2(D^4_\Sigma;\Lambda)$, we therefore obtain
\begin{align*}
 x^TL^\partial  (t^{-1}-1)\overline{y}
&= x^TL^\partial \overline{(t-1) y}
=\lambda_{D^4_\Sigma}^\partial(x,q(y))
=\widehat{\lambda}^\partial(q(y))(x)\\
&=\widehat{\lambda}(y)(x)
=\lambda(x,y)
=x^TL\overline{y}.
\end{align*}
It follows that~$L^\partial (t^{-1}-1)=L$.
Combine this with~$\mathcal{H}_2=\bsm
0&t-1 \\
t^{-1}-1&0
\esm
=\bsm
0&-t \\
1&0
\esm
(t^{-1}-1)$ to rewrite~\eqref{eq:StableIso} as:
$$  A^T \left(L^\partial  (t^{-1}-1)\oplus \begin{pmatrix}
0&-t\\
1 & 0
\end{pmatrix}(t^{-1}-1)\right)\overline{A}=
\begin{pmatrix}
0&-t\\
1 & 0
\end{pmatrix}^{\oplus g+n}(t^{-1}-1).$$
As~$(t^{-1}-1)\cdot \id$ is diagonal with constant diagonal coefficients, it is central and thus we obtain
$$(t^{-1}-1) \cdot A^T\left(L^\partial \oplus \begin{pmatrix}
0&-t \\
1&0
\end{pmatrix}^{\oplus n}\right)\overline{A} =(t^{-1}-1)\begin{pmatrix}
0&-t \\
1&0
\end{pmatrix}^{\oplus(g+n)}_.$$
As~$(t^{-1}-1)\cdot \id$ is nondegenerate over~$\Lambda$, it follows that
\begin{equation}
\label{eq:SymmetricDesired}
A^T\left(L^\partial \oplus \begin{pmatrix}
0&-t \\
1&0
\end{pmatrix}^{\oplus n}\right) \overline{A}=\begin{pmatrix}
0&-t \\
1&0
\end{pmatrix}^{\oplus(g+n)}_.
\end{equation}
In particular~$L^\partial$ stabilises, via~$(-t)$-Hermitian matrices, to a~$(-t)$-Hermitian form, so it follows that~$L^\partial$ is itself~$(-t)$-Hermitian.

In order to apply cancellation, we refine~\eqref{eq:SymmetricDesired} to a statement about quadratic forms.
Set
$$\Psi:=\begin{pmatrix} 0&0 \\ 1&0 \end{pmatrix}.$$
Then $(1+T_{-t})\Psi = \bsm 0&-t \\ 1&0 \esm$, so $\Psi$ is a quadratic refinement for $\bsm 0&-t \\ 1&0 \esm$. Therefore $\Psi^{\oplus (g+n)}$ is a quadratic refinement for $\bsm 0&-t \\ 1&0 \esm^{\oplus (g+n)}$, which is the right hand side of~\eqref{eq:SymmetricDesired}. Thus $\Psi^{\oplus (g+n)}$ is a quadratic refinement for \[A^T\left(L^\partial \oplus \begin{pmatrix}
0&-t \\
1&0
\end{pmatrix}^{\oplus n}\right) \overline{A}.\]
 Rearranging,
we deduce that~$A^{-T}\Psi^{\oplus(g+n)}\overline{A}^{-1}$ is a quadratic refinement for the~$(-t)$-Hermitian matrix
$L^\partial \oplus \bsm 0&-t \\ 1&0 \esm^{\oplus n}$.
Now, corresponding to this decomposition,  write:
$$A^{-T}\Psi^{\oplus(g+n)}\overline{A}^{-1}=\begin{pmatrix}
\Psi^\partial &U \\ V &\Psi_{\operatorname{ind}}
\end{pmatrix}
$$
where this defined the block entries on the right hand side.
Since this is a quadratic refinement of the block diagonal~$(-t)$-Hermitian matrix~$L^\partial \oplus \bsm 0 & -t \\ 1 & 0 \esm^{\oplus n}$, we deduce that~$V-t\ol{U}^T=0$ and therefore
\begin{equation}\label{eqn:another-one}
A^{-T}\Psi^{\oplus(g+n)}\overline{A}^{-1}=\begin{pmatrix}
\Psi^\partial &0 \\ 0&\Psi_{\operatorname{ind}}
\end{pmatrix}+\begin{pmatrix}
0&U \\ t\ol{U}^T & 0
\end{pmatrix}=\Psi^\partial  \oplus \Psi_{\operatorname{ind}} \in Q_{-t}(\Lambda^{\oplus 2(g+n)}).
\end{equation}
This procedure therefore endows both the~$(-t)$-Hermitian matrices~$L^\partial$ and~$\bsm 0&-t \\ 1&0 \esm^{\oplus n}$ with quadratic refinements, which we denoted by~$\Psi^\partial$ and~$\Psi_{\operatorname{ind}}$ respectively.
Rearranging \eqref{eqn:another-one} we obtain the following isometry of quadratic forms:
 \begin{equation}
\label{eq:QuadraticNearlyDesired}
 A^T(\Psi^\partial \oplus \Psi_{\operatorname{ind}}) \overline{A}=\Psi^{\oplus (g+n)} \in Q_{-t}(\Lambda^{\oplus 2(g+n)}).
 \end{equation}
In order to apply cancellation, we need that~$\Psi_{\operatorname{ind}}$ agrees with the quadratic refinement~$\Psi^{\oplus n}$ of~$\bsm 0&-t \\ 1&0 \esm^{\oplus n}$ in the~$Q$-group~$Q_{-t}(\Lambda^{2n})$.
That both are quadratic refinements of the form $\bsm 0&-t \\ 1&0 \esm^{\oplus n}$ in $Q^{-t}(\Lambda^{2n})$ can be restated as
\[(1+T_{-t}) (\Psi_{\operatorname{ind}}) = (1+T_{-t}) (\Psi^{\oplus n}) \in  Q^{-t}(\Lambda^{2n}).\]
By Lemma~\ref{lem:Q-group-map}, $(1+T_{-t})$ is injective, and so $\Psi_{\operatorname{ind}} = \Psi^{\oplus n}$ in $Q_{-t}(\Lambda^{2n})$, as desired.
We deduce from~\eqref{eq:QuadraticNearlyDesired} that
\begin{equation}
\label{eq:QuadraticNearlyNearlyDesired}
 A^T\left(\Psi^\partial \oplus \Psi^{\oplus n}\right) \overline{A}=\Psi^{\oplus (g+n)} \in Q_{-t}(\Lambda^{2(g+n)}).
 \end{equation}
Since~$\bsm
   0&1 \\1 &0
\esm \bsm
  0 & 0\\1 &0
\esm \bsm
   0& 1\\1 & 0
\esm = \bsm
   0& 1\\0 & 0
\esm,$
we see that~$\bsm
0&0 \\
1&0
\esm$ is congruent to the hyperbolic~$(-t)$-quadratic form~$H_{-t}(\Lambda)$.
We have established the following isometry of~$(-t)$-quadratic forms:
\begin{equation}
\label{eq:ApplyCancellation}
\Psi^\partial \oplus H_{-t}(\Lambda)^{\oplus n}  \cong  H_{-t}(\Lambda)^{\oplus g} \oplus H_{-t}(\Lambda)^{\oplus n}  \in Q_{-t}(\Lambda^{2(g+n)}).
 \end{equation}
Since~$g \geq 3$, we now apply the cancellation result of Proposition~\ref{prop:StablyHyperbolicImpliesHyperbolic} to~\eqref{eq:ApplyCancellation} and symmetrise the result to deduce that~$ L^\partial~$ is congruent to~$H^{-t}(\Lambda)^{\oplus g}$, and therefore to~$\bsm 0&-t \\ 1&0 \esm^{\oplus g}.$
Multiplying both sides of this congruence by~$t^{-1}-1$ shows that ~$L$ is congruent to~$\mathcal{H}_2^{\oplus g}$.
This concludes the proof of Theorem~\ref{thm:HyperbolicIntersectionForm}.
\end{proof}

\subsection{Surfaces in $S^4$ and $D^4$.}
\label{sub:SurfacesS4D4}

We collect and prove our results concerning (rel.\ boundary) ambient isotopy of surfaces in $D^4$ and~$S^4$.
Since the proofs of these results were already outlined in the introduction or elsewhere, we proceed swiftly.
\begin{theorem}
\label{thm:D4Summary}
Let $\Sigma_0,\Sigma_1 \subseteq D^4$ be $\Z$-surfaces of genus $g$ with boundary $K$.
\begin{enumerate}
\item Suppose there is an isometry~$F \in \Iso(\lambda_{D^4_{\Sigma_0}},\lambda_{D^4_{\Sigma_1}})$ and write~$\partial F=h_K \oplus h_\Sigma$.
If $h_K$ is represented by an orientation-preserving homeomorphism that is the identity on $\partial E_K$, then~$\Sigma_0$ and~$\Sigma_1$ are topologically ambiently isotopic rel.\ boundary.
\item If $\Delta_K=1,g= 1,2$ and $\lambda_{D^4_{\Sigma_0}} \cong \lambda_{D^4_{\Sigma_1}}$, then~$\Sigma_0$ and~$\Sigma_1$ are topologically ambiently isotopic rel.\ boundary.
\item If $\Delta_K=1$ and $g \neq 1,2$, then~$\Sigma_0$ and~$\Sigma_1$ are topologically ambiently isotopic rel.\ boundary.
\end{enumerate}
\end{theorem}
\begin{proof}
The first assertion is the second item of Theorem~\ref{thm:WithBoundary}.
When $\Delta_K=1$, we have~$H_1(E_K;\Lambda)=0$, so $h_K=\id$ is automatic, and the second assertion follows from the first.
For the third assertion, additionally use that $\lambda_{D^4_{\Sigma_0}} \cong \lambda_{D^4_{\Sigma_1}}$ when $g \neq 1,2$ by Theorem~\ref{thm:HyperbolicIntersectionForm}.
\end{proof}

\begin{theorem}
Let $\Sigma_0,\Sigma_1 \subseteq S^4$ be closed $\Z$-surfaces of genus $g$.
\begin{enumerate}
\item For $g = 1,2$, if $\lambda_{\Sigma_0} \cong \lambda_{\Sigma_1}$, then $\Sigma_0$ and $\Sigma_1$ are topologically ambiently isotopic.
\item For $g \neq 1,2$, the surfaces $\Sigma_0$ and $\Sigma_1$ are topologically ambiently isotopic.
\end{enumerate}
\end{theorem}
\begin{proof}
Both assertions follow from Theorem~\ref{thm:D4Summary} by removing a $(D^4,D^2)$-pair from $(S^4,\Sigma_i)$.
\end{proof}

\subsection{Pushed-in Seifert surfaces for Alexander polynomial one knots.}
\label{sub:PushedIn}

We prove Theorem~\ref{thm:PushedInIntro} from the introduction, which states that any two pushed-in Seifert surfaces of the same genus for an Alexander polynomial one knot are topologically ambiently isotopic in $D^4$.
\medbreak
Recall that a \emph{sublagrangian} for a nonsingular $\varepsilon$-Hermitian form $(H,\lambda)$ is a direct summand $M \subseteq H$ such that $\lambda(M \times M)=0$ (i.e.\ $M \subseteq M^\perp$).
A sublagrangian is a \emph{Lagrangian} if~$M=M^\perp$.
For a nonsingular $\varepsilon$-quadratic form $(H,\lambda,\mu)$, the definitions are identical, with the additional requirement that~$\mu(M)=0$.
The following lemma is known to those familiar with L-theory. For instance, given a sublagrangian $i \colon M \hookrightarrow H$, the result follows quite promptly from~\cite[Proposition 2.2]{RanickiAlgebraicSurgery}, provided one assumes that $i^* \widehat{\lambda} \colon H \to M^*$ is surjective.
Since our argument is elementary (Ranicki's proof is more involved because his statement is more general) and since most of our work goes into establishing that $i^* \widehat{\lambda}$ is surjective, we include a proof for the readers' convenience.

\begin{lemma}
\label{lem:HalfRankLemma}
Let~$(H,\lambda)$ be a nonsingular~$\varepsilon$-Hermitian form on a free $\Lambda$-module~$H$.
If~$M \subseteq H$ is a half-rank sublagrangian, then~$M$ is a Lagrangian.
\end{lemma}

\begin{proof}
We know that~$M \subseteq M^\perp$ and must show that~$M=M^\perp$.
Equivalently, we must show that~$M/M^\perp=0$.
We will show that~$M/M^\perp$ is both torsion and torsion-free.
\begin{claim}
The inclusion~$M^\perp \hookrightarrow H$ is split and~$M^\perp \subseteq H$ is a half-rank summand.
\end{claim}
\begin{proof}
Consider the exact sequence~$0 \to M^\perp \to H \xrightarrow{\iota^* \widehat{\lambda}} M^*$.
It now suffices to show that~$\iota^* \widehat{\lambda}$ is surjective: since $M^*$ is free and~$\operatorname{rk}_\Lambda(M^*)=\operatorname{rk}_\Lambda(M)=\frac{1}{2}\operatorname{rk}_\Lambda(H)$, the result would then follow.
Since~$\lambda$ is nonsingular,~$\widehat{\lambda}$ is surjective and so it suffices to prove that~$\iota^*$ is injective.
This occurs if and only if~$\operatorname{Ext}_\Lambda^1(H/M,\Lambda)=0$.
But since~$M$ is a summand, it follows that~$H \cong M \oplus H/M$.
Since~$H$ is free,~$H/M$ is projective over $\Lambda$ and therefore free over $\Lambda$~\cite[Chapter V, Corollary 4.12]{LamSerre}.
Thus~$\operatorname{Ext}_\Lambda^1(H/M,\Lambda)=0$, so ~$\iota^* \widehat{\lambda}$ is surjective and thus~$M^\perp \hookrightarrow H$ is a split injection, as claimed.
\end{proof}

The fact that~$M/M^\perp$ is torsion follows from the claim because~$M^\perp \subseteq M$ and $\operatorname{rk}_\Lambda(M^\perp)=\operatorname{rk}_\Lambda(M)$.
We now show that~$M/M^\perp$ is torsion-free.
Since submodules of torsion-free modules are torsion-free, it suffices to show that~$H/M^\perp \supseteq M/M^\perp$ is torsion-free.
But now by the claim, we know that~$H/M^\perp$ is (isomorphic to) a submodule of the free module~$H \cong M^\perp \oplus H/M^\perp$ and so it is indeed torsion-free.
This concludes the proof of the lemma.
\end{proof}

Given a Seifert matrix~$A$ for a genus $g$ Seifert surface $S$, we write~$(H,A):=(\Z^{2g},A)$.
Use $F \subseteq D^4$ to denote the result of pushing $S$ into $D^4$.
After some choice of bases, the equivariant intersection form~$(H_2(D^4_F;\Lambda),\lambda_{D^4_F})$ is isometric to~$(H \otimes_\Z \Lambda, (1-t)A^T+(1-t^{-1})A)$~\cite[Section~3]{KoSeifert} (see also~\cite[proof of Lemma 5.4]{CochranOrrTeichnerStructure}).

\begin{remark}
\label{rem:NonSingular}
Observe that $(H \otimes_\Z \Lambda,A-tA^T,A)$ defines a $(-t)$-quadratic form.
Furthermore, if~$A$ is a Seifert matrix for an Alexander polynomial
one knot $K$, then this quadratic form is nonsingular since $ \det(A-tA^T)\doteq \Delta_K(t)  \doteq 1$.
\end{remark}

Recall that a \emph{metaboliser} for a Seifert matrix~$(\Z^{2g},A)$ is a half-rank direct summand~$L \subseteq \Z^{2g}$ such that~$x^TAy=0$ for all~$x,y \in L$.
Since Alexander polynomial one knots are (in particular) algebraically slice, their Seifert matrices admit metabolisers.

\begin{proposition}
\label{prop:Lagrangian}
Let~$(H,A):=(\Z^{2g},A)$ be a Seifert matrix for an Alexander polynomial one knot~$K$.
If~$L \subseteq H$ is a metaboliser for~$(H,A)$, then~$L \otimes_\Z \Lambda$ is a Lagrangian for~$(H \otimes_\Z \Lambda,A-tA^T,A)$ and so this latter $(-t)$-quadratic form is hyperbolic.
\end{proposition}

\begin{proof}
We verify the three points in the definition of a Lagrangian.
Since~$A$ vanishes on~$L$, the extended form (also denoted~$A$) vanishes on~$M:=L \otimes_\Z \Lambda$.
It follows that the $(-t)$-Hermitian form $A-tA^T$ vanishes on $M \times M$.
Since~$L \subseteq H$ is a summand, so is~$M \subseteq H \otimes_\Z \Lambda$.
Thus~$M \subseteq (H \otimes_\Z \Lambda,A-tA^T)$ is a sublagrangian of a~$(-t)$-Hermitian form, with~$H \otimes_\Z \Lambda$ a free module.
We showed in Lemma~\ref{lem:HalfRankLemma} that since~$M$ is half-rank, this forces~$M^\perp =M$.

The last sentence of the proposition follows because it is known that if a nonsingular~$\varepsilon$-quadratic form~$(H,\lambda,\mu)$ admits a Lagrangian,
 then it is isometric to the standard hyperbolic form~\cite[Proposition 2.2]{RanickiAlgebraicSurgery}.
\end{proof}

\begin{proposition}
\label{prop:Hyperbolic}
If $A$ is a Seifert matrix for an Alexander polynomial one knot~$K$, then
$$ (H \otimes_\Z \Lambda,(1-t)A^T+(1-t^{-1})A) \cong \mathcal{H}_2^{\oplus g}.$$
\end{proposition}

\begin{proof}
Using Proposition~\ref{prop:Lagrangian}, we deduce the following sequence of isometries:
\begin{align*} (H \otimes_\Z \Lambda,A-tA^T,A) \cong H_{-t}(\Lambda)^{\oplus g}
&=
\left(\Lambda^{2g}, \begin{pmatrix}
0&1 \\
-t&0
\end{pmatrix}^{\oplus g},
 \begin{pmatrix}
0&1 \\
0&0
\end{pmatrix}^{\oplus g}
\right) \\
&\cong
\left(\Lambda^{2g}, \begin{pmatrix}
0&-t \\
1&0
\end{pmatrix}^{\oplus g},
 \begin{pmatrix}
0&0 \\
1&0
\end{pmatrix}^{\oplus g}
\right).
\end{align*}
Multiplying both sides by~$(1-t^{-1})$ then gives the assertion.
\end{proof}

We are now ready to prove Theorem~\ref{thm:PushedInIntro} from the introduction.

\begin{theorem}
\label{thm:PushedIn}
If~$F_0,F_1 \subseteq D^4$ are genus~$g$ pushed-in Seifert surfaces for an Alexander polynomial one knot~$K$, then they are topologically ambiently isotopic rel.\ boundary.
\end{theorem}

\begin{proof}
Let $A_0$ and $A_1$ be Seifert matrices for $F_0$ and $F_1$.
As mentioned above, the equivariant intersection form~$(H_2(D^4_{F_i};\Lambda),\lambda_{D^4_{F_i}})$ is isometric to~$(H \otimes_\Z \Lambda, (1-t)A_i^T+(1-t^{-1})A_i)$.
By Proposition~\ref{prop:Hyperbolic}, both forms are isometric to $\mathcal{H}_2^{\oplus g}$.
Since $K$ is an Alexander polynomial one knot, the result now follows from Theorem~\ref{thm:D4Summary}.
\end{proof}

\section{Rim surgery on surfaces with knot group~$\Z$.}
\label{sec:RimSurgery}

We prove Theorems~\ref{thm:RimSurgeryClosedIntro} and~\ref{thm:RimSurgeryIntro} which concern rim surgery on surfaces in~$4$-manifolds whose knot group is~$\Z$.
In Subsection~\ref{sub:RimSurgeryDef}, we review the definition of rim surgery, in Subsection~\ref{sub:RimSurgeryIntersectionForm}, and in Subsection~\ref{sub:RimSurgeryProof}, we prove the main results.

\subsection{Knot surgery and twist rim surgery}
\label{sub:RimSurgeryDef}
We review some facts about knot surgery and rim surgery. References include~\cite{FintushelStern,KimHJ, KimRuberman, BaykurSunukjian}.
\medbreak
Let~$Z$ be a compact, oriented~$4$-manifold containing a locally flat embedded torus~$T$ with trivial normal bundle, and let~$J \subseteq S^3$ be a knot.
Use~$\mu_T$ to denote the meridian of~$T\subseteq Z$ and~$\mu_J,\lambda_J$ for the meridian and~$0$-framed longitude of~$J$.
Let~$\varphi \colon \partial \nu (T) \to S^1 \times \partial E_J$ be any diffeomorphism such that~$\varphi_*(\mu_T)=\lambda_J$.
The~$4$-manifold obtained by \emph{knot surgery} along~$J$ and~$\varphi$ is defined as
$$ Z_J(\varphi)=(Z \setminus \nu(T) )\cup_\varphi (S^1 \times E_J).$$
Given a compact~$4$-manifold~$W$ and a locally flat embedded orientable surface~$\Sigma \subseteq W$, we assume either that~$W$ and~$\Sigma$ are closed or that~$\Sigma \subseteq W$ is a properly embedded.
As we now describe, rim surgery arises from a particular type of knot surgery on the surface exterior~$W_\Sigma :=W \setminus \nu(\Sigma)$.
Choose a simple closed curve~$\alpha \subseteq \Sigma$ and a trivialisation of the normal bundles over~$\alpha$ and~$\Sigma$ so that~$(\nu (\Sigma),\nu (\alpha))=(\Sigma \times D^2,\alpha \times I \times D^2)$; in other words the normal bundle of~$\alpha$ inside~$\Sigma$ is~$\alpha \times I$.
This way, it is understood that~$\mu_\Sigma=\lbrace \operatorname{pt} \rbrace \times \partial D^2$, and the \emph{rim torus}~$T$ is~$\partial (\nu(\Sigma)|_\alpha)=\alpha \times \mu_\Sigma$.
The rim torus is framed: if we write~$\nu(\Sigma)|_{\alpha \times I}=\alpha \times I \times D^2$, then a framing of~$\nu T$ is given by the $I$-direction and the radial direction in the polar coordinates of $D^2$.

In order to perform a knot surgery on~$W_\Sigma$ along~$J$ using the rim torus~$T$, 
we consider the homeomorphism~$\varphi \colon \partial \overline{\nu}(T) \to S^1 \times \partial E_J$ determined by
\begin{align*}
\varphi(\alpha)&=\mu_J + S^1 \\
\varphi(\mu_\Sigma)&=\mu_J, \\
\varphi(\mu_T)&=\lambda_J.
\end{align*}
Gluing~$\Sigma \times D^2$ back into the result of this knot surgery produces a manifold~$W_J(\varphi)$ that is homeomorphic to $W$, and in fact if $W$ is smooth then $W_J(\varphi)$ is diffeomorphic to~$W$; see~\cite[Lemma 2.4]{KimHJ}, in which a specific diffeomorphism is constructed.
Thus, we obtain a new embedding~$\Sigma_n^m(\alpha,J) \subseteq~W$.

\begin{definition}
\label{def:RimSurgery}
Let~$\alpha \subseteq \Sigma$ be a simple closed curve, and let~$J$ be a knot.
The \emph{1-twist} rim surgery of a locally flat, properly embedded, orientable surface~$\Sigma \subseteq W$ is the image~$\Sigma(\alpha,J)$ of~$\Sigma$ under the homeomorphism~$W \cong W_J(\varphi)$ mentioned above.
If~$W$ is smooth and~$\Sigma$ is smoothly embedded, then since $W \cong W_J(\varphi)$ is a diffeomorphism, $\Sigma(\alpha,J)$ is smoothly embedded.
\end{definition}

The exterior~$W_{\Sigma(\alpha,J)}$ of the rim surgery~$\Sigma(\alpha,J)$ is the knot surgery on~$W_\Sigma$ along the homeomorphism~$\varphi$.
The proof of the following lemma can be found in~\cite[Proposition~3.3]{KimHJ} and~\cite[Proposition 2.5]{KimRuberman-2}.

\begin{lemma}
\label{lem:HomologyEquivalence}
For a compact~$4$-manifold~$W$ and a locally flat, embedded, orientable surface~$\Sigma \subseteq~W$, we assume either that~$W$ and~$\Sigma$ are closed or that~$\Sigma \subseteq W$ is properly embedded.
Let~$\alpha \subseteq \Sigma$ be a simple closed curve,  and let~$J$ be a knot.
If~$\pi_1(W_\Sigma)=\Z$, then
$$\pi_1(W_{\Sigma(\alpha,J)})=\Z.$$
\end{lemma}

\begin{remark}\label{remark:U-rim-surgery-does-nothing}
We note that if $J=U$ is the unknot then $\Sigma(\alpha,U)= \Sigma$, by \cite[Lemma~2.2]{KimRuberman}.
\end{remark}

\subsection{The \texorpdfstring{$\Lambda$-}{equivariant }intersection form}
\label{sub:RimSurgeryIntersectionForm}
Now we focus on the case where the ambient manifold~$N$ has a boundary and show that the~$\Lambda$-intersection form of~$N_{\Sigma(\alpha,J)}$ agrees with that of~$N_\Sigma$.
\medbreak

We start with an observation concerning the rim torus.
\begin{lemma}
\label{lem:RimTorus}
Let $\Sigma \subseteq N=X \setminus \mathring{D}^4$ be a $\Z$-surface.
The exterior of a rim torus~$T$ satisfies
$$H_1(N_\Sigma \setminus \nu(T))=\Z\la \mu_\Sigma \ra \oplus \la \Z\mu_T \ra.$$
\end{lemma}

\begin{proof}
Consider the Mayer-Vietoris sequence of~$N_\Sigma=(N_\Sigma \setminus \nu(T)) \cup \overline{\nu}(T)$:
\begin{equation}
\label{eq:MV}
\cdots \to H_2(N_\Sigma) \xrightarrow{\partial} H_1(T \times S^1) \xrightarrow{\theta} H_1(\overline{\nu}(T)) \oplus H_1(N_\Sigma \setminus \nu(T)) \to H_1(N_\Sigma) \to 0.
\end{equation}
We show that~$\partial=0$.
Note that \[H_1(T \times S^1) \cong H_1(T) \oplus H_1(S^1) \xrightarrow{\theta} H_1(\overline{\nu}(T)) \cong H_1(T)\]
is projection onto the first factor, so for~$x \neq 0$ in~$H_1(T)$,~$\theta(x,n) \neq 0$.  Therefore in the image of~$\partial$, the first coordinate vanishes.
Now for a closed surface~$S \subseteq N$, the definition of the connecting homomorphism implies that \[\partial([S])=(0,Q_N([T],[S])) \in H_1(T \times S^1) \cong H_1(T) \oplus H_1(S^1).\]
Since~$T$ is isotopic to a torus in~$\Sigma \times S^1 \subseteq \partial N_\Sigma$, we deduce that~$Q_N([T],[S])=0$. It follows that~$\partial=0$ as asserted.
The Mayer-Vietoris sequence in~\eqref{eq:MV} then reduces to
$$ 0 \to  \Z\la \mu_T \ra \to  H_1(N_\Sigma \setminus \nu(T)) \to H_1(N_\Sigma) \to 0.$$
Since~$H_1(N_\Sigma)=\Z\la \mu_\Sigma \ra$ is free, this sequence splits, establishing the lemma.
\end{proof}

Next, we extend the coefficient system on~$\pi_1(N_\Sigma)$ over~$\pi_1(N_{\Sigma(\alpha,J)})$.
\begin{remark}
\label{rem:CoefficientSystem}
The map~$\pi_1(N_\Sigma) \to \Z$ restricts to a map~$\pi_1(N_\Sigma\setminus\nu(T)) \to \Z$ that sends~$\mu_\Sigma$ to~$1$ and~$\mu_T$ to~$0$ (recall from Lemma~\ref{lem:RimTorus} that $H_1(N_\Sigma\setminus\nu(T)) \cong \Z\langle \mu_\Sigma \rangle \oplus \Z\langle \mu_T \rangle)$.
Consider the map~$\pi_1(S^1 \times E_J) \to \Z$ that sends the class of the~$S^1$-factor to~$-1$ and that is the abelianisation on~$\pi_1(E_J)$.
We verify that these maps extend to a map~$\pi_1(N_{\Sigma(\alpha,J)}) \to~\Z$.
We are identifying the~$3$-torus~$\partial \overline{\nu}(T) \subseteq \partial N_\Sigma\setminus\nu(T)$ with~$S^1 \times \partial E_J$ via the homeomorphism~$\varphi$.
We have the identifications~$[\alpha] \sim [S^1][\mu_J], [\mu_\Sigma] \sim [\mu_J]$ and~$[\mu_T] \sim [\lambda_J]$.
The loop~$\alpha \subseteq \Sigma$ is mapped to zero under~$\pi_1(N_\Sigma) \to \Z$; the same is true for~$[S^1][\mu_J]$ under the map~$\pi_1(S^1 \times E_J) \to \Z$.
Similarly, both~$\mu_\Sigma$ and~$\mu_J$ (resp~$\mu_T$ and~$\lambda_J$) are sent to~$1$ (resp.~$0$) under the respective maps to~$\Z$.
\end{remark}

Arguing as in~\cite[Proof~of~Lemma~2.3]{KimRuberman}, we construct a degree one map \[\Psi \colon N_{\Sigma(\alpha,J)} \to~N_\Sigma.\]
Using obstruction theory, one can construct a degree one map~$E_J \to E_U$, where $U \subseteq S^3$ is the unknot, that takes meridian to meridian and longitude to longitude.
Cross it with the identity to obtain a degree one map~$S^1 \times E_J \to S^1 \times E_U$.
This can be glued to the identity map~$\id_{N_\Sigma \setminus \nu(T)}$ to obtain a degree one map~$N_{\Sigma(\alpha,J)} \to N_{\Sigma_n(\alpha,U)}$.
By~\cite[Lemma 2.2]{KimRuberman}, there is a homeomorphism~$N_{\Sigma_n(\alpha,U)} \cong N_\Sigma$ and we have therefore obtained a degree one map
$$ \Psi \colon N_{\Sigma(\alpha,J)} \to N_{\Sigma(\alpha,U)} \cong N_\Sigma.$$
By construction,~$\Psi$ is the identity on~$\partial N_\Sigma$.
Furthermore, since the fundamental groups of both~$N_{\Sigma(\alpha,J)}~$ and~$N_\Sigma$ are generated by the meridians of the surface,~$\Psi$ induces an isomorphism on fundamental groups: indeed by construction~$\Psi$ takes meridian to meridian.
In particular~$\Psi$ induces homomorphisms on the~$\Lambda$-homology groups.

In order to prove that $\Psi$ in fact induces an isometry of the equivariant intersection forms, we compute the~$\Lambda$-homology of~$S^1 \times E_J$.

\begin{lemma}
\label{lem:KunnethSS}
Let~$ J\subseteq S^3$ be a knot.
With respect to the coefficient system on~$S^1 \times E_J$ introduced in Remark~\ref{rem:CoefficientSystem}, we have~$H_2(S^1 \times E_J;\Lambda)=0$ and~$H_1(S^1 \times E_J;\Lambda)=\Z$.
Additionally, the degree one map $\Psi$ induces an isometry
$$\Psi_* \colon ( H_2(N_{\Sigma(\alpha,J)};\Lambda),\lambda_{N_{\Sigma(\alpha,J)}} )  \xrightarrow{\cong} ( H_2(N_\Sigma;\Lambda),\lambda_{N_\Sigma} ).$$
\end{lemma}

\begin{proof}
We use the K\"unneth spectral sequence with \[E^2_{p,q}=\bigoplus_{q_1+q_2=q}\operatorname{Tor}_p(H_{q_1}(S^1;\Lambda),H_{q_2}(E_J;\Lambda))\] and which converges to~$H_*(S^1 \times E_J;\Lambda \otimes_\Lambda \Lambda) \cong H_*(S^1 \times E_J;\Lambda)$, where $\Lambda \otimes_\Lambda \Lambda$ is a right $\Z[\pi_1(E_K) \times \pi_1(S^1)]$-module via the diagonal action induced by $p \otimes q \cdot (e,s)=pe \otimes qs$ for $p,q\in \Lambda$ and $(e,s) \in \pi_1(E_K) \times \pi_1(S^1)$.
Write~$\varepsilon \colon \Lambda \to \Z,p(t) \mapsto p(1)$ for the augmentation map and~$\Z_\varepsilon$ for the resulting~$\Lambda$-module structure on~$\Z$.
By definition of the coefficient system in Remark~\ref{rem:CoefficientSystem}, we have~$H_1(S^1;\Lambda)=0$ and~$H_0(S^1;\Lambda)=\Lambda/(t-1)=\Z_\varepsilon$.
We also have~$H_0(E_J;\Lambda)=\Z$ and~$H_i(E_J;\Lambda)=0$ for~$i \geq 2$.
This implies that~$E^2_{0,2}=0$.
In fact, we also deduce that~$E_{1,0}^2=\Z$ and~$E_{2,0}^2=0$ because
$$E_{i,0}^2=\operatorname{Tor}_i^{\Z[\Z]}(\Z,\Z)=H_i(\Z;\Z)=H_i(S^1).$$
We claim that~$\operatorname{Tor}_i^\Lambda(\Z_\varepsilon,H_1(E_J;\Lambda))=0$ for~$i=0,1$.
These Tor groups are computed as the homology of the complex obtained tensoring the resolution~$0 \to \Lambda \xrightarrow{t-1} \Lambda \to \Z_\varepsilon \to 0~$ with the Alexander module~$H_1(E_J;\Lambda)$.
Since multiplication by~$t-1$ induces an isomorphism on the Alexander module~\cite[Proposition 1.2]{LevineKnotModules}, these Tor groups vanish and the claim is proved.

Using the claim, we deduce that~$E^2_{0,1}=0$ and~$E^2_{1,1}=0$.
It follows that~$d_{2,0}=0$ and now the first assertion is a consequence of a standard spectral sequence computation.

We now assert that $\Psi$ induces an isomorphism on $H_2(-;\Lambda)$.
Recall the decompositions $N_{\Sigma(\alpha,J)}=(N_\Sigma \setminus \nu(T)) \cup_{\varphi} (S^1 \times E_J)$ and $N_\Sigma \cong N_{\Sigma_n(\alpha,U)}= (N_\Sigma \setminus \nu(T)) \cup_{\varphi} (S^1 \times E_U)$.
By construction,~$\Psi$ restricts to the identity on $N_\Sigma \setminus \nu(T)$ and $\partial \overline{\nu}(T)$ and therefore induces the identity on the $\Lambda$-homology of these spaces.
Next, $\Psi_* \colon H_1(S^1 \times E_J;\Lambda) \to H_1(S^1 \times E_U;\Lambda)$ is also an isomorphism: indeed, the first assertion shows that both modules are isomorphic to $\Z=\Lambda/(t-1)$, and this term comes from $H_0(E_J;\Lambda)=\Z_\varepsilon$, on which $\Psi$ does indeed induce an isomorphism.
Since we showed that $H_2(S^1 \times E_J;\Lambda)=0$, the assertion now follows from the five lemma applied to the following commutative diagram, where $\Lambda$-coefficients are understood:
$$
\xymatrix@R0.5cm@C0.5cm{
H_2(\partial \overline{\nu}(T)) \ar[r]\ar[d]^=_{\Psi_*}& H_2(N_\Sigma \setminus \nu(T)) \ar[r]\ar[d]_{\Psi_*}^=& H_2(N_{\Sigma(\alpha,J)}) \ar[r]\ar[d]_{\Psi_*}& H_1(\partial \overline{\nu}(T)) \ar[r]\ar[d]_{\Psi_*}^=& H_1(N_\Sigma\setminus \nu(T)) \oplus \Z \ar[r]\ar[d]_{\Psi_*}^\cong& 0 \\
H_2(\partial \overline{\nu}(T)) \ar[r]& H_2(N_\Sigma \setminus \nu(T)) \ar[r]& H_2(N_\Sigma) \ar[r]& H_1(\partial \overline{\nu}(T)) \ar[r]& H_1(N_\Sigma\setminus \nu(T)) \oplus \Z \ar[r]& 0.
}
$$
Since $\Psi$ induces a $\Lambda$-isomorphism on $H_2(-;\Lambda)$, it induces one on $H_2(-;\Lambda)^*$, and therefore on $H^2(-;\Lambda)$ (recall the UCSS argument from Lemma~\ref{lem:Homology}) and thus on the second relative $\Lambda$-homology groups (because $\operatorname{deg}(\Psi)=1$).
We conclude that $\Psi_*$ in fact induces an isometry of the $\Lambda$-intersection forms.
\end{proof}

\subsection{Topological triviality of rim surgery on surfaces with knot group~$\Z$}
\label{sub:RimSurgeryProof}
Next we prove Theorem~\ref{thm:RimSurgeryIntro} from the introduction.

\begin{theorem}
\label{thm:RimSurgery}
Let~$\Sigma \subseteq N=X \setminus \mathring{D}^4$ be a $\Z$-surface, let~$\alpha \subseteq \Sigma$ be a simple closed curve and let~$J$ be a knot,
There is a rel.\ boundary orientation-preserving homeomorphism of pairs
$$\Phi \colon (N,\Sigma(\alpha,J))  \xrightarrow{\cong} (N,\Sigma).$$
that induces the same isometry as the degree one map $\Psi$ on the equivariant intersection forms of the surface exteriors.
If~$N=D^4$, then the surfaces are topologically isotopic rel.\ boundary.
\end{theorem}

\begin{proof}
By Lemma~\ref{lem:HomologyEquivalence} both surfaces have knot group~$\Z$.
Lemma~\ref{lem:KunnethSS} states that
the degree one map $\Psi \colon N_{\Sigma(\alpha,J)} \to N_\Sigma$ induces
an isometry~$\Psi_* \colon \lambda_{N_{\Sigma(\alpha,J)}} \cong \lambda_{N_{\Sigma}}$.
As~$\Psi$ is the identity on~$\partial N_\Sigma$, we see that~$\partial \Psi$ induces the identity isometry between the boundary Blanchfield forms.
The result now follows from Theorem~\ref{thm:WithBoundary}.
\end{proof}

The remaining result we owe a proof of is Theorem~\ref{thm:RimSurgeryClosedIntro} from the introduction.
To prove this result, which involves obtaining an isotopy between closed surfaces, we need to control the map $F_\Z$ discussed in Lemma~\ref{lem:InducedIsometryOfH2(X)}.
This will rely on the following result.

\begin{lemma}
\label{lem:RimLemma2}
Let $\Sigma \subseteq N$ be a $\Z$-surface, let~$\alpha \subseteq \Sigma$ be a simple closed curve with associated rim torus~$T$, and let~$J$ be a knot.
The following sequences of inclusion induced maps are exact:
\begin{align*}
&0 \to \Z\la \alpha \times \mu_T\ra \to H_2(N_\Sigma\setminus\nu(T);\Lambda) \to H_2(N_{\Sigma(\alpha,J)};\Lambda) \to 0, \\
&0 \to \Z\la \alpha \times \mu_T\ra \to H_2(N_\Sigma\setminus\nu(T);\Lambda) \to H_2(N_{\Sigma};\Lambda) \to 0.
\end{align*}
\end{lemma}
\begin{proof}
As we noted in Lemma~\ref{lem:HomologyEquivalence} that~$\pi_1(N_{\Sigma(\alpha,J)})=\Z$ and since~$N_{\Sigma(\alpha,J)}$ has ribbon boundary, we know that~$H_1(N_{\Sigma(\alpha,J)};\Lambda)=0$ as well as~$H_3(N_{\Sigma(\alpha,J)};\Lambda)=0$, again by Lemma~\ref{lem:Homology}.
Consequently, the Mayer-Vietoris sequence for~$N_{\Sigma(\alpha,J)}=N_\Sigma\setminus\nu(T) \cup_{\varphi} (S^1 \times E_J)$ with~$\Lambda$ coefficients gives
\begin{align*}
 0 &\to H_2(\partial \overline{\nu}(T);\Lambda) \to H_2(N_\Sigma\setminus\nu(T);\Lambda) \oplus H_2(S^1 \times E_J;\Lambda) \to H_2(N_{\Sigma(\alpha,J)};\Lambda)  \\
& \to H_1(\partial \overline{\nu}(T);\Lambda) \to H_1(N_\Sigma\setminus\nu(T);\Lambda) \oplus H_1(S^1 \times E_J;\Lambda) \to 0.
\end{align*}
Since there is a homotopy equivalence~$\overline{\nu}(T) \simeq \alpha \times \mu_\Sigma$, the discussion of coefficient systems from Remark~\ref{rem:CoefficientSystem} implies that~$H_i(\partial \overline{\nu}(T);\Lambda)=H_i(\alpha \times \mu_\Sigma \times \mu_T;\Lambda)=H_i(\alpha \times \mu_T)$ for~$i=1,2$.
We saw in Lemma~\ref{lem:KunnethSS} that~$H_2(E_J \times S^1;\Lambda)=0$ and~$H_1(E_J \times S^1;\Lambda)=\Z$.
The previous sequence therefore reduces to
\begin{align}
\label{eq:Rim2}
 0 \to \Z\la\alpha \times \mu_T\ra &\to H_2(N_\Sigma\setminus\nu(T);\Lambda) \to H_2(N_{\Sigma(\alpha,J)};\Lambda)  \to \Z\la\alpha\ra \oplus \Z\la\mu_T\ra \nonumber \\
  &\to H_1(N_\Sigma\setminus\nu(T);\Lambda) \oplus \Z \to 0.
\end{align}
When $J=U$ is the unknot, $H_1(S^1 \times E_J;\Lambda)=H_1(\alpha)=\Z\langle \alpha \rangle$,
and therefore $\Z\la\mu_T\ra$ surjects onto $H_1(N_\Sigma\setminus\nu(T);\Lambda) $.
In particular, as an abelian group,~$H_1(N_\Sigma\setminus\nu(T);\Lambda)$ is cyclic.


\begin{claim}
\label{claim:ItisZ}
The map~$\Z\la \mu_T \ra \stackrel{\iota}{\twoheadrightarrow} H_1(N_\Sigma\setminus\nu(T);\Lambda)$ is an isomorphism.
\end{claim}
\begin{proof}
Since we already noted that, as an abelian group,~$H_1(N_\Sigma\setminus\nu(T);\Lambda)$ is cyclic, it suffices to show that~$H_1(N_\Sigma\setminus\nu(T);\Lambda) \otimes_\Z \Q \neq 0$.
Equivalently, we must establish that \[H:=H_1(N_\Sigma\setminus\nu(T);\Q[t^{\pm 1}]) \neq~0.\]
By way of contradiction, assume that~$H=0$.
Use once again~$\varepsilon \colon \Lambda \to \Z$ to denote the augmentation map and write~$\Z_\varepsilon$ and~$\Q_\varepsilon$ for the resulting~$\Lambda$-module structures.
Since~$\operatorname{Tor}_2^\Lambda(\Z_\varepsilon,\Q_\varepsilon)=0$, the universal coefficient spectral sequence applied to~$H_1(N_\Sigma\setminus\nu(T);\Q)=H_1(N_\Sigma\setminus\nu(T);\Q[t^{\pm 1}] \otimes_{\Lambda}\Q_\varepsilon)$ produces the short exact sequence
$$0 \to H \otimes_{\Lambda} \Q_\varepsilon \to H_1(N_\Sigma\setminus\nu(T);\Q) \to \operatorname{Tor}^\Lambda_1(\Z_{\varepsilon},\Q_{\varepsilon}) \to 0.$$
We assumed that~$H=0$, so~$H \otimes_{\Lambda} \Q_\varepsilon=0$, and thus~$H_1(N_\Sigma\setminus\nu(T);\Q) \cong  \operatorname{Tor}^\Lambda_1(\Z,\Q)$.
Using Lemma~\ref{lem:RimTorus},~$H_1(N_\Sigma\setminus\nu(T);\Q)=\Q^2$.
On the other hand, we have~$ \operatorname{Tor}^\Lambda_1(\Z_{\varepsilon},\Q_{\varepsilon})=H_1(S^1;\Q)=~\Q$.
This is a contradiction and the claim is established.
\end{proof}

Using Claim~\ref{claim:ItisZ}, the penultimate map in~\eqref{eq:Rim2} is a surjection~$\Z^2 \to \Z^2$ and therefore an isomorphism.
The first short exact sequence in the statement of the lemma now follows from the one displayed in~\eqref{eq:Rim2}.
The exactness of the second sequence follows from exactness of the first by taking $J=U$ and recalling that $\Sigma_n(\alpha,U) = \Sigma$ by Remark~\ref{remark:U-rim-surgery-does-nothing}.
\end{proof}

\begin{theorem}
\label{thm:RimSurgeryClosed}
Let~$\Sigma \subseteq X$ be a closed $\Z$-surface, let~$\alpha \subseteq \Sigma$ be a simple closed curve and let~$J$ be a knot.
The surfaces~$\Sigma$ and~$\Sigma(\alpha,J)$ are topologically ambiently isotopic.
\end{theorem}

\begin{proof}
After an ambient isotopy, we may assume that the surfaces~$\Sigma_0 := \Sigma(\alpha,J)$ and~$\Sigma_1 := \Sigma$ coincide on a disc~$D^2 \subseteq \Sigma \cap \Sigma(\alpha,J)$ that does not intersect~$\alpha \subseteq \Sigma$.
Assume that the normal bundles also coincide over this~$D^2$.
Consider the preimage~$\mathring{D}^2 \times \R^2 \subseteq \nu \Sigma_i$.
This is homeomorphic to an open 4-ball~$\mathring{D}^4$.
Remove this~$(\mathring{D}^4,\mathring{D}^2)$ from~$(X,\Sigma)$ and~$(X,\Sigma(\alpha,J))$ to obtain properly embedded surfaces~$\widetilde{\Sigma} \subseteq N$ and~$\widetilde{\Sigma}(\alpha,J) \subseteq N$ with the unknot as a common boundary.
By construction,~$\widetilde{\Sigma}(\alpha,J)$ is obtained by rim surgery on~$\widetilde{\Sigma}$.
Apply Theorem~\ref{thm:RimSurgery} to obtain a rel.\ boundary homeomorphism of pairs~$\Phi \colon (N,\widetilde{\Sigma}(\alpha,J)) \to  (N,\widetilde{\Sigma})$.
Recall that on $H_2(-;\Lambda)$, we have $\Phi_*=\Psi_*$, where $\Psi$ is the degree one map described above Lemma~\ref{lem:KunnethSS}.
Construct a homeomorphism of pairs~$(X,\Sigma(\alpha,J)) \to (X,\Sigma)$  by gluing~$\Phi$ with the identity homeomorphism~$(D^4,D^2) \to (D^4,D^2)$.

We refine this argument to obtain the required ambient isotopy.
Use~$F$ to denote the isometry induced by these homeomorphisms on~$H_2(-;\Lambda)$ of the surface exteriors.
In Lemma~\ref{lem:KunnethSS}, we argued that this isometry fits into the following commutative diagram with exact rows (here we also used Lemma~\ref{lem:RimLemma2} to simplify the Mayer-Vietoris sequences):
\begin{equation}
\label{eq:DiagramOfIdenties}
\xymatrix@R0.5cm{
0\ar[r] &\Z\langle \alpha \times \mu_T \rangle \ar[r]\ar[d]_=^{\id}& H_2(N_{\widetilde{\Sigma}} \setminus \nu(T);\Lambda) \ar[r]\ar[d]^{\id}_=& H_2(N_{\widetilde{\Sigma}(\alpha,J)};\Lambda)  \ar[d]^{F=\Psi_*} \ar[r] &0 \\
0\ar[r] & \Z\langle \alpha \times \mu_T \rangle \ar[r]& H_2(N_{\widetilde{\Sigma}} \setminus \nu(T);\Lambda) \ar[r]& H_2(N_{\widetilde{\Sigma}};\Lambda)  \ar[r] &0.
}
\end{equation}
Recall that~$\Sigma_0 := \Sigma(\alpha,J)$ and~$\Sigma_1 := \Sigma$, so~$X_{\Sigma_0} \cong N_{\widetilde{\Sigma}(\alpha,J)}$ and~$X_{\Sigma_1} \cong N_{\widetilde{\Sigma}}$.
In Lemma~\ref{lem:InducedIsometryOfH2(X)},~$F_\Z$ was defined by the following commutative diagram with exact rows
\begin{equation}
\label{eq:InducedIsomOfXForIsotopy}
\xymatrix@R0.5cm{
0 \ar[r] & \Z^{2g} \ar[r]\ar[d]^{(F \otimes_\Lambda \id_{\Z})_|} & H_2(X_{\Sigma_0}) \ar[r]_{p_0}\ar[d]^{F \otimes_\Lambda \id_{\Z}}  & H_2(X) \ar[r] \ar@{->}[d]^{F_\Z} &0  \\
0 \ar[r] & \Z^{2g} \ar[r] & H_2(X_{\Sigma_1}) \ar[r]^{p_1} & H_2(X) \ar[r]&0.
}
\end{equation}
By~\eqref{eq:InducedIsomOfXForIsotopy}, every element~$x$ of~$H_2(X)$ can be represented by a class in~$H_2(X_{\Sigma_0})$, and by~\eqref{eq:DiagramOfIdenties} every class here can be represented by a surface~$S$ in~$N_{\widetilde{\Sigma}(\alpha,J)} \setminus \nu(T) \subseteq N_{\widetilde{\Sigma}(\alpha,J)} \cong X_{\Sigma_0}$.  Since~$F$ is induced
by the degree one map~$\Psi$, mapping our surface to the same surface~$S$ in~$N_{\widetilde{\Sigma}} \setminus \nu(T) \subseteq N_{\widetilde{\Sigma}} \cong X_{\Sigma_1}$ yields~$(F\otimes_{\La} \id_{\Z})([S]) \in H_2(X_{\Sigma_1})$.
The inclusion induced map~$p_1 \colon H_2(X_{\Sigma_1}) \to H_2(X)$ then sends~$(F\otimes_{\La} \id_{\Z})([S])$ to~$x$.
Therefore~$F_\Z \colon H_2(X) \to H_2(X)$ is indeed the identity, so the second item of Theorem~\ref{thm:Unknotting4Manifold} implies that~$\Sigma_0$ and~$\Sigma_1$ are topologically ambiently isotopic.
\end{proof}

\appendix
\section{Stable isotopy for surfaces in a topological~$4$-manifold with boundary}
\label{sec:BS}


The next result is an extension of \cite[Theorem~5]{BaykurSunukjian} to the topological case, and allowing nonempty boundary, but restricting the ambient 4-manifolds somewhat.
Since we need the given extension, we provide details of the proof.
The main ideas are due to~\cite{BaykurSunukjian}.  We fill in some details in their argument for constructing a map to~$S^1$ in the course of the proof. The case with nonempty boundary was also stated as~\cite[Proposition~2.13]{JuhaszZemke} in the smooth category.

\begin{theorem}\label{thm:BS}
Let~$N$ be an oriented, connected, simply connected, compact topological 4-manifold with boundary~$S^3$.
Let~$\Sigma$ be a compact, connected, orientable surface with one boundary component.
Let~$\Sigma_0, \Sigma_1 \subseteq N$ be two properly embedded, locally flat, oriented surfaces homeomorphic to~$\Sigma$, and suppose that~$\partial \Sigma_0 = \partial \Sigma_1$.
Assume that for each~$i$,~$\pi_1(N \sm \Sigma_i) \cong \Z$ is infinite cyclic generated by an oriented meridian to~$\Sigma_i$.
Then some finite number of trivial 1-handle stabilisations results in ambiently isotopic surfaces~$\Sigma_0'$ and~$\Sigma_1'$ with~$\partial \Sigma_i = \partial \Sigma_i'$ for~$i=0,1$.
\end{theorem}

We begin with a lemma on framings, some of which is a recollection from Section~\ref{sec:Knotted}.
Note that any locally flat embedded surface in a 4-manifold admits a normal bundle with linear structure group~\cite[Theorem~9.3]{FreedmanQuinn}.

\begin{lemma}\label{lem:induced-framing}
 For~$i=0,1$ the normal bundle of~$\Sigma_i$ is trivial and for any choice of framing there is a well-defined induced~$($homotopy class of$)$ framing on the normal bundle of~$\partial \Sigma_i$.
The surface~$\Sigma_i$ is null-homologous in~$H_2(N,\partial N)$, and the induced framing on~$\partial \Sigma_i$ equals the Seifert framing.
\end{lemma}

\begin{proof}
Since the surface~$\Sigma_i$ has nonempty boundary and is oriented, its normal bundle is trivial.
Two choices of framings on~$\Sigma_i$ differ by a map~$\Sigma_i \to SO(2)$.  The difference between the two framings on the boundary~$\partial \Sigma_i$ is governed by the homotopy class of the composite
\[S^1 \xrightarrow{\cong} \partial \Sigma_i \to \Sigma_i \to SO(2) \cong S^1.\]
Since the boundary determines a commutator in~$\pi_1(\Sigma_i)$, the displayed map is null-homotopic. It follows that the two framings induced on ~$\partial \Sigma_i$ agree up to homotopy. This proves that the induced framing is well-defined.
The second sentence was already shown in Lemmas~\ref{lem:Nullhomologous} and~\ref{lem:trivial-normal-bundle}.
\end{proof}

\begin{proof}[Proof of Theorem~\ref{thm:BS}]
Push the boundary of~$\Sigma_1$ off the boundary of~$\Sigma_0$, using the Seifert framing, to arrange that~$\partial \Sigma_0$ and~$\partial \Sigma_1$ in~$\partial N \cong S^3$ are parallel circles.
We may assume that there is a collar~$\partial N \times I$ with~$\partial N \times \{0\} = \partial N$, and an embedding
 ~$g \colon S^1 \times I \times I \hookrightarrow \partial N \times I$ corresponding to the Seifert framing push-off, with
  \begin{align*}
 g(S^1 \times I \times \{i\}) &= \Sigma_i \cap (\partial N \times I), \,\,\, i=0,1 \text{, and}\\
 g(S^1 \times \{t\} \times I) &\subseteq \partial N \times \{t\} \text{ for all } t \in I.
 \end{align*}
Let~$A \subseteq \partial N \cong S^3$ be the annulus
  \[A:= g(S^1 \times \{0\} \times I)\]
  connecting~$\partial \Sigma_0$ and~$\partial \Sigma_1$, arising from the trace of the push.
Note that~$A$ induces the Seifert framing on~$\partial \Sigma_i$, for~$i=0,1$, and that by Lemma~\ref{lem:induced-framing} this equals the framing induced by some choice of framing of the normal bundle of~$\Sigma_i$.

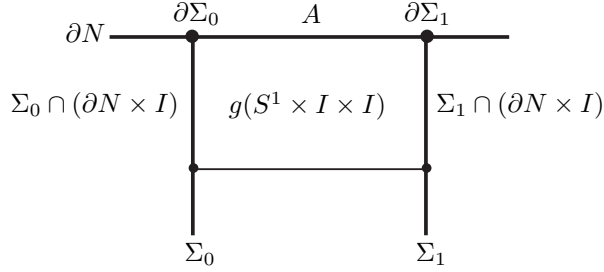
\begin{figure}[!htb]
\centering
\labellist
    \small
    \pinlabel {$\partial \Sigma_0$} at 75 140
        \pinlabel {$\partial\Sigma_1$} at 200 140
    \pinlabel {$\Sigma_0$} at 77 9
        \pinlabel {$\Sigma_1$} at 202 9
    \pinlabel {$\Sigma_0 \cap (\partial N \times I)$} at 20 90
        \pinlabel {$\Sigma_1 \cap (\partial N \times I)$} at 250 90
    \pinlabel {$\partial N$} at 15 130
    \pinlabel {$g(S^1 \times I \times I)$} at 135 90
    \pinlabel {$A$} at 137 140
   \endlabellist
\includegraphics[scale=0.7]{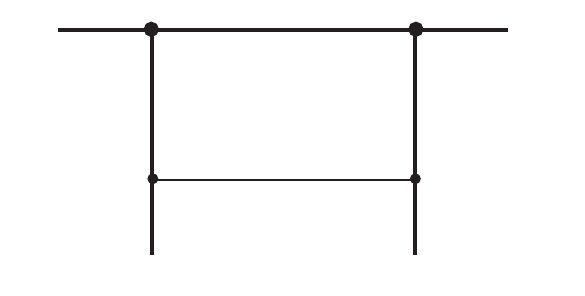}
\caption{A schematic diagram indicating the relation of $g(S^1 \times I \times I) \subseteq \partial N \times I$ to $\partial \Sigma_0$, $\partial \Sigma_1$, and the annulus $A \subseteq S^3$ that joins them.}
\label{fig:notation-schematic}
\end{figure}

By an isotopy, arrange further that~$\Sigma_0$ and~$\Sigma_1$ intersect transversely in their interiors~\cite[Theorem~9.5]{FreedmanQuinn}, and since~$\Sigma$ is compact we may also assume there are finitely many intersection points.

Cap off each~$\partial \Sigma_i$ with a Seifert surface, and by a small isotopy of the capped-off~$\Sigma_1$, arrange that the capped-off surfaces are disjoint in~$\partial N \times I$.  By Lemma~\ref{lem:induced-framing} (and since~$H_2(\partial N) = H_2(S^3)=0$), the capped-off surfaces are null-homologous in~$H_2(N)$, and so intersect algebraically zero times.  By tubing~$\Sigma_0$ to itself, that is by 1-handle stabilisations, arrange that~$\Sigma_0$ and~$\Sigma_1$ are disjoint. To achieve this, pair up points with opposite signs in $\Sigma_0 \pitchfork \Sigma_1$ and for each pair $\{p,q\}$ choose a path $\gamma$ in $\Sigma_1$ connecting the two intersection points and away from the other intersection points.  We can choose these paths to be mutually disjoint, but this is not obligatory.  For each pair of points $\{p,q\}$, remove two open discs from $\Sigma_0$, one for each pair of points, and add a tube, coming from the normal circle bundle of $\Sigma_1$ restricted to~$\gamma$, as shown in Figure~\ref{fig:1-handle-stab}. In case different paths on $\Sigma_1$ intersect, vary the radii of the tubes to keep them disjoint. This stabilises $\Sigma_0$ to a surface disjoint from $\Sigma_1$.

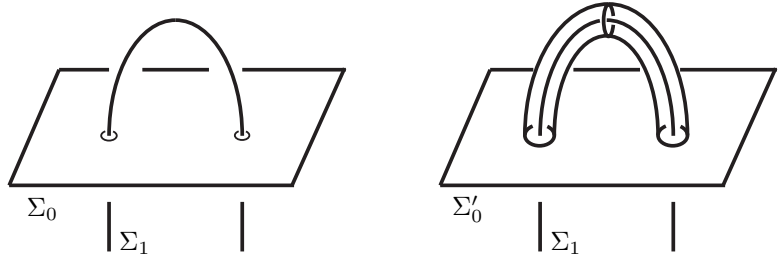
\begin{figure}[!htb]
\centering
\labellist
    \small
    \pinlabel {$\Sigma_0$} at 90 39
        \pinlabel {$\Sigma_1$} at 140 20
    \pinlabel {$\Sigma_0'$} at 320 39
        \pinlabel {$\Sigma_1$} at 373 20
    \endlabellist
\includegraphics[scale=0.7]{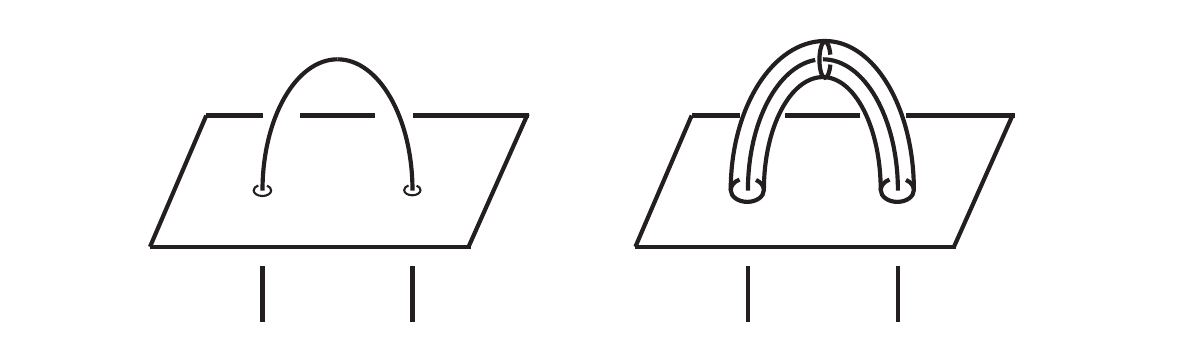}
\caption{Stabilising $\Sigma_0$ to $\Sigma_0'$ by adding a 1-handle in a neighbourhood of an arc~$\gamma$ in~$\Sigma_1$.}
\label{fig:1-handle-stab}
\end{figure}

  Perform the same number of trivial 1-handle stabilisations to~$\Sigma_1$, so that the surfaces are still abstractly homeomorphic.  We will abuse notation and still refer to the resulting surfaces as~$\Sigma_0$ and~$\Sigma_1$.
Let \[N_{\Sigma_0,\Sigma_1} := N \sm (\nu \Sigma_0 \cup \nu\Sigma_1).\]
We will construct an oriented 3-dimensional relative Seifert manifold~$Y \subseteq N_{\Sigma_0,\Sigma_1}$, that is a rel.\ boundary cobordism between~$\Sigma_0$ and~$\Sigma_1$, embedded in~$N$ and with~$\partial Y = S$, where
\[S := \Sigma_0 \cup A \cup -\Sigma_1.\]

\begin{claim*}
  There is a locally flat, embedded, compact, orientable~$3$-manifold~$Y$ with~$\partial Y=S$ and corners at~$\Sigma_i \cap A$, as described in the preceding paragraph.
\end{claim*}

\begin{proof}
The strategy to construct~$Y$ is as follows: define a suitable map~$\partial N_{\Sigma_0,\Sigma_1} \to S^1$, extend it to a map~$N_{\Sigma_0,\Sigma_1} \to S^1$ while controlling the restriction to~$\partial N_{\Sigma_0,\Sigma_1}$, and then take~$Y$ to be the inverse image of a transverse regular point in~$S^1$.

The first step is to construct a map~$\alpha \colon N_{\Sigma_0,\Sigma_1} \to S^1$.
Recall that~$N_{\Sigma_i} := N \sm \nu \Sigma_i$.
We can express~$N = N_{\Sigma_0} \cup N_{\Sigma_1},$ with~$N_{\Sigma_0,\Sigma_1} = N_{\Sigma_0} \cap N_{\Sigma_1}.$
A portion of the resulting Mayer-Vietoris sequence (with~$\Z$ coefficients) is
\[H_2(N) \xrightarrow{\delta} H_1(N_{\Sigma_0,\Sigma_1}) \to H_1(N_{\Sigma_0}) \oplus H_1(N_{\Sigma_1}) \to H_1(N).\]
By hypothesis~$H_1(N) = 0$ and~$H_1(N_{\Sigma_i})\cong \Z$, so that we obtain a short exact sequence
\[0 \to \im \delta \to H_1(N_{\Sigma_0,\Sigma_1}) \to \Z \oplus \Z \to 0.\]
Since~$\Z \oplus \Z$ is free abelian, this splits and we have that
\[H_1(N_{\Sigma_0,\Sigma_1}) \cong \Z \oplus \Z \oplus \im \delta.\]
By the hypotheses of Theorem~\ref{thm:BS}, the first two summands are generated by meridians~$\mu_{\Sigma_0}$ and~$\mu_{\Sigma_1}$ to the surfaces~$\Sigma_0$ and~$\Sigma_1$ respectively.
Consider the dual element
\[\alpha := \mu_{\Sigma_0}^* - \mu_{\Sigma_1}^* \in \Hom(H_1(N_{\Sigma_0,\Sigma_1}),\Z) \cong H^1(N_{\Sigma_0,\Sigma_1}) \cong [N_{\Sigma_0,\Sigma_1},S^1],\]
 sending the~$\im \delta$ summand identically to~$0$.
 We will also write~$\alpha \colon N_{\Sigma_0,\Sigma_1} \to S^1$ for a corresponding representing map.
 We use that since~$N_{\Sigma_0,\Sigma_1}$ is a topological 4-manifold, it is homotopy equivalent to a CW-complex~\cite[Theorem~4.5]{FriedlNagelOrsonPowell}, and therefore we may identify~$H^1(N_{\Sigma_0,\Sigma_1}) \cong [N_{\Sigma_0,\Sigma_1},S^1]$.

Now we consider the restriction to~$\partial N_{\Sigma_0,\Sigma_1}$ under the inclusion-induced map
\[\iota^* \colon H^1(N_{\Sigma_0,\Sigma_1}) \to H^1(\partial N_{\Sigma_0,\Sigma_1}) \cong [\partial N_{\Sigma_0,\Sigma_1},S^1],\]
where~$\iota \colon \partial N_{\Sigma_0,\Sigma_1} \to N_{\Sigma_0,\Sigma_1}$ is the inclusion map.
We have a decomposition \[\partial N_{\Sigma_0,\Sigma_1} = \Sigma_0 \times S^1 \cup \partial_0 N_{\Sigma_0,\Sigma_1} \cup \Sigma_1 \times S^1.\]
An elementary Mayer-Vietoris argument yields that
\[H^1(\partial N_{\Sigma_0,\Sigma_1}) \cong \Z \oplus \Z \oplus \Z^{2g} \oplus \Z^{2g}.\]
The first two summands are generated by duals to the meridians,~$\mu_{\Sigma_0}^*$ and ~$\mu_{\Sigma_1}^*$.
By our work with framings above, for~$i=0,1$ there is a framing of the normal bundle of~$\Sigma_i$, determining an identification of the tubular neighbourhood~$\ol{\nu} \Sigma_i$ with~$\Sigma_i \times D^2$, that agrees on~$\partial \Sigma_i$ with a corresponding identification induced from the Seifert framing on~$\partial \Sigma_i$.
This determines an identification
\[\partial \ol{\nu} \Sigma_i \sm \nu \partial \Sigma_i \cong \Sigma_i \times S^1.\]
The two~$\Z^{2g}$ summands of $H^1(\partial N_{\Sigma_0,\Sigma_1})$ are generated by dual classes to curves of the form~$\gamma_k \times \{-1\}$ where~$\gamma_k$ is simple closed curve forming part of a symplectic basis for~$H_1(\Sigma_i)$, for some~$i$.  Here we use the chosen framing of the normal bundle of~$\Sigma_i$ to fix representatives for the~$\Z^{2g}$ summands.

The restriction~$\iota^*\alpha \in  H^1(\partial N_{\Sigma_0,\Sigma_1})$ is~$(1,-1,x,y)$, for some~$x,y \in \Z^{2g}$.  But by changing the choice of framing of~$\nu \Sigma_i$ along the curves~$\gamma_k$, we can arrange that~$x=y=0$.  By changing the framing of the normal bundle of~$\Sigma_i$ by some number of full twists along some basis curve~$\gamma_k$ in~$H_1(\Sigma_i)$, we change the meaning of~$\gamma_k \times \{-1\}$ in the previous paragraph.

Let us explain the operation of ``changing the framing'' in more detail. Such a change is occasioned by the action of~$[\cup_k \gamma_k,SO(2)]$ on the set of framings of the normal bundle, to alter the given framing on the union of the curves~$\{\gamma_k\}$  by some number of full twists for each curve. Any such alteration automatically extends over the 2-skeleton since the attaching map of the 2-cell of~$\Sigma_i$ is a commutator in the~$\gamma_k$ times~$\partial \Sigma_i$. Any two choices of extension over the 2-cells are homotopic, since~$\pi_2(SO(2))=0$.  Therefore we have a well-defined notion of altering the framing along the curves~$\gamma_k$.
This changes the entry of~$(x,y)$ corresponding to~$\gamma_k$, since the map~$\alpha \colon N_{\Sigma_0,\Sigma_1} \to S^1$ now sends ~$\gamma_k \times \{-1\}$ to a curve representing a different element of~$H_1(S^1)$.
By Lemma~\ref{lem:induced-framing}, changing the choice of framing on a basis element for~$H_1(\Sigma_i)$ does not change the induced framing on the boundary.

Now we define a map~$f \colon \partial N_{\Sigma_0,\Sigma_1} \to S^1$.
 On~$\Sigma_i\times S^1 \subseteq \partial N_{\Sigma_0,\Sigma_1}$, define the map  to~$S^1$ by the projection~$f|_{\Sigma_i \times S^1} \colon \Sigma_i \times S^1 \to S^1$ onto the second factor.
On the remainder of~$\partial N_{\Sigma_0,\Sigma_1}$, namely
\[\partial_0 N_{\Sigma_0,\Sigma_1} := \partial N \sm (\nu \Sigma_0 \cup \nu \Sigma_1),\]
or in other words the link exterior~$S^3 \sm (\nu \partial \Sigma_0 \cup \nu \partial \Sigma_1)$,
 define a Pontryagin-Thom style collapse map $f|_{\partial_0 N_{\Sigma_0,\Sigma_1}} \colon \partial_0 N_{\Sigma_0,\Sigma_1} \to S^1$ by choosing a tubular neighbourhood~$A \times [-1,1] \subseteq \partial_0 N_{\Sigma_0,\Sigma_1}$ and sending~$(a,x) \mapsto e^{\pi i x} \in S^1$ for~$a \in A$ and~$x \in [-1,1]$, then sending~$\partial_0 N_{\Sigma_0,\Sigma_1} \sm (A \times [-1,1])$ to~$-1 \in S^1$.
 Then~$f|_{\partial_0 N_{\Sigma_0,\Sigma_1}}^{-1}(\{1\}) = A \cap \partial_0 N_{\Sigma_0,\Sigma_1}$.
Note that we may assume that our maps to~$S^1$ agree on the torus overlaps~$\partial \Sigma_i \times S^1$, since the framings on~$\partial \Sigma_i$ all agree up to homotopy.  This completes the construction of a map~$f \colon \partial N_{\Sigma_0,\Sigma_1} \to S^1$.
 Note that~$f \colon \partial N_{\Sigma_0,\Sigma_1} \to S^1$ corresponds to the element \[(1,-1,0,0) \in \Z \oplus \Z \oplus \Z^{2g} \oplus \Z^{2g}.\]
Therefore the cohomology classes~$\iota^*\alpha$ and~$f$ agree in ~$H^1(\partial N_{\Sigma_0,\Sigma_1}) \cong [\partial N_{\Sigma_0,\Sigma_1},S^1]$, and so by a homotopy of~$\alpha$ supported in a collar of~$\partial N_{\Sigma_0,\Sigma_1}$ we obtain a map~$F \colon N_{\Sigma_0,\Sigma_1} \to~S^1$ with~$F|_{\partial N_{\Sigma_0,\Sigma_1}} = f \colon \partial N_{\Sigma_0,\Sigma_1} \to S^1$.



The inverse image under $F$ of a transverse regular point in~$S^1$ yields a 3-dimensional relative Seifert manifold, locally flatly embedded in~$N_{\Sigma_0,\Sigma_1} = N \sm (\nu \Sigma_0 \cup \nu\Sigma_1)$. See \cite[Section~10.2]{FriedlNagelOrsonPowell} for information on map transversality in the topological category.
Add collars~$\Sigma_0 \times I$ and~$\Sigma_1 \times I$ in~$\nu \Sigma_0$ and~$\nu \Sigma_1$ respectively, to obtain the 3-manifold~$Y$ that we seek.  This completes the proof of the claim.
\end{proof}

Morse theory on~$Y$ gives rise to a Heegaard decomposition relative to the annulus~$A$. The Heegaard surface can be obtained from both~$\Sigma_0$ and~$\Sigma_1$ by 1-handle stabilisations and ambient isotopy.

Since by our assumptions~$\pi_1(N_{\Sigma_i})$ is generated by a meridian to~$\Sigma_i$, Boyle's~\cite{Boyle} proof shows that every 1-handle stabilisation is a trivial stabilisation.  He applied \cite[Theorem~4]{Hudson-72} of Hudson for the statement that~$D^1$ cores of handle additions that are homotopic rel.\ endpoints are in fact smoothly ambiently isotopic fixing the endpoints.
To apply Boyle's work to topologically embedded surfaces in a compact 4-manifold~$N$, remove a point from~$N$ and smooth~$N \sm \{\pt\}$ in such a way that~$\Sigma$ is smoothly embedded.  Then Boyle's application of Hudson's result yields a smooth ambient isotopy, which gives rise to a topological ambient isotopy once the point is added back to~$N$.
We note that Boyle works in~$S^4$, but his proof applies to any oriented ambient~$4$-manifold.

We therefore have that after finitely many trivial stabilisations,~$\Sigma_0$ and~$\Sigma_1$ are ambiently isotopic in~$N$ relative to the constant isotopy on the boundary, as desired.
\end{proof}

\bibliographystyle{alpha}
\bibliography{bibliopi1Z}
\end{document}